\documentclass[12pt,leqno]{article}
\usepackage{amsmath}
\usepackage{amssymb}
\usepackage{amsxtra}
\usepackage{amscd}
\usepackage{amsthm}
\usepackage{tikz-cd}
\usepackage[all]{xy}
\usepackage{enumitem}
\usepackage[T1]{fontenc}
\usepackage[utopia]{mathdesign}
\usepackage{hyperref}

\theoremstyle{plain}
\newtheorem{thm}[subsection]{Theorem}

\newtheorem{sbthm}[subsubsection]{Theorem}
\newtheorem{sbprop}[subsubsection]{Proposition}
\newtheorem{sbcor}[subsubsection]{Corollary}
\newtheorem{sblem}[subsubsection]{Lemma}

\theoremstyle{definition}

\newtheorem{para}[subsection]{}

\newtheorem{sbeg}[subsubsection]{Example}
\newtheorem{sbrem}[subsubsection]{Remark}

\newtheorem{sbpara}[subsubsection]{}
\newtheorem*{example}{Example}
\newtheorem*{ack}{Acknowledgments}

\newenvironment{pf}{\proof[\proofname]}{\endproof}

\DeclareMathOperator{\PGL}{PGL}
\DeclareMathOperator{\GL}{GL}
\DeclareMathOperator{\SL}{SL}
\DeclareMathOperator{\PO}{PO}
\DeclareMathOperator{\PU}{PU}
\DeclareMathOperator{\SO}{SO}
\DeclareMathOperator{\SU}{SU}
\DeclareMathOperator{\diag}{diag}
\DeclareMathOperator{\Iw}{Iw}
\DeclareMathOperator{\BM}{BM}
\DeclareMathOperator{\KP}{KP}

\newcommand{\mb}{\mathbb}

\setlist[1]{itemsep=0.5ex, topsep=1ex}
\setlist[2]{nosep}

\begin{document}

\title{Compactifications of $S$-arithmetic quotients for the projective general linear group}

\author{Takako Fukaya, Kazuya Kato, Romyar Sharifi}
%\subjclass{Primary 14M25; Secondary 14F20}

\maketitle

\newcommand\Cal{\mathcal}
\newcommand\define{\newcommand}

\define\gp{\mathrm{gp}}%
\define\fs{\mathrm{fs}}%
\define\an{\mathrm{an}}%
\define\mult{\mathrm{mult}}%
\define\Ker{\mathrm{Ker}\,}%
\define\Coker{\mathrm{Coker}\,}%
\define\Hom{\mathrm{Hom}\,}%
\define\Ext{\mathrm{Ext}\,}%
\define\rank{\mathrm{rank}\,}%
\define\gr{\mathrm{gr}}%
\define\cHom{\Cal{Hom}}
\define\cExt{\Cal Ext\,}%

\define\cB{\Cal B}
\define\cC{\Cal C}
\define\cD{\Cal D}
\define\cO{\Cal O}
\define\cS{\Cal S}
\define\cT{\Cal T}
\define\cM{\Cal M}
\define\cG{\Cal G}
\define\cH{\Cal H}
\define\cR{\Cal R}
\define\cE{\Cal E}
\define\cF{\Cal F}
\define\cN{\Cal N}
\define\fF{\frak F}
\define\Dc{\check{D}}
\define\Ec{\check{E}}

\newcommand{\N}{{\mathbb{N}}}
\newcommand{\Q}{{\mathbb{Q}}}
\newcommand{\Z}{{\mathbb{Z}}}
\newcommand{\R}{{\mathbb{R}}}
\newcommand{\C}{{\mathbb{C}}}
\newcommand{\bN}{{\mathbb{N}}}
\newcommand{\bQ}{{\mathbb{Q}}}
\newcommand{\bF}{{\mathbb{F}}}
\newcommand{\bZ}{{\mathbb{Z}}}
\newcommand{\bP}{{\mathbb{P}}}
\newcommand{\bR}{{\mathbb{R}}}
\newcommand{\bC}{{\mathbb{C}}}
\newcommand{\bbQ}{{\bar \mathbb{Q}}}
\newcommand{\ol}[1]{\overline{#1}}
\newcommand{\too}{\longrightarrow}
\newcommand{\respect}{\rightsquigarrow}
\newcommand{\compatible}{\leftrightsquigarrow}
\newcommand{\upc}[1]{\overset {\lower 0.3ex \hbox{${\;}_{\circ}$}}{#1}}
\newcommand{\Gmlog}{\bG_{m, \log}}%{\mathbb{G}_{m,\log}}
\newcommand{\Gm}{\bG_m}%{\mathbb{G}_m}
\newcommand{\ep}{\varepsilon}
\newcommand{\Spec}{\operatorname{Spec}}
\newcommand{\val}{{\mathrm{val}}} 
\newcommand{\tor}{\mathrm{tor}}
\newcommand{\n}{\operatorname{naive}}
\newcommand{\bs}{\backslash}
\newcommand{\Gal}{\operatorname{{Gal}}}
\newcommand{\gal}{{\rm {Gal}}({\bar \Q}/{\Q})}
\newcommand{\galp}{{\rm {Gal}}({\bar \Q}_p/{\Q}_p)}
\newcommand{\gall}{{\rm{Gal}}({\bar \Q}_\ell/\Q_\ell)}
\newcommand{\wep}{W({\bar \Q}_p/\Q_p)}
\newcommand{\wel}{W({\bar \Q}_\ell/\Q_\ell)}
\newcommand{\Ad}{{\rm{Ad}}}
\newcommand{\BS}{{\rm {BS}}}
\newcommand{\even}{\operatorname{even}}
\newcommand{\End}{{\rm {End}}}
\newcommand{\odd}{\operatorname{odd}}
\newcommand{\np}{\text{non-$p$}}
\newcommand{\g}{{\gamma}}
\newcommand{\G}{{\Gamma}}
\newcommand{\Lam}{{\Lambda}}
\newcommand{\La}{{\Lambda}}
\newcommand{\lam}{{\lambda}}
\newcommand{\la}{{\lambda}}
\newcommand{\uL}{{{\hat {L}}^{\rm {ur}}}}
\newcommand{\uQp}{{{\hat \Q}_p}^{\text{ur}}}
\newcommand{\sel}{\operatorname{Sel}}
\newcommand{\dt}{{\rm{Det}}}
\newcommand{\Sig}{\Sigma}
\newcommand{\fil}{{\rm{fil}}}
\newcommand{\spl}{{\rm{spl}}}
\newcommand{\st}{{\rm{st}}}
\newcommand{\Isom}{{\rm {Isom}}}
\newcommand{\Mor}{{\rm {Mor}}}
\newcommand{\bg}{\bar{g}}
\newcommand{\id}{{\rm {id}}}
\newcommand{\cone}{{\rm {cone}}}
\newcommand{\al}{a}
\newcommand{\ChL}{{\cal{C}}(\La)}
\newcommand{\Image}{{\rm {Image}}}
\newcommand{\toric}{{\operatorname{toric}}}
\newcommand{\torus}{{\operatorname{torus}}}
\newcommand{\Aut}{{\rm {Aut}}}
\newcommand{\Qp}{{\mathbb{Q}}_p}
\newcommand{\barQp}{{\mathbb{Q}}_p}
\newcommand{\Qpur}{{\mathbb{Q}}_p^{\rm {ur}}}
\newcommand{\Zp}{{\mathbb{Z}}_p}
\newcommand{\Zl}{{\mathbb{Z}}_l}
\newcommand{\Ql}{{\mathbb{Q}}_l}
\newcommand{\Qlur}{{\mathbb{Q}}_l^{\rm {ur}}}
\newcommand{\F}{{\mathbb{F}}}
\newcommand{\eps}{{\epsilon}}
\newcommand{\epsLa}{{\epsilon}_{\La}}
\newcommand{\epsLaVxi}{{\epsilon}_{\La}(V, \xi)}
\newcommand{\epsOLaVxi}{{\epsilon}_{0,\La}(V, \xi)}
\newcommand{\Qplin}{{\mathbb{Q}}_p(\mu_{l^{\infty}})}
\newcommand{\otimesQplin}{\otimes_{\Qp}{\mathbb{Q}}_p(\mu_{l^{\infty}})}
\newcommand{\galFl}{{\rm{Gal}}({\bar {\Bbb F}}_\ell/{\Bbb F}_\ell)}
\newcommand{\gallur}{{\rm{Gal}}({\bar \Q}_\ell/\Q_\ell^{\rm {ur}})}
\newcommand{\galFF}{{\rm {Gal}}(F_{\infty}/F)}
\newcommand{\galFv}{{\rm {Gal}}(\bar{F}_v/F_v)}
\newcommand{\galF}{{\rm {Gal}}(\bar{F}/F)}
\newcommand{\epsVxi}{{\epsilon}(V, \xi)}
\newcommand{\epsOVxi}{{\epsilon}_0(V, \xi)}
\newcommand{\plim}{\lim_
{\scriptstyle 
\longleftarrow \atop \scriptstyle n}}
\newcommand{\BT}{{\cal B}{\cal T}}
\newcommand{\sig}{{\sigma}}
\newcommand{\ga}{{\gamma}}
\newcommand{\del}{{\delta}}
\newcommand{\Vss}{V^{\rm {ss}}}
\newcommand{\Bst}{B_{\rm {st}}}
\newcommand{\Dpst}{D_{\rm {pst}}}
\newcommand{\Dcrys}{D_{\rm {crys}}}
\newcommand{\DdR}{D_{\rm {dR}}}
\newcommand{\Fin}{F_{\infty}}
\newcommand{\Kla}{K_{\lambda}}
\newcommand{\Ola}{O_{\lambda}}
\newcommand{\Mla}{M_{\lambda}}
\newcommand{\Det}{{\rm{Det}}}
\newcommand{\Sym}{{\rm{Sym}}}
\newcommand{\LaSa}{{\La_{S^*}}}
\newcommand{\cX}{{\cal {X}}}
\newcommand{\MHG}{{\frak {M}}_H(G)}
\newcommand{\tauMla}{\tau(M_{\lambda})}
\newcommand{\Fvur}{{F_v^{\rm {ur}}}}
\newcommand{\Lie}{{\rm {Lie}}}
\newcommand{\cL}{{\cal {L}}}
\newcommand{\cW}{{\cal {W}}}
\newcommand{\fq}{{\frak {q}}}
\newcommand{\cont}{{\rm {cont}}}
\newcommand{\SC}{{SC}}
\newcommand{\Om}{{\Omega}}
\newcommand{\dR}{{\rm {dR}}}
\newcommand{\crys}{{\rm {crys}}}
\newcommand{\hatSig}{{\hat{\Sigma}}}
\newcommand{\rdet}{{{\rm {det}}}}
\newcommand{\ord}{{{\rm {ord}}}}
\newcommand{\BdR}{{B_{\rm {dR}}}}
\newcommand{\BdRO}{{B^0_{\rm {dR}}}}
\newcommand{\Bcrys}{{B_{\rm {crys}}}}
\newcommand{\Qw}{{\mathbb{Q}}_w}
\newcommand{\barkappa}{{\bar{\kappa}}}
\newcommand{\cP}{{\Cal {P}}}
\newcommand{\cQ}{{\Cal {Q}}}
\newcommand{\cZ}{{\Cal {Z}}}
\newcommand{\oppLa}{{\Lambda^{\circ}}}
\newcommand{\sn}{{{\rm {sn}}}}
\newcommand{\et}{\text{\'et}}

\begin{center}
Dedicated to Professor  John Coates on the occasion of his 70th birthday 
\end{center}

\begin{abstract}
Let $F$ be a global field, let $S$ be a nonempty finite set of places of $F$ which contains the archimedean places of $F$, let $d\geq 1$, and let $X =\prod_{v\in S} X_v$ where $X_v$ is the symmetric space (resp., Bruhat-Tits building) associated to $\PGL_d(F_v)$ if $v$ is archimedean (resp., non-archimedean). 
 In this paper, we construct compactifications $\Gamma\bs \bar X$ of the quotient spaces $\Gamma \bs X$ for $S$-arithmetic subgroups  $\Gamma$  of $\PGL_d(F)$.  
The constructions make delicate use of the maximal Satake compactification of $X_v$ (resp., the polyhedral compactification of $X_v$ of G\'erardin and Landvogt) for $v$ archimedean (resp., non-archimedean). 
 We also consider a variant of $\bar X$ in which we use the standard Satake compactification of $X_v$ (resp., the compactification of $X_v$ due to Werner). 
\end{abstract}

\section{Introduction}

\begin{para}\label{Intro1} Let $d\geq 1$, and let 
$X=\PGL_d(\R)/\PO_d(\R)\cong \SL_d(\R)/\SO_d(\R)$. The Borel-Serre space (resp., reductive Borel-Serre space) $\bar X$ contains $X$ as a dense open subspace \cite{BS} (resp., \cite{Z}).  If $\Gamma$ is a subgroup of $\PGL_d(\Z)$ of finite index, this gives rise to a compactification 
	$\Gamma \bs \bar X$ of $\Gamma \bs X$.

 \end{para}

\begin{para} Let $F$ be a global field, which is to say either a number field or a function field in one variable over a finite field.  For a place $v$ of $F$, let $F_v$ be the local field of $F$ at $v$.  Fix $d\geq 1$.

In this paper, we 
will consider the space $X_v$ of all homothety  classes of norms on $F_v^d$ and a certain space 
  $\bar X_{F,v}$ which contains $X_v$ as a dense open subset.  For $F=\Q$ and $v$ the real place, $X_v$ is identified with $\PGL_d(\R)/\PO_d(\R)$, and $\bar X_{F,v}$ is identified with the reductive Borel-Serre space associated to $\PGL_d(F_v)$. We have the following analogue of \ref{Intro1}.
  
  \end{para}
  
  \begin{thm}\label{thm0} Let $F$ be a function field in one variable over a finite field, let $v$ be a place of $F$, and let $O$ be the subring of $F$ consisting of all elements which are integral outside $v$. Then for any subgroup $\Gamma$ of $\PGL_d(O)$ of finite index, the quotient $\Gamma \bs \bar X_{F,v}$ is a compact Hausdorff space 
  which contains $\Gamma \bs X_v$ as a dense open subset. 
  \end{thm}
  
  \begin{para}
  Our space $\bar X_{F,v}$ is not a very new object. 
  In the case that $v$ is non-archimedean, $X_v$ is identified as a topological space with the Bruhat-Tits building of $\PGL_d(F_v)$. In this case, $\bar X_{F,v}$ is similar to the polyhedral compactification of $X_v$ of G\'erardin \cite{G} and Landvogt \cite{L}, which we denote by $\bar X_v$.  To each element of $\bar X_v$ is associated a
 parabolic subgroup 
  of $\PGL_{d,F_v}$. We define $\bar X_{F,v}$ as the subset of $\bar X_v$ consisting of all elements for which the associated parabolic subgroup is $F$-rational. 
  We endow $\bar X_{F,v}$ with a topology which is different from its topology as a subspace of $\bar X_v$. 
   
   In the case $d=2$, the boundary $\bar X_v\setminus X_v$ of $\bar X_v$ is $\mb{P}^1(F_v)$, whereas the boundary $\bar X_{F,v}\setminus X_v$ of $\bar X_{F,v}$ is $\mb{P}^1(F)$.
Unlike
  $\bar X_v$, the space $\bar X_{F,v}$ is not compact, but the arithmetic quotient as in \ref{Intro1} and \ref{thm0} is compact (see \ref{thm1}).

 \end{para}

\begin{para} In \S\ref{quotS}, we derive the following generalization of \ref{Intro1} and \ref{thm0}. 

Let $F$ be a global field. For a nonempty finite set $S$ of places of $F$, let $\bar X_{F,S}$ be the subspace of $\prod_{v\in S} \bar X_{F,v}$ consisting of all elements $(x_v)_{v\in S}$ such that the $F$-parabolic subgroup associated to $x_v$ is independent of $v$.  Let $X_S$ denote the subspace $\prod_{v \in S} X_v$ of $\bar X_{F,S}$.

Let $S_1$ be a nonempty finite set of places of $F$ containing all archimedean places of $F$, let $S_2$ be a finite set of places of $F$ which is disjoint from $S_1$, and let $S=S_1\cup S_2$. 
Let $O_S$ be the subring of $F$ consisting of all elements which are integral outside $S$. 

Our main result is the following theorem (see Theorem \ref{main}).

\end{para}

\begin{thm}\label{thm1} Let $\Gamma$ be a subgroup of $\PGL_d(O_S)$ of finite index. Then the quotient $\Gamma \bs (\bar X_{F, S_1} \times X_{S_2})$ is a compact Hausdorff space which contains $\Gamma \bs X_S$ as a dense open subset.

\end{thm}

\begin{para} If $F$ is a number field and $S_1$ coincides with the set of archimedean places of $F$, then the space $\bar X_{F,S_1}$ is the maximal Satake space of the Weil restriction of $\PGL_{d,F}$ from $F$ to $\Q$.  In this case, the theorem is known for $S = S_1$ through the work of Satake \cite{Sa2} and in general through the work of Ji, Murty, Saper, and Scherk \cite[4.4]{JMSS}.  
\end{para}

\begin{para}

We also consider a variant $\bar X_{F,v}^{\flat}$ of $\bar X_{F,v}$ and a variant $\bar X_{F,S}^{\flat}$ of $\bar X_{F,S}$ with continuous surjections 
$$\bar X_{F,v}\to \bar X_{F,v}^{\flat}, \quad \bar X_{F,S} \to \bar X_{F,S}^{\flat}.$$  
In the case $v$ is non-archimedean (resp., archimedean), $\bar X_{F,v}^{\flat}$ is the part with ``$F$-rational boundary'' in Werner's compactification (resp., the standard Satake compactification) $\bar X^{\flat}_v$ of $X_v$ \cite{W1, W2} (resp., \cite{Sa1}), endowed with a new topology.  We will obtain an analogue of \ref{thm1} for this variant.

To grasp the relationship with the Borel-Serre compactification \cite{BS}, we also consider a variant $\bar X_{F,v}^{\sharp}$ of $\bar X_{F,v}$ which has a continuous surjection 
$\bar X_{F,v}^{\sharp}\to \bar X_{F,v}$,   
and we show that in the case that $F=\Q$ and $v$ is the real place, $\bar X_{\Q,v}^{\sharp}$ coincides with the Borel-Serre space associated to $\PGL_{d,\Q}$ (\ref{Qcase}).   
If $v$ is non-archimedean, the space $\bar X_{F,v}^{\sharp}$ is not Hausdorff (\ref{nonH}) and does not seem useful. 
\end{para}

\begin{para} What we do in this paper is closely related to what Satake did in \cite{Sa1} and \cite{Sa2}. In \cite{Sa1}, he defined a compactification of a symmetric Riemannian space.  In \cite{Sa2}, he took the part of this compactification with ``rational boundary'' and endowed it with the Satake topology. Then he showed that the quotient of this part by an arithmetic group is compact.  We take the part $\bar X_{F,v}$ of $\bar X_v$ with ``$F$-rational boundary''  to have a compact quotient by an arithmetic group. 
So, the main results and their proofs in this paper might be  evident to the experts in the theory of Bruhat-Tits buildings, but we have not found them in the literature. 
\end{para}

\begin{para}
We intend to apply the compactification \ref{thm0} to the construction of toroidal compactifications of the moduli space of Drinfeld modules of rank $d$ in a forthcoming paper.  In Section \ref{toro}, we give a short explanation of this plan, along with two other potential applications, to asymptotic behavior of heights of motives and to modular symbols over function fields.

\end{para}

\begin{para}  We plan to generalize the results of this paper from $\PGL_d$ to general reductive groups in another forthcoming paper.  The reason why we separate the $\PGL_d$-case from the general case  is as follows.  For $\PGL_d$, we can describe the space $\bar X_{F,v}$ via norms on finite-dimensional vector spaces over $F_v$ (this method is not used for general reductive groups), and  these norms play an important role in the analytic theory of  toroidal compactifications.

\end{para}

\begin{para} In \S\ref{sec2}, we review the compactifications of Bruhat-Tits buildings in the non-archimedean setting and symmetric spaces in the archimedean setting.  In \S\ref{globspace} and \S\ref{quotS}, we discuss our compactifications. 

\end{para}

\begin{para} 
We plan to apply the results of this paper to the study of Iwasawa theory over a function field $F$. 
We dedicate this paper to John Coates, who has played a leading role in the development of Iwasawa theory.
\end{para}

\begin{ack}
The work of the first two authors was supported in part by the National Science Foundation under
	Grant No.~1001729.
	The work of the third author was partially supported by the National Science Foundation under Grant Nos.\ 
	1401122/1661568 and 1360583, and by a grant from the Simons Foundation (304824 to R.S.).
\end{ack}

\section{Spaces associated to local fields}\label{sec2}

In this section, we briefly review the compactification of the symmetric space (resp., of the Bruhat-Tits building) associated to $\PGL_d$ of an archimedean (resp., non-archimedean) local field. See the papers of Satake \cite{Sa1} and Borel-Serre \cite{BS} (resp., G\'erardin \cite{G}, Landvogt  \cite{L}, and Werner \cite{W1, W2}) for details.

Let $E$ be a local field. This means that $E$ is a locally compact topological field with a non-discrete topology.  That is, $E$ is isomorphic to $\R$, $\C$, a finite extension of $\Q_p$ for some prime number $p$, or $\F_q(\!(T)\!)$ for a finite field $\F_q$. 

Let $|\;\;| \colon E \to \R_{\geq 0}$ be the normalized absolute value.  If $E\cong \R$, this is the usual absolute value. If $E\cong \C$, this is the square of the usual absolute value. If $E$ is non-archimedean, this is the unique multiplicative map $E\to \R_{\geq 0}$ such that $|a|=\sharp(O_E/aO_E)^{-1}$ if $a$ is a nonzero element of the valuation ring $O_E$ of $E$.

Fix a positive integer $d$ and a $d$-dimensional $E$-vector space $V$. 

\subsection{Norms}

\begin{sbpara}\label{met} We recall the definitions of norms and semi-norms on $V$.   

A norm (resp., semi-norm) on $V$ is a map $\mu \colon V \to \R_{\geq 0}$ for which there exist an $E$-basis $(e_i)_{1\leq i\leq d}$ of $V$ and an element $(r_i)_{1\leq i\leq d}$ of $\R^d_{>0}$ (resp.,  $\R^d_{\ge 0}$) such that 
$$
\mu(a_1e_1+\dots+a_de_d)= 
\begin{cases}
(r_1^2|a_1|^2\dots + r_d^2|a_d|^2)^{1/2} & \text{if $E\cong \R$},\\
r_1|a_1|+\dots +r_d|a_d| & \text{if $E\cong \C$},\\
\max(r_1|a_1|, \dots, r_d|a_d|) & \text{otherwise.}
\end{cases}
$$
for all $a_1,\dots, a_d\in E$.

\end{sbpara}

\begin{sbpara}\label{met2}
We will call the norm (resp., semi-norm) $\mu$ in the above, the norm (resp., semi-norm) given by the basis $(e_i)_i$ and by $(r_i)_i$.
\end{sbpara}

\begin{sbpara} We have the following characterizations of norms and semi-norms.
\begin{enumerate}
	\item[(1)] If $E \cong \R$ (resp., $E \cong \C$), then there is a one-to-one correspondence 
	between semi-norms on $V$ and 
	symmetric bilinear (resp., Hermitian) forms $(\;\,,\;)$ on $V$ such that $(x,x) \ge 0$ for all $x \in V$.
	The semi-norm $\mu$ corresponding to $(\;\,,\;)$ is given by $\mu(x) = (x,x)^{1/2}$
	(resp., $\mu(x) = (x,x)$).  This restricts to a correspondence between norms and forms that
	are positive definite.
	\item[(2)] If $E$ is non-archimedean, then (as in \cite{GI}) a map $\mu \colon  V\to \R_{\geq 0}$
	is a norm (resp., semi-norm) if and only if $\mu$ satisfies the
	following (i)--(iii) (resp., (i) and (ii)):
	\begin{enumerate}
		\item[(i)] $\mu(ax)= |a|\mu(x)$ for all $a\in E$ and $x\in V$,
		\item[(ii)] $\mu(x+y) \leq \max(\mu(x), \mu(y))$ for all $x,y\in V$, and
		\item[(iii)] $\mu(x) >0$ if $x\in V\setminus \{0\}$.
	\end{enumerate}
\end{enumerate}
These well-known facts imply that if $\mu$ is a norm (resp., semi-norm) on $V$
and $V'$ is an $E$-subspace of $V$, then the restriction of $\mu$ to
$V'$ is a norm (resp., semi-norm) on $V'$.

\end{sbpara}

\begin{sbpara}\label{eqv} We say that two norms (resp., semi-norms) $\mu$ and $\mu'$ on $V$ are equivalent if $\mu' =c\mu$ for some $c\in \R_{>0}$.

\end{sbpara}

\begin{sbpara}\label{actn} The group $\GL_V(E)$ acts on the set of all norms (resp., semi-norms) on $V$: for $g\in \GL_V(E)$ and a norm (resp., semi-norm) $\mu$ on $V$, $g\mu$ is defined as $\mu\circ g^{-1}$. This action preserves the equivalence in \ref{eqv}.

\end{sbpara}

\begin{sbpara}\label{dual} Let $V^*$ be the dual space of $V$. Then there is a bijection between the set of norms on $V$ and the set of norms on $V^*$.  That is, for a norm $\mu$ on $V$, the corresponding norm $\mu^*$ on $V^*$ is given by 
\begin{eqnarray*}
	\mu^*(\varphi)= \sup\Bigg( \frac{|\varphi(x)|}{\mu(x)} \mid x\in V\setminus \{0\} \Bigg)
	&& \text{for } \varphi \in V^*.
\end{eqnarray*}
For a norm $\mu$ on $V$ associated to a basis $(e_i)_i$ of $V$ and $(r_i)_i \in \R^d_{>0}$,  
the norm $\mu^*$ on $V^*$ is associated to the dual basis $(e_i^*)_i$ of $V^*$ and 
$(r_i^{-1})_i$.  This proves the bijectivity.

\end{sbpara}

\begin{sbpara}\label{actdual}  For a norm $\mu$ on $V$ and for $g\in \GL_V(E)$, we have
$$(\mu \circ g)^*= \mu^* \circ (g^*)^{-1},$$
which is to say $(g\mu)^*= (g^*)^{-1}\mu^*$,
where $g^*\in \GL_{V^*}(E)$ is the transpose of $g$. 

\end{sbpara}

\subsection{Definitions of the spaces}

\begin{sbpara}\label{bs1} 
Let $X_V$ denote the set of all equivalence classes of norms on $V$ (as in \ref{eqv}). We endow $X_V$ with the quotient topology of the subspace topology on the set of all norms on $V$ inside $\R^V$.
\end{sbpara}

\begin{sbpara} \label{XV} In the case that $E$ is archimedean, we have 
$$
	X_V\cong \begin{cases} \PGL_d(\R)/\PO_d(\R)\cong \SL_d(\R)/\SO_d(\R) & \text{if } E \cong \R\\
	\PGL_d(\C)/\PU(d)\cong \SL_d(\C)/\SU(d) & \text{if } E \cong \C.
	\end{cases}
$$ 
In the case $E$ is non-archimedean, $X_V$ is identified with (a geometric realization of) the Bruhat-Tits building associated to $\PGL_V$ \cite{BT} (see also \cite[Section 2]{DH}).

\end{sbpara}

\begin{sbpara}\label{par} Recall that for a finite-dimensional vector space $H\neq 0$ over a field $I$, the following four objects are in one-to-one correspondence:
\begin{enumerate}
\item[(i)] a parabolic subgroup of the algebraic group $\GL_H$ over $I$,
\item[(ii)] a parabolic subgroup of the algebraic group $\PGL_H$ over $I$, 
\item[(iii)] a parabolic subgroup of the algebraic group $\SL_H$ over $I$, and
\item[(iv)] a flag of $I$-subspaces of $H$ (i.e., a set of subspaces containing $\{0\}$ and $H$ and totally ordered under inclusion).
\end{enumerate}
The bijections (ii) $\mapsto$ (i) and (i) $\mapsto$ (iii) are the taking of inverse images. The bijection (i) $\mapsto$ (iv) sends a parabolic subgroup $P$ to the set of all $P$-stable $I$-subspaces of $H$, and the converse map takes a flag to its isotropy subgroup in $\GL_H$. 
\end{sbpara}

\begin{sbpara}\label{bs2}  Let $\bar X_V$ be the set of all  pairs $(P, \mu)$, where $P$ is a parabolic subgroup of the algebraic group $\PGL_V$ over $E$ and, if 
$$0=V_{-1}\subsetneq V_0 \subsetneq\dots \subsetneq V_m=V$$
denotes the flag corresponding to $P$ (\ref{par}), then $\mu$ is a family $(\mu_i)_{0\leq i\leq m}$, where $\mu_i$ is an equivalence class of norms on $V_i/V_{i-1}$. 

We have an embedding $X_V \hookrightarrow \bar X_V$ which sends $\mu$ to $(\PGL_V, \mu)$.

\end{sbpara}

\begin{sbpara}\label{bs3} Let $\bar X_V^{\flat}$ be the set of all equivalence classes of nonzero semi-norms on the dual space $V^*$ of $V$ (\ref{eqv}). 
We have an embedding $X_V \hookrightarrow \bar X^{\flat}_V$ which sends $\mu$ to $\mu^*$ (\ref{dual}).

This set $\bar X_V^{\flat}$ is also identified with the set of pairs $(W, \mu)$ with $W$ a nonzero $E$-subspace of $V$ and $\mu$ an equivalence class of a norm on $W$. In fact, $\mu$ corresponds to an equivalence class $\mu^*$ of a norm on the dual space $W^*$ of $W$ (\ref{dual}), and $\mu^*$ is identified via the projection $V^*\to W^*$ with an equivalence class of semi-norms on $V^*$. 

We call the understanding of $\bar X_V^{\flat}$ as the set of such pairs $(W,\mu)$ the definition of $\bar X_V^{\flat}$ in the second style. 
In this interpretation of $\bar X^{\flat}_V$, the above embedding $X_V\to \bar X_V^{\flat}$ is written as $\mu\mapsto (V, \mu)$.

\end{sbpara}

\begin{sbpara}
In the case that $E$ is non-archimedean, $\bar X_V$  is 
the polyhedral compactification  of the Bruhat-Tits building $X_V$ by G\'erardin \cite{G} and Landvogt \cite{L} (see also \cite[Proposition 19]{GR}), and $\bar X_V^{\flat}$ is the compactification of $X_V$ by Werner \cite{W1, W2}.  
In the case that $E$ is archimedean, $\bar X_V$ is the maximal Satake compactification, and $\bar X_V^{\flat}$
is the minimal Satake compactification for the standard projective representation of $\PGL_V(E)$, as constructed by Satake in \cite{Sa1} (see also \cite[1.4]{BJ}).  The topologies of $\bar X_V$ and $\bar X_V^{\flat}$ are reviewed in Section \ref{2.3} below.

\end{sbpara}

\begin{sbpara}\label{cansur} We have a canonical surjection
$ \bar X_V\to \bar X_V^{\flat}$
which sends $(P, \mu)$ to $(V_0, \mu_0)$, where $V_0$ is as in \ref{bs2}, and where we use the definition of $\bar X_V^{\flat}$ of the second style in \ref{bs3}.
This surjection is compatible with the inclusion maps from $X_V$ to these spaces. 
\end{sbpara}

\begin{sbpara}\label{act2} We have the natural actions of  $\PGL_V(E)$ on $X_V$, $\bar X_V$ and $\bar X_V^{\flat}$ by \ref{actn}. 
These actions are compatible with the canonical maps between these spaces. 
\end{sbpara}

\subsection{Topologies}\label{2.3}

\begin{sbpara}\label{cover} We define a topology on $\bar X_V$.

Take a basis $(e_i)_i$ of $V$.
We have a commutative diagram
%$$
%\SelectTips{cm}{}
%\xymatrix{
%\PGL_V(E) \times \R^{d-1}_{>0} \ar[r] \ar@{^{(}->}[d] & X_V  \ar@{^{(}->}[d] \\
%\PGL_V(E) \times \R^{d-1}_{\geq 0} \ar[r] & \bar X_V. 
%}
%$$
$$
\begin{tikzcd}[column sep = large, row sep = small]\PGL_V(E) \times \R^{d-1}_{>0} \arrow{r} \arrow[hook]{d} & X_V  \arrow[hook]{d} \\
\PGL_V(E) \times \R^{d-1}_{\geq 0} \arrow{r} & \bar X_V. 
\end{tikzcd}
$$
Here the upper arrow is $(g,t) \mapsto g\mu$,
where $\mu$ is the class of the norm on $V$ associated to $((e_i)_i, (r_i)_i)$
with $r_i= \prod_{1\leq j<i} t_j^{-1}$, and where $g\mu$ is defined by the action of $\PGL_V(E)$ on $X_V$ (\ref{act2}).
 The lower arrow is $(g,t)\mapsto g(P, \mu)$, where $(P, \mu)\in \bar X_V$ is defined as follows, and 
$g(P,\mu)$ is then defined by the action of $\PGL_V(E)$ on $\bar X_V$ (\ref{act2}).
Let 
$$
	I=\{j \mid t_j=0\} \subset \{1, \dots, d-1 \},
$$ and write 
$$
 	I=\{c(i)\mid 0\leq i\leq m-1\},
$$ 
where $m=\sharp I$ and $1\leq c(0)<\dots < c(m-1)\leq d-1$.
If we also let $c(-1)=0$ and $c(m)=d$, then the set of 
$$
	V_i= \sum_{j=1}^{c(i)} F e_j
$$
with $-1 \le i \le m$ forms a flag in $V$, and
$P$ is defined to be the corresponding parabolic subgroup of $\PGL_V$ (\ref{par}). For $0\leq i \leq m$, 
we take $\mu_i$ to be the equivalence class of the norm on $V_i/V_{i-1}$ given by the basis $(e_j)_{c(i-1)<j\leq c(i)}$
and the sequence $(r_j)_{c(i-1)<j\leq c(i)}$ with $r_j= \prod_{c(i-1) <k < j} t_k^{-1}$.

Both the upper and the lower horizontal arrows in the diagram are surjective, and the topology on $X_V$ coincides with the topology as a quotient space of $\PGL_V(E) \times \R_{>0}^{d-1}$ via the upper horizontal arrow. 
The topology on $\bar X_V$ is defined as the quotient topology of the topology on $\PGL_V(E)\times \R_{\geq 0}^{d-1}$ via the lower horizontal arrow. 
It is easily seen that this topology is independent of the choice of the basis $(e_i)_i$. 

\end{sbpara}

\begin{sbpara} The space $\bar X_V^{\flat}$ has the following topology: the space of all nonzero semi-norms on $V^*$ has a topology as a subspace of the product $\R^{V^*}$, and $\bar X_V^{\flat}$ has a topology as a quotient of it.

\end{sbpara}

\begin{sbpara}\label{comH}  Both $\bar X_V$ and $\bar X_V^{\flat}$ are compact Hausdorff spaces containing $X_V$ as a dense open subset. This is proved in \cite{G, L, W1, W2} in the case that $E$ is non-archimedean and in \cite{Sa1, BJ} in the archimedean case.

\end{sbpara}

\begin{sbpara}

The topology on $\bar X^{\flat}_V$ coincides with the image of the topology on $\bar X_V$. In fact, it is easily seen that the canonical map $\bar X_V\to \bar X_V^{\flat}$ is continuous (using, for instance, \cite[Theorem 5.1]{W2}). Since both spaces are compact Hausdorff and this continuous map is surjective, the topology on $\bar X_V^{\flat}$ is the image of that of $\bar X_V$. 
\end{sbpara}

\section{Spaces associated to global fields} \label{globspace}
Let $F$ be a global field, which is to say, either a number field or a function field in one variable over a finite field. We fix a finite-dimensional $F$-vector space $V$ of dimension $d\geq 1$. 
For a place $v$ of $F$, let $V_v=F_v \otimes_F V$. We set $X_v = X_{V_v}$ and $X_v^{\flat} = X_{V_v}^{\flat}$ for 
brevity.  If $v$ is non-archimedean, 
we let $O_v$, $k_v$, $q_v$, $\varpi_v$ denote the valuation ring of $F_v$, the residue field of $O_v$, the order of $k_v$, and a fixed uniformizer in $O_v$, respectively.

In this section, we define sets $\bar X_{F,v}^{\star}$ containing $X_v$ for $\star \in \{ \sharp,\;\,, \flat\}$, which serve as our rational partial compactifications.  Here, $\bar X_{F,v}$ (resp., $\bar X_{F,v}^{\flat}$) is defined as a subset of 
$X_v$ (resp., $\bar X_v^{\flat}$),  and $\bar X^{\sharp}_{F,v}$ has $\bar X_{F,v}$ as a quotient.  In \S\ref{uphp}, by way of example, we describe these sets and various topologies on them in the case that $d = 2$, $F = \Q$, and $v$ is the real place. For $\star \neq \sharp$, we construct more generally sets $\bar X_{F,S}^{\star}$ for a nonempty finite set $S$ of places of $F$.   In \S\ref{defspace}, we describe $\bar X_{F,S}^{\star}$ as a subset of $\prod_{v\in S} \bar X_{F,v}^{\star}$.

In \S\ref{BSsec} and \S\ref{Satake}, we define topologies on these sets.  That is, in \S\ref{BSsec}, we define the ``Borel-Serre topology'', while in \S\ref{Satake}, we define the ``Satake topology'' on $\bar X_{F,v}$ and, assuming $S$ contains all archimedean places of $F$, on $\bar X_{F,S}^{\flat}$. 
 In \S\ref{propX}, we prove results on $\bar X_{F,v}$. 
  In \S\ref{comptop}, we compare the following topologies on $\bar X_{F,v}$ (resp., $\bar X_{F,v}^{\flat}$): the Borel-Serre topology, the Satake topology, and the topology as a subspace of $\bar X_v$ (resp., $\bar X_v^{\flat}$). In \S\ref{relations}, we describe the relationship between these spaces and Borel-Serre and reductive Borel-Serre spaces.

\subsection{Definitions of the spaces} \label{defspace}

\begin{sbpara}\label{bs6} Let $\bar X_{F,v}=\bar X_{V, F, v}$ be the subset of $\bar X_v$ consisting of all elements $(P, \mu)$  such that $P$ is $F$-rational. If $P$ comes from a parabolic subgroup $P'$ of $\PGL_V$ over $F$, we also denote $(P, \mu)$ by $(P',\mu)$.

\end{sbpara}

\begin{sbpara}
Let $\bar X_{F,v}^{\flat}$ be the subset of $\bar X_v^{\flat}$ consisting of all elements $(W, \mu)$ such that $W$ is $F$-rational (using the definition of $\bar X_v^{\flat}$ in the second style in \ref{bs3}). If $W$ comes from an $F$-subspace $W'$ of $V$, we also denote $(W,\mu)$ by $(W',\mu)$. 

\end{sbpara}

\begin{sbpara}\label{sharp}  Let $\bar X_{F,v}^{\sharp}$ be the set of all triples $(P, \mu, s)$ such that $(P, \mu)\in \bar X_{F,v}$ and $s$ is a splitting 
$$
	s \colon \bigoplus_{i=0}^m\;  (V_i/V_{i-1})_v \xrightarrow{\sim} V_v
$$ 
over $F_v$ of the filtration $(V_i)_{-1 \le i \le m}$ of $V$ corresponding to $P$.

We have an embedding 
$X_v \hookrightarrow \bar X_{F,v}^{\sharp}$ that sends $\mu$ to $(\PGL_V, \mu, s)$, where $s$ is the identity map of $V_v$. 

\end{sbpara}

\begin{sbpara}\label{bs7} We have a diagram with a commutative square
%$$\SelectTips{cm}{} \xymatrix{
%\bar X_{F,v}^{\sharp} \ar@{->>}[r] & \bar X_{F,v} \ar@{->>}[r] \ar@{^{(}->}[d]& \bar X_{F,v}^{\flat} 
%\ar@{^{(}->}[d] \\
%& \bar X_v \ar@{->>}[r] & \bar X_v^{\flat}.}
%$$
$$\begin{tikzcd}[row sep = small] 
\bar X_{F,v}^{\sharp} \arrow[two heads]{r} & \bar X_{F,v} \arrow[two heads]{r} \arrow[hook]{d}& \bar X_{F,v}^{\flat} \arrow[hook]{d} \\
& \bar X_v \arrow[two heads]{r} & \bar X_v^{\flat}.
\end{tikzcd}$$
Here, the first arrow in the upper row forgets the splitting $s$, and the second arrow in the upper row is $(P, \mu) \mapsto (V_0, \mu_0)$, as is the lower arrow (\ref{cansur}).
\end{sbpara}

\begin{sbpara} The group $\PGL_V(F)$ acts on the sets $\bar X_{F,v}$, $\bar X_{F,v}^{\flat}$ and $\bar X_{F,v}^{\sharp}$
in the canonical manner.

\end{sbpara}

 \begin{sbpara}  Now let $S$ be a nonempty finite set of places of $F$. 
 \begin{itemize}
 \itemsep = 0ex
\item Let $\bar X_{F,S}$ be the subset of $\prod_{v\in S} \bar X_{F,v}$ consisting of all elements $(x_v)_{v\in S}$ such that the parabolic subgroup of $G=\PGL_V$ associated to $x_v$  is independent of $v$. 
\item Let $\bar X_{F,S}^{\flat}$ be the subset of $\prod_{v\in S} \bar X_{F,v}^{\flat}$ consisting of all elements $(x_v)_{v\in S}$ such that the $F$-subspace of $V$ associated to $x_v$  is independent of $v$. 
\end{itemize}
We will denote an element of $\bar X_{F,S}$ as $(P,\mu)$, where $P$ is a parabolic subgroup of $G$ and $\mu\in \prod_{v\in S, 0\leq i\leq m} X_{(V_i/V_{i-1})_v}$ with $(V_i)_i$ the flag corresponding to $P$. We will denote an element of $\bar X_{F,S}^{\flat}$ as $(W,\mu)$, where $W$ is a nonzero $F$-subspace of $V$ and $\mu\in \prod_{v\in S} X_{W_v}$. 
We have a canonical surjective map 
$$ \bar X_{F,S} \to \bar X_{F,S}^{\flat}$$
which commutes with the inclusion maps from $X_S$ to these spaces. 
  \end{sbpara}

\subsection{Example: Upper half-plane} \label{uphp}

\begin{sbpara} Suppose that $F = \Q$, $v$ is the real place, and $d=2$.  
 
In this case, the sets $X_v$, $\bar X_v=\bar X_v^{\flat}$, $\bar X_{\Q,v}=\bar X_{\Q,v}^{\flat}$, and $\bar X_{\Q,v}^{\sharp}$ are described by using the upper half-plane. In \S\ref{sec2}, we discussed topologies on the first two spaces.  The remaining spaces also have natural topologies, as will be discussed in \S\ref{BSsec} and \S\ref{Satake}: the space $\bar X_{\Q,v}^{\sharp}$ is endowed with the Borel-Serre  topology, and $\bar X_{\Q,v}$ has two topologies, the Borel-Serre topology and Satake topology, which are both different from its topology as a subspace of $\bar X_v$.  In this section, as a prelude to \S\ref{BSsec} and \S\ref{Satake}, we describe what the Borel-Serre and Satake topologies look like in this special case. 

\end{sbpara}

\begin{sbpara} Let $\frak H=\{x+yi\mid x,y\in \R, y>0\}$ be the upper half-plane.  
 Fix a basis $(e_i)_{i=1,2}$ of $V$.
For $z \in \frak H$, let $\mu_z$ denote the class of the norm on $V$ corresponding to the class of the norm on $V^*$
given by $ae^*_1+be_2^*\mapsto |az+b|$ for $a,b\in \R$.   
Here $(e_i^*)_{1 \le i \le d}$ is the dual basis of $(e_i)_i$, and $|\;\;|$ denotes the usual absolute value (not the normalized absolute value) on $\C$. 
We have a homeomorphism
\begin{eqnarray*}
	\frak H \xrightarrow{\sim} X_v, \quad z\mapsto \mu_z
\end{eqnarray*}
which is compatible with the actions of $\SL_2(\R)$.

For the square root $i\in \frak H$ of $-1$, the norm $ae_1+be_2 \mapsto (a^2+b^2)^{1/2}$ has class $\mu_i$.  For $z=x+yi$, we have
$$\mu_z = \begin{pmatrix} y & x \\0 & 1\end{pmatrix}\mu_i.$$
The action of $\begin{pmatrix} -1 & 0 \\ 0 & 1 \end{pmatrix}\in \GL_2(\R)$ on $X_v$ corresponds to $x+yi\mapsto -x+yi$ on $\frak H$. 
\end{sbpara}

\begin{sbpara}
The inclusions $$X_v\subset \bar X_{\Q,v}\subset \bar X_v$$
can be identified with
$$\frak H \subset \frak H \cup \mb{P}^1(\Q) \subset \frak H \cup \mb{P}^1(\R).$$
Here $z \in \mb{P}^1(\R)=\R \cup \{\infty\}$ corresponds to the class in $\bar X_v^{\flat}= \bar X_v$ of the semi-norm $ae_1^*+be_2^*\mapsto |az +b|$ (resp., $ae_1^*+be_2^*\mapsto |a|$) on $V^*$ if $z \in \R$ (resp., $z=\infty$). These identifications are compatible with the actions of $\PGL_V(\Q)$.

The topology on $\bar X_v$ of \ref{cover} is the topology as a subspace of $\mb{P}^1(\C)$.

\end{sbpara}

\begin{sbpara} Let $B$ be the Borel subgroup of $\PGL_V$ consisting of all upper triangular matrices for the basis $(e_i)_i$, and let $0=V_{-1}\subsetneq V_0= \Q e_1\subsetneq V_1=V$ be the corresponding flag. Then $\infty\in \mb{P}^1(\Q)$ is understood as the point $(B, \mu)$ of $\bar X_{\Q,v}$, where $\mu$ is the unique element of
$X_{(V_0)_v}\times X_{(V/V_0)_v} $.

Let $\bar X_{\Q,v}(B)= \frak H\cup \{\infty\}\subset \bar X_{\Q,v}$ and let $\bar X_{\Q,v}^{\sharp}(B)$ be the inverse image of $\bar X_{\Q,v}(B)$ in $\bar X_{\Q,v}^{\sharp}$. 
Then for the Borel-Serre topology defined in \S\ref{BSsec}, we have a homeomorphism 
$$\bar X_{\Q,v}^{\sharp}(B)\cong \{x+yi  \mid  x\in \R, 0<y\leq \infty\}\supset \frak H.$$
Here $x+\infty i$ corresponds to $(B, \mu, s)$ where
 $s$ is the splitting of the filtration $(V_{i,v})_i$ 
given by the embedding $(V/V_0)_v \to V_v$ that sends the class of $e_2$ to $xe_1+e_2$. 

The Borel-Serre topology on $\bar X_{\Q,v}^{\sharp}$ is characterized by the properties that
\begin{enumerate}
\item[(i)] the action of the discrete group $\GL_V(\Q)$ on $\bar X_{\Q,v}^{\sharp}$ is continuous,
\item[(ii)] the subset $\bar X_{\Q,v}^{\sharp}(B)$ is open, and 
\item[(iii)] as a subspace, $\bar X_{\Q,v}^{\sharp}(B)$ is homeomorphic to $\{x+yi \mid x \in \R, 0 < y \le \infty \}$
as above.
\end{enumerate}

\end{sbpara}

\begin{sbpara}\label{HBS}
The Borel-Serre and Satake topologies on $\bar X_{\Q,v}$ (defined in \S\ref{BSsec} and \S\ref{Satake}) 
are characterized by the following properties:
\begin{enumerate}
\item[(i)] The subspace topology on $X_v\subset \bar X_{\Q,v}$ coincides with the topology on $\frak H$. 
\item[(ii)] The action of the discrete group $\GL_V(\Q)$ on $\bar X_{\Q,v}$ is continuous.  
\item[(iii)] The following sets (a) (resp., (b)) form a base of neighborhoods of $\infty$ for the Borel-Serre (resp., Satake) topology:
\begin{itemize}
	\item[(a)] the sets $U_f =\{x+yi\in \frak H \mid y \ge f(x)\}\cup \{\infty\}$ for continuous 
	$f \colon \R\to \R$,
	\item[(b)] the sets $U_c= \{x+yi\in \frak H \mid y \ge c\}\cup \{\infty\}$ with $c\in \R_{>0}$. 
\end{itemize}
\end{enumerate}
The Borel-Serre topology on $\bar X_{\Q,v}$ is the image of the Borel-Serre topology on $\bar X_{\Q,v}^{\sharp}$. 
\end{sbpara}

\begin{sbpara} For example, the set $\{x+yi\in \frak H \mid y > x\}\cup \{\infty\}$ is a neighborhood of $\infty$ for the Borel-Serre topology, but it is not a neighborhood of $\infty$ for the Satake topology.

\end{sbpara}

\begin{sbpara} For any subgroup $\Gamma$ of $\PGL_2(\Z)$ of finite index, the Borel-Serre and Satake topologies induce the same topology on the quotient space $X(\Gamma) = \Gamma \bs \bar X_{\Q,v}$.  Under this quotient topology, $X(\Gamma)$ is compact Hausdorff.   If $\Gamma$ is the image of a congruence subgroup of $\SL_2(\Z)$, then 
this is the usual topology on the modular curve $X(\Gamma)$.

\end{sbpara}

\subsection{Borel-Serre topology} \label{BSsec}

\begin{sbpara}
For a parabolic subgroup $P$ of $\PGL_V$, let $\bar X_{F,v}(P)$ (resp., $\bar X_{F,v}^{\sharp}(P)$) be the subset of $\bar X_{F,v}$ (resp., $\bar X_{F,v}^{\sharp}$) consisting of all elements $(Q, \mu)$ (resp., $(Q,\mu,s)$) such that $Q\supset P$. 

The action of $\PGL_V(F_v)$ on $\bar X_v$ induces an action of $P(F_v)$ on $\bar X_{F,v}(P)$. We have also an action of $P(F_v)$ on  $\bar X_{F,v}^{\sharp}(P)$ given by 
$$
	g(\alpha, s) = (g\alpha, g\circ s\circ g^{-1})
$$ 
for $g\in P(F_v)$, $\alpha \in \bar X_{F,v}(P)$, and $s$ a splitting of the filtration. 

\end{sbpara}

\begin{sbpara}\label{Pstr1} Fix a basis $(e_i)_i$ of $V$. 
Let $P$ be a parabolic subgroup of $\PGL_V$ such that
\begin{itemize}
	\item if $0=V_{-1}\subsetneq V_0\subsetneq \dots \subsetneq V_m=V$ denotes the flag of 
	$F$-subspaces corresponding to $P$, 
	then each $V_i$ is generated by the $e_j$ with $1\leq j\leq c(i)$, where $c(i)= \dim(V_i)$. 
\end{itemize}
This condition on $P$ is equivalent to the condition that $P$ contains the Borel subgroup $B$ of $\PGL_V$ consisting of all upper triangular matrices with respect to $(e_i)_i$. 
Where useful, we will identify $\PGL_V$ over $F$ with
$\PGL_d$ over $F$ via the basis $(e_i)_i$.  

Let 
$$
	\Delta(P)=\{\dim(V_j) \mid 0\leq j\leq m-1\}\subset \{1,\dots, d-1\},
$$
and let $\Delta'(P)$ be the complement of $\Delta(P)$ in $\{1,\dots, d-1\}$. 
Let $\R_{\geq 0}^{d-1}(P)$ be the open subset of $\R_{\geq 0}^{d-1}$ given by
$$
	\R_{\geq 0}^{d-1}(P) = \{ (t_i)_{1\leq i\leq d-1} \in \R_{\geq 0}^{d-1} \mid t_i > 0 \text{ for all } i \in \Delta'(P) \}.
$$
In particular, we have
$$
	\R_{\geq 0}^{d-1}(P)\cong \R_{>0}^{\Delta'(P)}\times \R_{\geq 0}^{\Delta(P)}.
$$ 
\end{sbpara}

\begin{sbpara}\label{cover2} 
With $P$ as in \ref{Pstr1}, the map $\PGL_V(F_v) \times \R^{d-1}_{\geq 0} \to \bar X_v$ in \ref{cover} induces a map 
 $$\bar{\pi}_{P,v} \colon P(F_v) \times \R^{d-1}_{\geq 0}(P)\to \bar X_{F,v}(P),$$
 which restricts to a map
 $$
 	\pi_{P,v} \colon P(F_v) \times \R^{d-1}_{> 0} \to X_{F,v}.
 $$ 
The map $\bar{\pi}_{P,v}$ is induced by a map
$$ \bar{\pi}_{P,v}^{\sharp} \colon P(F_v) \times \R^{d-1}_{\geq 0}(P)\to \bar X_{F,v}^{\sharp}(P)$$
defined as $(g,t) \mapsto g(P, \mu, s)$ where $(P,\mu)$ is as in \ref{cover}  and $s$ is the splitting of the filtration $(V_i)_{-1 \le i \le m}$ defined by the basis $(e_i)_i$.  For this splitting $s$, we set 
$$
	V^{(i)} = s(V_i/V_{i-1}) = \sum_{c(i-1)<j\le c(i)} Fe_j
$$
for $0 \le i \le m$ so that $V_i = V_{i-1} \oplus V^{(i)}$ and
$V = \bigoplus_{i=0}^mV^{(i)}$.  If $P = B$, then we will often omit the subscript $B$ from our notation for these maps.
\end{sbpara}

\begin{sbpara} \label{decomp}
	We review the Iwasawa decomposition.
	For $v$ archimedean (resp., non-archimedean), let $A_v \le \PGL_d(F_v)$ be the subgroup of elements of 
	that lift to 
	diagonal matrices in $\GL_d(F_v)$ with positive real entries (resp., with entries that are powers of $\varpi_v$).  
	Let $K_v$ denote the standard maximal compact subgroup of $\PGL_d(F_v)$ given by 
	$$
		K_v = \begin{cases} \PO_d(\R) & \text{if } v \text{ is real}, \\ 
		\PU_d & \text{if } v \text{ is complex}, \\
		\PGL_d(O_v) & \text{otherwise}. \end{cases}
	$$
	Let $B_u$ denote the upper-triangular unipotent matrices in the standard Borel $B$.
	The Iwasawa decomposition is given by the equality
	$$
		\PGL_d(F_v) = B_u(F_v)A_vK_v.
	$$ 
\end{sbpara}	
	
\begin{sbpara} \label{Iwarch}
	If $v$ is archimedean, then the expression of a matrix in $\PGL_d(F_v)$ as a product in the Iwasawa 
	decomposition is unique.
\end{sbpara}

\begin{sbpara}	\label{Iwnon}
	If $v$ is non-archimedean, then 
	the Bruhat decomposition is $\PGL_d(k_v) = B(k_v)S_d B(k_v)$, where the symmetric group $S_d$
	of degree $d$ is viewed as a subgroup of $\PGL_d$ over any field 
	via the permutation representation on the standard basis.  This implies that
	$\PGL_d(O_v)= B(O_v)S_d \Iw(O_v)$, where $\Iw(O_v)$ is the Iwahori subgroup
	consisting of those matrices in with upper triangular image in $\PGL_d(k_v)$.   
	Combining this with the Iwasawa decomposition
	(in the notation of \ref{decomp}), 
	we have
	$$
		\PGL_d(F_v)= B_u(F_v)A_v S_d \Iw(O_v).
	$$
	This decomposition is not unique, since $B_u(F_v) \cap \Iw(O_v) = B_u(O_v)$.
\end{sbpara}

\begin{sbpara} \label{diag}
	If $v$ is archimedean, then there is a bijection $\R_{>0}^{d-1} \xrightarrow{\sim} A_v$ given by 
	$$
		t = (t_k)_{1 \le k \le d-1} \mapsto a = \begin{cases} \diag(r_1, \ldots, r_d)^{-1} & \text{if } v \text{ is real},\\
		\diag(r_1^{1/2}, \ldots, r_d^{1/2})^{-1} & \text{if } v \text{ is complex},
		\end{cases}
	$$
	where $r_i = \prod_{k=1}^{i-1} t_k^{-1}$ as in \ref{cover}.
\end{sbpara}

\begin{sbprop}\label{lst0} \
\begin{enumerate}
\item[(1)] 
Let $P$ be a parabolic subgroup of $\PGL_V$ as in \ref{Pstr1}.
Then the maps 
\begin{eqnarray*}
	\bar{\pi}_{P,v} \colon P(F_v) \times \R^{d-1}_{\geq 0}(P) \to \bar X_{F,v}(P) &\text{and}& 
	\bar{\pi}_{P,v}^{\sharp} \colon P(F_v) \times \R^{d-1}_{\geq 0}(P) \to \bar X_{F,v}^{\sharp}(P)
\end{eqnarray*} 
of \ref{cover2} are surjective. 
\item[(2)] For the Borel subgroup $B$ of \ref{Pstr1},
the maps 
\begin{eqnarray*}
	&\pi_v \colon B_u(F_v) \times \R^{d-1}_{>0}\to  X_v, \qquad
	\bar{\pi}_v \colon B_u(F_v)\times \R^{d-1}_{\geq 0}  \to  \bar X_{F,v}(B),& \\ 
	 &\text{and} \quad \bar{\pi}_v^{\sharp} \colon B_u(F_v)\times \R^{d-1}_{\geq 0}  \to  \bar X_{F,v}^{\sharp}(B).&
\end{eqnarray*}
of \ref{cover2} are all surjective.
\item[(3)] If $v$ is archimedean, then $\pi_v$ and $\bar{\pi}_v^{\sharp}$ are bijective.
\item[(4)] If $v$ is non-archimedean, then $\bar{\pi}_v$ induces a bijection
$$(B_u(F_v) \times \R^{d-1}_{\geq 0})/{\sim} \to \bar X_{F,v}(B)$$
where $(g, (t_i)_i) \sim (g', (t'_i)_i)$ if and only if
\begin{enumerate}
\item[(i)] $t_i=t'_i$ for all $i$ and
\item[(ii)] 
$|(g^{-1}g')_{ij}| \leq (\prod_{i\leq k<j} t_k)^{-1}$ for all $1 \le i < j \le d$,
considering any $c \in \R$ to be less than $0^{-1}=\infty$.
\end{enumerate}
\end{enumerate}

\end{sbprop}

\begin{pf} 
If $\bar{\pi}_v^{\sharp}$ is surjective, then for any parabolic $P$ containing $B$, the restriction of $\bar{\pi}_v^{\sharp}$ to $B_u(F_v) \times \R_{\ge 0}^{d-1}(P)$ has image
$\bar{X}_{F,v}^{\sharp}(P)$.  Since $B_u(F_v) \subset P(F_v)$, this forces the surjectivity of $\bar{\pi}_{P,v}^{\sharp}$,
hence of $\bar{\pi}_{P,v}$ as well.  So, we turn our attention to (2)--(4).  
If $r \in \R_{> 0}^d$, we let $\mu^{(r)} \in X_v$ denote the class of the norm attached to the basis $(e_i)_i$ and $r$.

Suppose first that $v$ is archimedean.  
By the Iwasawa decomposition \ref{decomp}, and noting \ref{Iwarch} and \ref{XV}, we see that $B_u(F_v)A_v \to X_v$ given by $g \mapsto g\mu^{(1)}$ is bijective, where
$\mu^{(1)}$ denotes the class of the norm attached to $(e_i)_i$ and $1 = (1,\ldots,1) \in \R_{>0}^d$.  
For $t \in \R_{>0}^{d-1}$, let $a \in A_v$ be its image under the bijection in \ref{diag}.
Since $pa\mu^{(1)} = p\mu^{(r)}$, for $p \in B_u(F_v)$ and $r$ as in \ref{cover}, we have the bijectivity of $\pi_v$.

Consider the third map in (2).  For $t \in \R_{\ge 0}^{d-1}$, let $P$ be the parabolic that contains $B$ and is determined by
the set $\Delta(P)$ of $k \in \{1,\ldots,d-1\}$ for which $t_k = 0$.  Let $(V_i)_{-1 \le i \le m}$ be the corresponding flag.
Let $M$ denote the Levi subgroup of $P$.  (It is the quotient of $\prod_{i=0}^m \GL_{V^{(i)}} < \GL_V$ by scalars, where
$V=\bigoplus_{i=0}^m V^{(i)}$ as in \ref{cover2}, and $M \cap B_u$ is isomorphic to the product of the upper-triangular unipotent matrices in each $\PGL_{V^{(i)}}$.)  The product of the first maps in (2) for the blocks of $M$ is a bijection
$$
	(M \cap B_u)(F_v) \times \R_{>0}^{\Delta'(P)} \xrightarrow{\sim} 
	\prod_{i=0}^m X_{V_v^{(i)}}  \subset \bar{X}_{F,v}(P), 
$$
such that $(g,t')$ is sent to $(P,g\mu)$ in $\bar X_{F,v}$, where $\mu$ is the sequence of classes of norms determined by $t'$ and the standard basis.
The stabilizer of $\mu$ in $B_u$ is the unipotent radical $P_u$ of $P$, and this $P_u$ acts simply transitively on the set of splittings for the graded quotients $(V_i/V_{i-1})_v$.  Since
$B_u = (M \cap B_u)P_u$, and this decomposition is unique, we have the desired bijectivity of $\bar{\pi}_v^{\sharp}$, proving (3).

Suppose next that $v$ is non-archimedean.  We prove the surjectivity of the first map in (2). 
Using the natural actions of $A_v$ and the symmetric group $S_d$ on 
$\R^d_{>0}$, we see that any norm on $V_v$ can be written as 
$g\mu^{(r)}$, where $g \in \PGL_d(F_v)$ and $r = (r_i)_i \in \R_{>0}^d$, with $r$ satisfying
\begin{equation*}
r_1\leq r_2\leq \dots \leq r_d \leq q_vr_1. 
\end{equation*}
For such an $r$, the class $\mu^{(r)}$ is invariant under the action of $\Iw(O_v)$. Hence for such an $r$, 
any element of $S_d \Iw(O_v)\mu^{(r)}= S_d\mu^{(r)}$ is of the form $\mu^{(r')}$, where $r' = 
(r_{\sigma(i)})_i$ for some 
$\sigma \in S_d$.  Hence, any element of $A_v S_d \Iw(O_v)\mu^{(r)}$ for such an $r$
is of the form $\mu^{(r')}$ for some $r' = (r'_i)_i \in \R_{>0}^d$.  This proves the surjecivity of the first map of (2).
The surjectivity of the other maps in (2) is then shown using this, 
similarly to the archimedean case. 

Finally, we prove (4). It is easy to see that the map $\bar \pi_v$ factors through the quotient by the equivalence relation. We can deduce the bijectivity in question from the bijectivity of $(B_u(F_v) \times \R^{d-1}_{>0})/{\sim} \to X_v$, replacing $V$ by $V_i/V_{i-1}$ as in the above arguments for the archimedean case.  
Suppose that $\pi_v(g,t) = \pi_v(1,t')$ for $g \in B_u(F_v)$ and $t, t' \in \R_{> 0}^{d-1}$.
We must show that $(g,t) \sim (1,t')$.  
Write $\pi_v(g,t) = g\mu^{(r)}$ and $\pi_v(1,t') = \mu^{(r')}$ with $r = (r_i)_i$ and $r' = (r'_i)_i \in \R^d_{>0}$
such that $r_1 = 1$ and $r_j/r_i = (\prod_{i\le k < j} t_k)^{-1}$ for all $1 \le i < j \le d$, and similarly for $r'$ and $t'$. 
It then suffices to check that $r'=r$ and $r_i|g_{ij}|\leq r_j$ for all $i < j$.
Since $\mu^{(r)} = g^{-1}\mu^{(r')}$, there exists $c \in \R_{>0}$ such that
$$
	\max\{r_i|x_i| \mid  1\leq i \leq d\}= c\max\{ r'_i|(gx)_i| \mid 1\leq i\leq d\}
$$ 
for all $x=(x_i)_i\in F_v^d$.  
Taking $x = e_1$, we have $gx = e_1$ as well, so $c=1$.
Taking $x=e_i$, we obtain $r_i \geq r'_i$, and taking $x=g^{-1}e_i$, we obtain $r_i\leq r'_i$.  Thus $r = r'$, and
taking $x=e_j$ yields $r_j = \max\{r_i |g_{ij}| \mid 1\leq i\leq j\}$, which tells us that $r_j\geq r_i|g_{ij}|$ for $i<j$.  
\end{pf}

\begin{sbprop}\label{BStop}

There is a unique topology on $\bar X_{F,v}$ (resp., $\bar X_{F,v}^{\sharp}$) satisfying the following conditions (i) and (ii).
\begin{enumerate}
\item[(i)] For every parabolic subgroup $P$ of $\PGL_V$, the set $\bar X_{F,v}(P)$ (resp., $\bar X_{F,v}^{\sharp}(P)$) is open in $\bar X_{F,v}$ (resp., 
$\bar X_{F,v}^{\sharp}$). 
\item[(ii)] For every parabolic subgroup $P$ of $\PGL_V$ and basis $(e_i)_i$ of $V$ such that $P$ contains the Borel
subgroup with respect to $(e_i)_i$, the topology on $\bar X_{F,v}(P)$ (resp., $\bar X_{F,v}^{\sharp}(P)$) is the topology as a quotient of $P(F_v)\times \R_{\geq 0}^{d-1}(P)$ under the surjection of \ref{lst0}(1). 
\end{enumerate}
This topology is also characterized by (i) and the following (ii)'. 
\begin{enumerate}
\item[(ii)'] If $B$ is a Borel subgroup of $\PGL_V$ consisting of upper triangular matrices with respect to a basis $(e_i)_i$
of $V$, then the topology on $\bar X_{F,v}(B)$ (resp., $\bar X_{F,v}^{\sharp}(B)$) is the topology as a quotient of $B_u(F_v)\times \R_{\geq 0}^{d-1}$ under the surjection of \ref{lst0}(2).
\end{enumerate}
\end{sbprop}

\begin{pf} 

The uniqueness is clear if we have existence of a topology satisfying (i) and (ii).  Let $(e_i)_i$ be a basis of $V$, let $B$ be the Borel subgroup of $\PGL_V$ with respect to this basis, and let $P$ be a parabolic subgroup of $\PGL_V$  containing $B$.  It suffices to prove that for the topology on $\bar X_{F,v}(B)$ (resp., $\bar X_{F,v}^{\sharp}(B)$)  as a quotient of $B_u(F_v) \times \R^{d-1}_{\geq 0}(B)$, the subspace topology on $\bar X_{F,v}(P)$ (resp., $\bar X_{F,v}^{\sharp}(P)$)
coincides with the quotient topology from
$P(F_v) \times \R^{d-1}_{\geq 0}(P)$.  For this, it is enough to show that 
the action of the topological group $P(F_v)$ on $\bar X_{F,v}(P)$ (resp., $\bar X_{F,v}^{\sharp}(P)$) is continuous
with respect to the topology on $\bar X_{F,v}(P)$ (resp., $\bar X_{F,v}^{\sharp}(P)$) as a quotient of $B_u(F_v) \times \R^{d-1}_{\geq 0}(P)$.  We must demonstrate this continuity.

Let  $(V_i)_{-1 \le i \le m}$ 
be the flag corresponding to $P$, and let $c(i)=\dim(V_i)$. 
For $0\leq i\leq m$, we regard $\GL_{V^{(i)}}$ as a subgroup of $\GL_V$ via the decomposition $V=\bigoplus_{i=0}^m V^{(i)}$ of \ref{cover2}.  

Suppose first that $v$ is archimedean. For $0\leq i\leq m$, let $K_i$ be the compact subgroup of $\GL_{V^{(i)}}(F_v)$ 
that is the isotropy group of the norm on $V^{(i)}$ given by the basis $(e_j)_{c(i-1)<j\leq c(i)}$ and $(1,\dots,1) \in \prod_{c(i-1)<j\leq c(i)}\R_{>0}$.  We identify $\R_{>0}^{d-1}$ with $A_v$ as in \ref{diag}.
By the Iwasawa decomposition \ref{decomp} and its uniqueness in \ref{Iwarch}, the product on $P(F_v)$ induces a homeomorphism
$$
	(a,b,c) \colon P(F_v) \xrightarrow{\sim} B_u(F_v)\times \R_{>0}^{d-1} \times 
	\left(\prod_{i=0}^m K_i\right)/\{z\in F^\times_v \mid |z|=1\}.
$$ 
We also have a product map $\phi \colon P(F_v) \times B_u(F_v) \times \R_{>0}^{\Delta'(P)} \to P(F_v)$, where we identify $t' \in \R_{>0}^{\Delta'(P)}$ with 
the diagonal matrix $\diag(r_1, \ldots, r_d)^{-1}$ if $v$ is real and $\diag(r_1^{1/2}, \ldots, r_d^{1/2})^{-1}$ if $v$ is complex, with $r_j^{-1} = \prod_{c(i-1)<k<j} t'_k$ for $c(i-1) < j \le c(i)$ as in \ref{cover}.
These maps fit in a commutative diagram
%$$ \SelectTips{cm}{} \xymatrix@C=40pt{
%P(F_v) \times \R_{\geq 0}^{\Delta(P)} \ar[d]
%& P(F_v) \times B_u(F_v)\times \R_{>0}^{\Delta'(P)} \times \R_{\geq 0}^{\Delta(P)} \ar@{->>}[r]^-{(\id,\bar{\pi}_{P,v})} 
%\ar@{->>}[l]_-{(\phi,\id)}
%& P(F_v) \times \bar X_{F,v}(P) \ar[d] \\
%B_u(F_v) \times \R_{\geq 0}^{d-1}(P) \ar@{->>}[rr]^-{\bar{\pi}_{P,v}} && \bar X_{F,v}(P) }
%$$
$$ \begin{tikzcd}[column sep=large] 
P(F_v) \times \R_{\geq 0}^{\Delta(P)} \arrow{d}
& P(F_v) \times B_u(F_v)\times \R_{>0}^{\Delta'(P)} \times \R_{\geq 0}^{\Delta(P)} \arrow[two heads]{r}{(\id,\bar{\pi}_{P,v})}   
\arrow[two heads]{l}[swap]{(\phi,\id)}
& P(F_v) \times \bar X_{F,v}(P) \arrow{d} \\
B_u(F_v) \times \R_{\geq 0}^{d-1}(P) \arrow[two heads]{rr}{\bar{\pi}_{P,v}} && \bar X_{F,v}(P)
\end{tikzcd}$$
in which the right vertical arrow is the action of $P(F_v)$ on $\bar X_{F,v}(P)$, and the left vertical arrow is the 
continuous map
$$
	(u,t)\mapsto (a(u), b(u)\cdot(1,t)), \quad (u,t) \in P(F_v) \times \R_{\ge 0}^{\Delta(P)}
$$ 
for $(1, t)$ the element of $\R_{\geq 0}^{d-1}(P)$ with $\R_{>0}^{\Delta'(P)}$-component $1$ and $\R_{\ge 0}^{\Delta(P)}$-component $t$.  (To see the commutativity, note that $c(u)$ commutes with the block-scalar matrix determined by $(1,t)$.)
We also have a commutative diagram of the same form for $\bar X_{F,v}^{\sharp}$. 
Since the surjective horizontal arrows are quotient maps, we have the continuity of the action of $P(F_v)$.

Next, we consider the case that $v$ is non-archimedean. For $0\leq i\leq m$, let $S^{(i)}$ be the group of permutations of the set 
$$
	I_i = \{j\in \Z \mid c(i-1)<j \le c(i)\},
$$ 
and regard it as a subgroup of $\GL_{V^{(i)}}(F)$.  Let $A_v$ be the subgroup of the diagonal
torus of $\PGL_V(F_v)$ with respect to the basis $(e_i)_i$ with entries powers of a fixed uniformizer, as in \ref{decomp}.

Consider the action of $A_v \prod_{i=0}^m S^{(i)}\subset P(F_v)$ on $\R^{d-1}_{\geq 0}(P)$ that is compatible with the action of $P(F_v)$ on $\bar X_{F,v}(P)$ via the embedding $\R^{d-1}_{\geq 0}(P) \to \bar X_{F,v}(P)$. 
This action is described as follows.  Any matrix $a= \diag(a_1, \ldots, a_d) \in A_v$ sends $t\in \R_{\geq 0}^{d-1}(P)$ to 
$(t_j |a_{j+1}||a_j|^{-1})_j \in \R_{\ge 0}^{d-1}(P)$. 
The action of $\prod_{i=0}^m S^{(i)}$
on $\R_{\geq 0}^{d-1}(P)$ is the unique continuous action which is
compatible with the  evident action of $\prod_{i=0}^m S^{(i)}$ on
$\R_{>0}^d$ via the map $\R_{>0}^d \to \R^{d-1}_{\ge 0}(P)$ that sends
$(r_i)_i$ to $(t_j)_j$, where
$t_j=r_j/r_{j+1}$.  That is, for
$$
	\sigma = (\sig_i)_{0\leq i\leq m} \in \prod_{i=0}^m S^{(i)},
$$ 
let $f \in S_d$ be the unique permutation with $f|_{I_i} = \sig_i^{-1}$ for all $i$. Then $\sigma$ sends $t \in \R_{\geq 0}^{d-1}(P)$ to the element $t' = (t'_j)_j$ given by
 $$
 	t'_j=\begin{cases} \prod_{f(j)\leq k < f(j+1)} t_k & \text{if } f(j)<f(j+1),\\
	\prod_{f(j+1)\leq k < f(j)} t_k^{-1} & \text{if } f(j+1)<f(j).
	\end{cases}
$$ 

Let $C$ be the compact subset of $\R^{d-1}_{\geq 0}(P)$ given by
$$
	C = \Bigg\{ t = (t_j)_j \in \R^{d-1}_{\ge 0}(P) \cap [0,1]^{d-1} \mid \prod_{c(i-1)<j<c(i)} t_j\geq q_v^{-1}
	\text{ for all } 0 \le i \le m \Bigg\}.
$$
We claim that for each $x\in \R^{d-1}_{\geq 0}(P)$, 
there is a finite family $(h_k)_k$ of elements of $A_v \prod_{i=0}^m S^{(i)}$ such that the union  $\bigcup_k h_k C$ is a neighborhood of $x$.  This is quickly reduced to the following claim.

\medskip

{\bf Claim.}   
Consider the natural action of $H = A_vS_d \subset \PGL_V$ on the quotient space $\R^d_{>0}/\R_{>0}$, with the class of $(a_j)_j$ in $A_v$ acting as multiplication by $(|a_j|)_j$. Let $C$ be the image of 
$$
	\{r\in \R_{>0}^d \mid  r_1\leq r_2\leq \dots \leq r_d\leq q_v r_1\}
$$ 
in $\R_{>0}^d/\R_{>0}$. Then for each $x\in \R^d_{>0}/\R_{>0}$, there is a finite family $(h_k)_k$ of elements of $H$ such that $\bigcup_k h_kC$ is a neighborhood of $x$.

\begin{proof}[Proof of Claim] \renewcommand{\qedsymbol}{} 
	This is a well-known statement in the theory of Bruhat-Tits buildings: the quotient $\R^d_{>0}/\R_{>0}$ is called the apartment of the Bruhat-Tits building $X_v$ of $\PGL_V$, and the set $C$ is a ($d-1$)-simplex in this apartment.  Any ($d-1$)-simplex in this apartment has the form $hC$ for some $h\in H$, for any $x\in \R^d_{>0}/\R_{>0}$ there are only finitely many ($d-1$)-simplices in this apartment which contain $x$, and the union of these is a neighborhood of $x$ in $\R^d_{>0}/\R_{>0}$. 
\end{proof}

By compactness of $C$, the topology on the neighborhood $\bigcup_k h_kC$ of $x$ is the quotient
topology from $\coprod_k h_kC$.  Thus, it is enough to show that for each $h\in A_v \prod_{i=0}^m S^{(i)}$, the composition 
$$
	P(F_v) \times B_u(F_v) \times h C \xrightarrow{(\id,\pi_{P,v})} P(F_v) \times \bar X_{F,v}(P) \to \bar X_{F,v}(P)
$$ 
(where the second map is the action) and its analogue for $\bar X_{F,v}^{\sharp}$ are continuous. 

For $0\leq i\leq m$, let $\Iw_i$ be the Iwahori subgroup of $\GL_{V^{(i)}}(F_v)$
 for the basis $(e_j)_{c(i-1)<j\leq c(i)}$. By the the Iwasawa and Bruhat decompositions as in
 \ref{Iwnon}, 
the product on $P(F_v)$ induces a continuous surjection
 $$
	B_u(F_v) \times A_v\prod_{i=0}^m S^{(i)} \times  \prod_{i=0}^m  \Iw_i\to P(F_v),
$$ 
and it admits continuous sections locally on $P(F_v)$.  (Here, the middle group $A_v\prod_{i=0}^m S^{(i)}$ 
has the discrete topology.)
Therefore, there exist an open covering $(U_{\lam})_{\lam}$ of $P(F_v)$ and, for each $\lambda$, 
a subset $\mathcal{U}_{\lam}$ of the above product mapping homeomorphically to $U_{\lambda}$, together with a continuous 
map
 $$
 	(a_{\lam}, b_{\lam}, c_{\lam}) \colon \mathcal{U}_{\lam} \to B_u(F_v) \times A_v
 	\prod_{i=0}^m S^{(i)} \times  \prod_{i=0}^m  \Iw_i
$$  
such that its composition with the above product map is the map $\mathcal{U}_{\lam} \xrightarrow{\sim} U_{\lam}$. 
Let $U'_{\lam}$ denote the inverse image of $U_{\lam}$ under 
$$
	P(F_v) \times B_u(F_v)\to P(F_v), \quad (g, g')\mapsto gg'h,
$$ 
so that $(U'_{\lam})_{\lam}$ is an open covering of $P(F_v) \times B_u(F_v)$.
For any $\gamma$ in the indexing set of the cover, let $\mathcal{U}'_{\lam,\gamma}$ be the inverse image of $U'_{\lam}$ in 
$\mathcal{U}_{\gamma} \times B_u(F_v)$.  Then the images of the $\mathcal{U}'_{\lam,\gamma}$ form an open cover of $P(F_v)
\times B_u(F_v)$ as well.
Let $(a'_{\lam,\gamma},b'_{\lam,\gamma})$ be the composition 
$$
	\mathcal{U}'_{\lam,\gamma} \to \mathcal{U}_{\lambda} \xrightarrow{(a_{\lam},b_{\lam})} B_u(F_v) \times A_v\prod_{i=0}^mS^{(i)}.
$$
As $\prod_{i=0}^m \Iw_i$ fixes every element of $C$ under its embedding in $\bar X_{F,v}(P)$, we have a commutative diagram
%$$
% 	\SelectTips{cm}{} \xymatrix{ 
%	\mathcal{U}'_{\lam,\gamma} \times hC \ar@{^{(}->}[r] \ar[d] & P(F_v) \times B_u(F_v) \times hC \ar[d] \\
%	B_u(F_v) \times \R_{\geq 0}^{d-1}(P) \ar@{->>}[r] & \bar X_{F,v}(P) }
%$$
$$
 	\begin{tikzcd} \mathcal{U}'_{\lam,\gamma} \times hC \arrow[hook]{r} \arrow{d} & P(F_v) \times B_u(F_v) \times hC \arrow{d} \\
	B_u(F_v) \times \R_{\geq 0}^{d-1}(P) \arrow[two heads]{r} & \bar X_{F,v}(P)
	\end{tikzcd}
$$
in which the left vertical arrow is 
$$
	(u, hx) \mapsto (a'_{\lam,\gamma}(u), b'_{\lam,\gamma}(u)x)
$$ 
for $x\in C$.  We also have a commutative diagram of the same form for $\bar X_{F,v}^{\sharp}$. 
This proves the continuity of the action of $P(F_v)$.
\end{pf}

\begin{sbpara} \label{BStop2}

We call the topology on $\bar X_{F,v}$ (resp., $\bar X_{F,v}^{\sharp}$) in \ref{BStop} the Borel-Serre topology.   The Borel-Serre topology on $\bar X_{F,v}$ coincides with the quotient topology of  the Borel-Serre topology on $\bar X_{F,v}^{\sharp}$.  This topology on $\bar X_{F,v}$ is finer than the subspace topology from $\bar X_v$.
 
We define the Borel-Serre topology on $\bar X_{F,v}^{\flat}$ as the quotient topology of the Borel-Serre topology of $\bar X_{F,v}$.  This topology on $\bar X_{F,v}^{\flat}$ is finer than the subspace topology from $\bar X_v^{\flat}$. 

For a nonempty finite set $S$ of places of $F$, we define the Borel-Serre topology on $\bar X_{F,S}$ (resp., $\bar X_{F,S}^{\flat}$) as the subspace topology for the product topology on $\prod_{v \in S} \bar X_{F,v}$ 
(resp., $\prod_{v \in S} \bar X_{F,v}^{\flat}$) for the Borel-Serre topology on each $\bar X_{F,v}$ (resp., $\bar X_{F,v}^{\flat}$).
 
\end{sbpara}

\subsection{Satake topology} \label{Satake}

\begin{sbpara} For a nonempty finite set of places $S$ of $F$, 
we define the Satake topology on $\bar X_{F,S}$ and, under the assumption $S$ contains all archimedean places,  on $\bar X_{F,S}^{\flat}$.  

The Satake topology is coarser than the Borel-Serre topology of \ref{BStop2}.  On the other hand, the Satake topology and the Borel-Serre topology induce the same topology on the quotient space by an arithmetic group (\ref{same}). Thus, the Hausdorff compactness of this quotient space can be formulated without using the Satake topology (i.e., using only the Borel-Serre topology). However, arguments involving the Satake topology appear naturally in the proof of this property. 
One nice aspect of the Satake topology is that each point has an explicit base of neighborhoods  (\ref{HBS}, \ref{Sat5}, \ref{flat4}). 
\end{sbpara}

\begin{sbpara}\label{subq} Let $H$ be a finite-dimensional vector space over a local field $E$. Let $H'$ and $H''$ be $E$-subspaces of $H$ such that $H'\supset H''$. Then a norm $\mu$ on $H$ induces a norm $\nu$ on $H'/H''$ as follows. Let $\mu'$ be the restriction of $\mu$ to $H'$. Let $(\mu')^*$ be the norm on $(H')^*$ dual to $\mu'$. Let $\nu^*$ be the restriction of $(\mu')^*$ to the subspace $(H'/H'')^*$ of $(H')^*$. Let $\nu$ be the dual of $\nu^*$. 
This norm $\nu$ is given on $x \in H'/H''$ by 
$$
	\nu(x) = \inf\,\{\mu(\tilde{x}) \mid \tilde{x} \in H' \text{ such that } \tilde{x} + H'' = x\}.
$$
\end{sbpara}

\begin{sbpara}\label{togr} For a parabolic subgroup $P$ of $\PGL_V$, let $(V_i)_{-1 \le i \le m}$ be the corresponding flag.  Set
$$
	\bar X_{F,S}(P)=\{(P', \mu)\in \bar X_{F, S} \mid P'\supset P\}.
$$ 
For a place $v$ of $F$, let us set
\begin{eqnarray*}
	\frak{Z}_{F,v}(P) = \prod_{i=0}^m X_{(V_i/V_{i-1})_v} &\text{and}&
	\frak{Z}_{F,S}(P) = \prod_{v \in S} \frak{Z}_{F,v}(P).
\end{eqnarray*}
We let $P(F_v)$ act on $\frak{Z}_{F,v}(P)$ through $P(F_v)/P_u(F_v)$, using the 
$\PGL_{(V_i/V_{i-1})_v}(F_v)$-action on $X_{(V_i/V_{i-1})_v}$ for $0 \le i \le m$.
We define a $P(F_v)$-equivariant map 
$$
	\phi_{P,v} \colon \bar X_{F,v}(P)\to \frak{Z}_{F,v}(P)
$$ 
with the product of these over $v \in S$ giving rise to a map $\phi_{P,S} \colon \bar{X}_{F,S}(P)\to \frak{Z}_{F,S}(P)$. 

Let $(P', \mu) \in \bar X_{F,v}(P)$. Then the spaces in the flag 
$0=V'_{-1}\subsetneq V'_0\subsetneq \dots \subsetneq V'_{m'}=V$ corresponding to $P'$ form a 
subset of $\{V_i \mid -1 \le i \le m\}$. The image $\nu=(\nu_i)_{0\leq i\leq m}$ of $(P',\mu)$ under $\phi_{P,v}$ is as follows: there is a unique $j$ with $0\leq j\leq m'$ such that 
$$
	V'_j\supset V_i \supsetneq V_{i-1}\supset V'_{j-1},
$$ 
and $\nu_i$ is the norm 
induced from $\mu_j$ on the subquotient $(V_i/V_{i-1})_v$ of $(V'_j/V'_{j-1})_v$, in the sense of \ref{subq}.
The $P(F_v)$-equivariance of $\phi_{P,v}$ is easily seen using the actions on norms of \ref{actn} and \ref{actdual}. 

\end{sbpara}

Though the following map is not used in this subsection, we introduce it here by way of comparison between $\bar X_{F,S}$ and $\bar X_{F,S}^{\flat}$.

\begin{sbpara}\label{toW} 
Let $W$ be a nonzero $F$-subspace of $V$, and set
$$
	\bar X_{F,S}^{\flat}(W)=\{(W', \mu)\in \bar X_{F, S}^{\flat} \mid W'\supset W\}.
$$
For a place $v$ of $F$, we have a map
$$ \phi_{W,v}^{\flat} \colon \bar X_{F,v}^{\flat}(W) \to X_{W_v}$$
which sends $(W', \mu)\in \bar X_{F,v}^{\flat}(W)$ to the restriction of $\mu$ to $W_v$. 
The map $\phi_{W,v}^{\flat}$ is $P(F_v)$-equivariant, for $P$ the parabolic subgroup of $\PGL_V$ consisting of all elements that preserve $W$.   Setting $\frak{Z}_{F,S}^{\flat}(W) = \prod_{v \in S} X_{W_v}$, the product of these maps over $v \in S$ provides a map $\phi_{W,S}^{\flat} \colon \bar X_{F,S}^{\flat}(W) \to \frak{Z}_{F,S}^{\flat}(W)$.  

\end{sbpara}

\begin{sbpara} \label{ndef} For a finite-dimensional vector space $H$ over a local field $E$, a basis $e = (e_i)_{1 \le i \le d}$ of $H$, and a norm $\mu$ on $H$, we define the absolute value $|\mu: e|\in \R_{>0}$ of $\mu$ relative to $e$ as follows.  Suppose that $\mu$ is defined by a basis $e' = (e'_i)_{1\leq i\leq d}$ and a tuple $(r_i)_{1\leq i\leq d}\in \R_{>0}^d$.  Let $h \in \GL_H(E)$ be the element such that $e'=he$. We then define 
$$
	|\mu:e| = |\det(h)|^{-1} \prod_{i=1}^d r_i.
$$ 
This is independent of the choice of $e'$ and $(r_i)_i$.  Note that we have 
$$
	|g\mu:e|= |\det(g)|^{-1}|\mu: e|
$$ 
for all $g\in  \GL_H(E)$.

\end{sbpara}

\begin{sbpara}\label{togr2} Let $P$ and $(V_i)_i$  be as in \ref{togr}, and let $v$ be a place of $F$. 
Fix a basis $e^{(i)}$ of $(V_i/V_{i-1})_v$ for each $0\leq i\leq m$. 
Then we have a map
$$ \phi'_{P,v} \colon \bar X_{F,v}(P)\to \R_{\geq 0}^m, \qquad (P',\mu) \mapsto (t_i)_{1 \leq i \leq m}$$
where $(t_i)_{1\leq i \leq m}$ is defined as follows. Let $(V'_j)_{-1\leq j\leq m'}$ be the flag associated to $P'$. 
Let $1\leq i\leq m$. If $V_{i-1}$ belongs to $(V_j')_j$, let $t_i=0$. If $V_{i-1}$ does not belong to the last flag, then there is a unique $j$ such that $V'_j \supset V_i \supset V_{i-2}\supset V'_{j-1}$. Let $\tilde \mu_j$ be a norm on $(V'_j/V'_{j-1})_v$ which belongs to the class $\mu_j$, and let $\tilde \mu_{j,i}$ and $\tilde\mu_{j,i-1}$ be the norms induced by $\mu_j$ on the subquotients $(V_i/V_{i-1})_v$ and $(V_{i-1}/V_{i-2})_v$, respectively.  We then let
$$t_i = |\tilde \mu_{j,i-1}:e^{(i-1)}|^{1/d_{i-1}}\cdot  |\tilde \mu_{j,i}:e^{(i)}|^{-1/d_i},$$
where $d_i:=\dim(V_i/V_{i-1})$.

The map $\phi'_{P,v}$ is $P(F_v)$-equivariant for the following action of $P(F_v)$ on $\R_{\geq 0}^m$.  For $g\in P(F_v)$, let $\tilde g \in \GL_V(F_v)$ be a lift of $g$, and for $0\leq i\leq m$, let $g_i\in \GL_{V_i/V_{i-1}}(F_v)$ be the element induced by $\tilde g$. Then $g \in P(F_v)$ sends $t\in \R^m_{\geq 0}$ to $t'\in \R_{\geq 0}^m$ where
$$
	t'_i = |\det(g_i)|^{1/d_i}\cdot |\det(g_{i-1})|^{-1/d_{i-1}}\cdot t_i.
$$

If we have two families $e=(e^{(i)})_i$ and $f=(f^{(i)})_i$ of bases $e^{(i)}$ and $f^{(i)}$ of $(V_i/V_{i-1})_v$, and if 
the map $\phi'_{P,v}$ defined by $e$ (resp., $f$) sends an element to $t$ (resp., $t'$), then the same formula also describes the relationship between $t$ and $t'$, in this case taking $g_i$ to be the element of $\GL_{V_i/V_{i-1}}$ such that $e^{(i)}=g_if^{(i)}$. 
\end{sbpara}

\begin{sbpara}\label{togr3}  
Fix a basis $e^{(i)}$ of $V_i/V_{i-1}$ for each $0\leq i\leq m$. 
Then we have a map
$$ \phi'_{P,S} \colon \bar X_{F,S}(P)\to \R_{\geq 0}^m, \qquad (P', \mu) \mapsto (t_i)_{1\leq i\leq m}$$
where $t_i=\prod_{v\in S} t_{v,i}$, with $(t_{v,i})_i$ the image of $(P',\mu_v)$ under the map $\phi'_{P,v}$ of \ref{togr2}.

\end{sbpara}

\begin{sbpara}\label{Sat1} We define the Satake topology on $\bar X_{F,S}$ as follows. 

For a parabolic subgroup $P$ of $\PGL_V$, consider the map 
$$  \psi_{P,S} := (\phi_{P,S},\phi'_{P,S}) \colon \bar X_{F,S}(P) \to \frak{Z}_{F,S}(P) \times \R_{\geq 0}^m $$ 
from \ref{togr} and \ref{togr3}, which we recall depends on a choice of bases of the $V_i/V_{i-1}$.
We say that a subset of $\bar X_{F,S}(P)$ is $P$-open if it is the  
 inverse image of an open subset of 
$\frak{Z}_{F,S}(P) \times \R_{\geq 0}^m$.  By \ref{togr2}, the property of being $P$-open is independent
of the choice of bases.

We define the Satake topology on $\bar X_{F,S}$ to be the coarsest topology for which every $P$-open set for each parabolic subgroup $P$ of $\PGL_V$ is open. 

\end{sbpara}

By this definition, we have:
\begin{sbpara}\label{Sat5}  
Let $a \in \bar X_{F,S}$ be of the form $(P,\mu)$ for some $\mu$.  As $U$ ranges over neighborhoods of the image $(\mu,0)$ of $a$ in $\frak{Z}_{F,S}(P) \times \R^m_{\geq 0}$, the inverse images of the $U$ in $\bar X_{F,S}(P)$ under $\psi_{P,S}$ form a base of neighborhoods of $a$ in $\bar X_{F,S}$.  

\end{sbpara}

\begin{sbpara}
In \S\ref{propX} and \S\ref{comptop}, we explain that the Satake topology on $\bar X_{F,S}$ is strictly coarser than the Borel-Serre topology for $d \ge 2$. 
\end{sbpara}

\begin{sbpara} The Satake topology on $\bar X_{F,S}$ can differ from the subspace topology of the product topology for the Satake topology on each $\bar X_{F,v}$ with $v \in S$.

\begin{example} Let $F$ be a real quadratic field, let $V=F^2$, and let $S=\{v_1,v_2\}$ be  the set of real places of $F$. 
Consider the point $(\infty, \infty) \in (\frak H\cup \{\infty\})\times (\frak H\cup\{\infty\})\subset \bar X_{v_1}\times \bar X_{v_2}$ (see \S\ref{uphp}), which we regard as an element of 
$\bar X_{F,S}$. Then the sets 
$$
	U_c:=\{(x_1+y_1i, x_2+y_2i)\in \frak H \times \frak H \mid y_1y_2 \ge c\}\cup\{(\infty,\infty)\}
$$ 
with $c\in \R_{>0}$ form a base of neighborhoods of $(\infty,\infty)$ in $\bar X_{F,S}$ for the Satake topology, whereas the sets 
$$
	U'_c:=\{(x_1+y_1i, x_2+y_2i)\in \frak H \times \frak H \mid y_1 \ge c, y_2 \ge c\}\cup\{(\infty,\infty)\}
$$ 
for $c\in \R_{>0}$ form a base of neighborhoods of $(\infty, \infty)$ in $\bar X_{F,S}$ for the topology induced by the product of Satake topologies on $\bar X_{F,v_1}$ and $\bar X_{F,v_2}$.
\end{example}
\end{sbpara}

\begin{sbpara}\label{top7} 
Let $G=\PGL_V$, and let $\Gamma$ be a subgroup of $G(F)$.
\begin{itemize}
\item For a parabolic subgroup $P$ of $G$, let $\Gamma_{(P)}$ be the subgroup of $\Gamma \cap P(F)$ consisting of all elements with image in the center of $(P/P_u)(F)$. 
\item For a nonzero $F$-subspace $W$ of $V$, let $\Gamma_{(W)}$ denote the subgroup of elements of $\Gamma$ that can be lifted to elements of $\GL_V(F)$ which fix every element of $W$.
\end{itemize}
\end{sbpara}

\begin{sbpara}\label{top8} 
We let $\mb{A}_F$ denote the adeles of $F$, let $\mb{A}_F^S$ denote the adeles of $F$ outside of $S$, and let $\mb{A}_{F,S} = \prod_{v \in S} F_v$ so that
$\mb{A}_F = \mb{A}_F^S \times \mb{A}_{F,S}$.
Assume that $S$ contains all archimedean places of $F$. Let $G=\PGL_V$, 
let $K$ be a compact open subgroup of $G(\mb{A}^S_F)$, and let $\Gamma_K < G(F)$ be the inverse image of $K$ under $G(F)\to G(\mb{A}^S_F)$. 

\end{sbpara}

The following proposition will be proved in \ref{top3}.
\begin{sbprop}\label{top1} 
For $S$, $G$, $K$ and $\Gamma_K$ as in \ref{top8},  
the Satake topology on $\bar X_{F,S}$ is the coarsest topology such that for every parabolic subgroup $P$ of $G$,
a subset  $U$ of $\bar X_{F,S}(P)$ is open if
\begin{enumerate}
\item[(i)] it is open for  Borel-Serre topology, and
\item[(ii)] it is stable under the action of $\Gamma_{K, (P)}$ (see \ref{top7}).  
\end{enumerate}
\end{sbprop}

The following proposition follows easily from the fact that for any two compact open subgroups $K$ and $K'$ of $G(\mb{A}_F^S)$, the intersection $\Gamma_K\cap \Gamma_{K'}$ is of finite index in both $\Gamma_K$ and $\Gamma_{K'}$.

\begin{sbprop}\label{top2} 
For $S$, $K$ and $\Gamma_K$ as in \ref{top8},
consider the coarsest topology on $\bar X_{F,S}^{\flat}$ such that for every nonzero $F$-subspace $W$,
a subset $U$ of $\bar X_{F,S}^{\flat}(W)$ is open if
\begin{enumerate}
\item[(i)] it is open for  Borel-Serre topology, and
\item[(ii)] it is stable under the action of $\Gamma_{K,(W)}$ (see \ref{top7}). 
\end{enumerate}
Then this topology is independent of the choice of $K$. 
\end{sbprop}

\begin{sbpara}\label{assu5} 
We call the topology in \ref{top2} the Satake topology on $\bar X_{F,S}^{\flat}$. 
\end{sbpara}

\begin{sbprop}\label{top5} \ 
\begin{enumerate}
\item[(1)] Let $P$ be a parabolic subgroup of $\PGL_V$.  For both the Borel-Serre and Satake topologies on $\bar X_{F,S}$, the set $\bar X_{F,S}(P)$ is open in 
$\bar X_{F,S}$, and the action of the topological group $P(\mb{A}_{F,S})$ on $\bar X_{F,S}(P)$ is continuous. 
\item[(2)] The actions of the discrete group $\PGL_V(F)$ on the following spaces are continuous: $\bar X_{F,S}$ and $\bar X_{F,S}^{\flat}$ with their Borel-Serre topologies, $\bar X_{F,S}$ with the Satake topology, and assuming $S$ contains all archimedean places, $\bar X_{F,S}^{\flat}$ with the Satake topology.
\end{enumerate}
\end{sbprop}

\begin{sbprop}\label{top6}  Let $W$ be a nonzero $F$-subspace of $V$. Then for the Borel-Serre topology, and for the Satake topology if $S$ contains all archimedean places of $F$, the subset $\bar X_{F,S}^{\flat}(W)$ is open in 
$\bar X_{F,S}^{\flat}$.
\end{sbprop} 

Part (1) of \ref{top5} was shown in \S\ref{BSsec} for the Borel-Serre topology, and the result for the Satake topology on $\bar X_{F,S}$ follows from it. The rest of \ref{top5} and \ref{top6} is easily proven. 

\subsection{Properties of $\bar X_{F,S}$} \label{propX}

Let $S$ be a nonempty finite set of places of $F$. 

\begin{sbpara}\label{togr9} Let $P$ and $(V_i)_{-1\leq i \leq m}$ be as before. Fix a basis $e^{(i)}$ of $V_i/V_{i-1}$ for each $i$. 
Set
$$
	Y_0 = ( \R_{>0}^S \cup \{(0)_{v \in S}\} )^m \subset (\R_{\geq 0}^S)^m.
$$
The maps $\psi_{P,v} := (\phi_{P,v},\phi'_{P,v}) \colon \bar X_{F,v}(P) \to \frak{Z}_{F,v}(P) \times \R^m_{\geq 0}$ of \ref{togr} and \ref{togr2} 
for $v \in S$ combine to give the map
$$ \psi_{P,S} \colon \bar X_{F,S}(P)\to \frak{Z}_{F,S}(P) \times Y_0.$$

\end{sbpara}

\begin{sbpara} \label{weaktopY}
In addition to the usual topology on $Y_0$, we consider the weak topology on $Y_0$ that is the product
topology for the  topology on
$\R_{>0}^S\cup \{(0)_{v \in S}\}$ which extends the usual topology on
$\R_{>0}^S$ by taking the sets 
$$
	\Bigg\{(t_v)_{v\in S}\in \R_{>0}^S \mid \prod_{v\in S} t_v \le c\Bigg\}\cup \{{(0)_{v \in S}}\}
$$ 
for $c\in \R_{>0}$
as a base of neighborhoods of $(0)_{v \in S}$.
In the case that $S$ consists of a single place, we have $Y_0=\R_{\geq 0}^m$, and the natural topology and the weak topology on $Y_0$ coincide. 
\end{sbpara}

\begin{sbprop}\label{keyc1} The map $\psi_{P,S}$ of \ref{togr9} induces a homeomorphism 
$$
	P_u(\mb{A}_{F,S})\bs \bar X_{F,S}(P) \xrightarrow{\sim} \frak{Z}_{F,S}(P) \times Y_0
$$
for the Borel-Serre topology (resp., Satake topology) on $\bar X_{F,S}$ and the usual (resp., weak) topology on $Y_0$. 
This homeomorphism is equivariant for the action of $P(\mb{A}_{F,S})$,  with the action of $P(\mb{A}_{F,S})$ on $Y_0$ being that of  \ref{togr2}. 
\end{sbprop}

This has the following corollary, which is also the main step in the proof.

\begin{sbcor} \label{keycor}
	For any place $v$ of $F$, the map 
	$$P_u(F_v)\bs \bar X_{F,v}(P) \to \frak{Z}_{F,v}(P) \times \R_{\geq 0}^m$$
	is a homeomorphism for both the Borel-Serre and Satake topologies on $\bar X_{F,v}$.
\end{sbcor}

We state and prove preliminary results towards the proof of \ref{keyc1}. 

\begin{sbpara} Fix a basis $(e_i)_i$ of $V$
and a parabolic subgroup $P$ of $\PGL_V$ which satisfies the condition in \ref{Pstr1} for this basis.
Let $(V_i)_{-1 \le i \le m}$ be the flag corresponding to $P$, and for each $i$, set $c(i)= \dim(V_i)$. 
We define two maps
$$\xi, \xi^{\star} \colon P_u(F_v) \times \frak{Z}_{F,v}(P) \times \R^m_{\geq 0} \to \bar X_{F,v}(P).$$

\end{sbpara}

\begin{sbpara} \label{xidef}
First, we define the map $\xi$.

Set $\Delta(P)= \{c(0), \dots, c(m-1)\}$. Let $\Delta_i=\{j\in \Z\mid c(i-1)<j<c(i)\}$ for $0 \le i \le m$.
We then clearly have 
$$\{1,\dots,d-1\}= \Delta(P) \amalg \Bigg(\coprod_{i=0}^m \Delta_i\Bigg).$$
For $0\leq i\leq m$, let $V^{(i)}=\sum_{c(i-1)<j\leq c(i)} F e_j$, so $V_i = V_{i-1}\oplus V^{(i)}$. 
We have 
$$\R^{d-1}_{\geq 0}(P) = \R_{\geq 0}^{\Delta(P)} \times \prod_{i=0}^m \R^{\Delta_i}_{> 0} \cong \R_{\geq 0}^m \times \prod_{i=0}^m \R^{\Delta_i}_{> 0}.$$

Let $B$ be the Borel subgroup of $\PGL_V$ consisting of all upper triangular matrices for the basis $(e_i)_i$. Fix a place $v$ of $F$. We consider two surjections
\begin{eqnarray*} 
B_u(F_v) \times \R_{\geq 0}^{d-1}(P) \twoheadrightarrow \bar X_{F,v}(P) &\text{and}&
B_u(F_v) \times \R_{\geq 0}^{d-1}(P) \twoheadrightarrow P_u(F_v) \times \frak{Z}_{F,v}(P) \times \R^m_{\geq 0}.
\end{eqnarray*}
The first is induced by the surjection $\bar{\pi}_v$ of \ref{lst0}. 

The second map is obtained as follows.  For $0\leq i\leq m$, let $B_i$ be the image of $B$ in $\PGL_{V^{(i)}}$ under $P\to \PGL_{V_i/V_{i-1}}\cong \PGL_{V^{(i)}}$.  Then $B_i$ is a Borel subgroup of $\PGL_{V^{(i)}}$, and we have a canonical
bijection
$$
	P_u(F_v) \times \prod_{i=0}^m B_{i,u}(F_v) \xrightarrow{\sim} B_u(F_v).
$$ 
By \ref{lst0}, we have surjections $B_{i,u}(F_v) \times \R^{\Delta_i}_{> 0} \twoheadrightarrow X_{(V_i/V_{i-1})_v}$ for $0\leq i\leq m$.  The second (continuous) surjection is then the composite
\begin{align*}
	B_u(F_v) \times \R_{\ge 0}^{d-1}(P) &\xrightarrow{\sim} \Bigg(P_u(F_v) \times \prod_{i=0}^m B_{i,u}(F_v)\Bigg)
	\times \Bigg(\R_{\geq 0}^m \times  \prod_{i=0}^m \R_{>0}^{\Delta_i} \Bigg) \\
	&\twoheadrightarrow P_u(F_v) \times \Bigg( \prod_{i=0}^m X_{(V_i/V_{i-1})_v} \Bigg) \times \R_{\geq 0}^m
	= P_u(F_v) \times \frak{Z}_{F,v}(P) \times \R_{\geq 0}^m.
\end{align*}
\end{sbpara}

\begin{sbprop}\label{keyp}  
There is a unique surjective continuous map 
$$\xi \colon P_u(F_v) \times \frak{Z}_{F,v}(P) \times \R_{\geq 0}^m \twoheadrightarrow \bar X_{F,v}(P)$$
for the Borel-Serre topology on $\bar X_{F,v}(P)$
that is compatible with the surjections from $B_u(F_v) \times \R_{\geq 0}^{d-1}(P)$ to these sets. 
This map induces a homeomorphism
$$
	\frak{Z}_{F,v}(P) \times \R_{\geq 0}^m \xrightarrow{\sim} P_u(F_v)\bs \bar X_{F,v}(P)
$$
that restricts to a homeomorphism of $\frak{Z}_{F,v}(P) \times \R_{>0}^m$ with
$P_u(F_v)\bs  X_v$.
\end{sbprop}

This follows from \ref{lst0}. 

\begin{sbpara}\label{xistar1} Next, we define the map $\xi^{\star}$. 

For $g\in P_u(F_v)$, $(\mu_i)_i \in (X_{(V_i/V_{i-1})_v})_{0 \le i \le m}$, and $(t_i)_{1\leq i\leq m}\in \R_{\ge 0}^m$,
we let
$$
	\xi^{\star}(g, (\mu_i)_i, (t_i)_i)= g(P', \nu),
$$ 
where $P'$ and $\nu$ are as in (1) and (2) below, respectively. 
\begin{enumerate}
\item[(1)] Let $J=\{c(i-1) \mid 1\leq i\leq m,\ t_i=0\}$. 
Write $J=\{c'(0), \dots, c'(m'-1)\}$ with $c'(0)<\dots < c'(m'-1)$. Let $c'(-1)=0$ and $c'(m')=d$. For $-1\leq i\leq m'$, let 
$$
	V'_i= \sum_{j=1}^{c'(i)} F e_j\subset V.
$$ 
Let $P'\supset P$ be the parabolic subgroup of $\PGL_V$ corresponding to the flag $(V'_i)_i$. 
\item[(2)]  
For $0\leq i\leq m'$, set 
$$
	J_i = \{ j \mid c'(i-1) < c(j) \le c'(i) \} \subset \{1,\dots, m\}. 
$$
We identify $V'_i/V'_{i-1}$ with $\bigoplus_{j \in J_i} V^{(j)}$ via the basis $(e_k)_{c'(i-1)<k\le c'(i)}$.   
We define  a norm $\tilde \nu_i$ on $V_i'/V'_{i-1}$ as follows. 
Let $\tilde \mu_j$ be the unique norm on $V^{(j)}$ which belongs to $\mu_j$ and satisfies $|\tilde \mu_j:(e_k)_{c(j-1)<k\leq c(j)}|=1$. For 
$x=\sum_{j \in J_i} x_j$ 
with $x_j\in V^{(j)}$, set
$$
	\tilde{\nu}_i(x) = \begin{cases} \sum_{j \in J_i} (r_j^2\tilde \mu_j(x_j)^2)^{1/2} & \text{if } v \text{ is real},\\
	\sum_{j \in J_i} r_j\tilde \mu_j(x_j) &\text{if } v \text{ is complex}, \\
	\max_{j \in J_i}(r_j\tilde \mu_j(x_j)) &\text{if } v \text{ is non-archimedean}, \end{cases}
$$
where for $j \in J_i$, we set
$$
	r_j= \prod_{\substack{\ell \in J_i \\ \ell < j}} t_{\ell}^{-1}.
$$ 
Let $\nu_i\in X_{(V'_i/V'_{i-1})_v}$ be the class of the norm $\tilde \nu_i$.
\end{enumerate}
\end{sbpara}

We omit the proofs of the following two lemmas.

\begin{sblem}\label{xi2}  The composition 
$$P_u(F_v) \times \frak{Z}_{F,v}(P) \times \R^m_{\geq 0} \xrightarrow{\xi^{\star}}  \bar X_{F,v}(P) 
\xrightarrow{\psi_{P,v}} \frak{Z}_{F,v}(P) \times \R^m_{\geq 0}$$
coincides with the canonical projection. Here, the definition of the second arrow uses the basis $(e_j\bmod V_{i-1})_{c(i-1)<j\leq c(i)}$  of $V_i/V_{i-1}$. 

\end{sblem}

\begin{sblem}\label{xixi} We have a commutative diagram
%$$
%	\SelectTips{cm}{} \xymatrix{ 
%	P_u(F_v) \times \frak{Z}_{F,v}(P) \times \R^m_{\geq 0} \ar[r]^-{\xi} \ar[d]^{\wr} & \bar X_{F,v}(P) \ar@{=}[d] \\
%	P_u(F_v) \times \frak{Z}_{F,v}(P) \times \R^m_{\geq 0} \ar[r]^-{\xi^{\star}} & \bar X_{F,v}(P) }
%$$
$$\begin{tikzcd}[column sep = large, row sep = small] P_u(F_v) \times \frak{Z}_{F,v}(P) \times \R^m_{\geq 0} \rar{\xi} \dar & \bar X_{F,v}(P) \arrow[equals]{d} \\
P_u(F_v) \times \frak{Z}_{F,v}(P) \times \R^m_{\geq 0} \rar{\xi^{\star}} & \bar X_{F,v}(P) 
\end{tikzcd}
$$
in which the left vertical arrow is 
$(u, \mu, t) \mapsto (u, \mu, t')$, for $t'$ defined as follows. 
 Let 
$I_i \colon X_{(V_i/V_{i-1})_v} \to \R_{>0}^{\Delta_i}$ be the unique continuous map for which the composition  
$$
	B_{i,u}(F_v) \times \R^{\Delta_i}_{>0}\to X_{(V_i/V_{i-1})_v} \xrightarrow{I_i} \R_{>0}^{\Delta_i}
$$ 
is projection onto the second factor, and for $j \in \Delta_i$, let $I_{i,j} \colon X_{(V_i/V_{i-1})_v} \to \R_{>0}$
denote the composition of $I_i$ with projection onto the factor of $\R_{>0}^{\Delta_i}$ corresponding to $j$. Then 
$$t'_i = t_i \cdot \prod_{j\in \Delta_{i-1}}  I_{i-1,j}(\mu_i)^{\frac{j-c(i-2)}{c(i-1)-c(i-2)}}\cdot \prod_{j\in \Delta_i} I_{i,j}(\mu_i)^{\frac{c(i)-j}{c(i)-c(i-1)}}$$ 
for $1\leq i\leq m$.

\end{sblem}

\begin{sbpara} Proposition \ref{keyc1} is quickly reduced to Corollary \ref{keycor}, which now follows from \ref{keyp}, \ref{xi2} and \ref{xixi}.

\end{sbpara}

\begin{sbpara}\label{BSs} 
For two topologies $T_1$, $T_2$ on a set $Z$, we use $T_1\geq T_2$ to denote that the identity map of $Z$ is a continuous map from $Z$ with $T_1$ to $Z$ with $T_2$, and $T_1>T_2$ to denote that $T_1\geq T_2$ and $T_1\neq T_2$.  In other words, $T_1 \ge T_2$ if $T_1$ is finer than $T_2$ and $T_1 > T_2$ if $T_1$ is strictly finer than $T_2$.

By \ref{keyc1}, the map $\psi_{P,S} \colon \bar X_{F,S}(P)\to \frak{Z}_{F,S}(P) \times Y_0$ is continuous for the Borel-Serre topology on $\bar{X}_{F,S}$ and usual topology on $Y_0$.  On $\bar{X}_{F,S}$, we therefore have 
$$
\text{Borel-Serre topology} \geq \text{Satake topology}.
$$

\end{sbpara}
\begin{sbcor}\label{togrW}  For any nonempty finite set $S$ of places of $F$, the 
map $\phi_{W,S}^{\flat} \colon \bar X_{F,S}^{\flat}(W) \to \frak{Z}_{F,S}^{\flat}(W)$ of \ref{toW} is continuous for the Borel-Serre  topology on $\bar X_{F,S}^{\flat}$. If $S$ contains all archimedean places of $F$, it is continuous for the Satake topology. 
\end{sbcor}

\begin{pf} The continuity for the Borel-Serre topology follows from the continuity of $\psi_{P,S}$, noting that the Borel-Serre topology on $\bar X_{F,S}^{\flat}$ is the quotient topology of the Borel-Serre topology on $\bar X_{F,S}$.  Suppose that $S$ contains all archimedean places.  As $\phi_{W,S}^{\flat}$ is $\Gamma_{K, (W)}$-equivariant, and $\Gamma_{K,(W)}$ acts trivially on $\frak{Z}_{F,S}^{\flat}(W)$, the continuity for the Satake topology is reduced to the continuity for the Borel-Serre topology.   
\end{pf}
\begin{sbrem}

We remark that the map $\phi_{P,v} \colon \bar X_{F,v}(P)\to \frak{Z}_{F,v}(P)$ of \ref{togr} 
need not be continuous for the  topology on $\bar X_{F,v}$ as a subspace of $\bar X_v$.   Similarly, the map $\phi_{W,v}^{\flat} \colon \bar X_{F,v}^{\flat}(W) \to X_{W_v}$ of \ref{toW} need not be continuous for the subspace
topology on $\bar X_{F,v}^{\flat} \subset \bar X_v^{\flat}$.
See \ref{toW2} and \ref{phiweak}. 

\end{sbrem}

\begin{sbpara}\label{top3} We prove Proposition \ref{top1}. 
\begin{proof}

Let $\alpha=(P, \mu)\in \bar X_{F,S}$. Let $U$ be a  neighborhood of $\alpha$ for the Borel-Serre topology which is stable under the action of $\Gamma_{K,(P)}$.  By \ref{BSs}, it is sufficient to prove that there is a neighborhood $W$ of $\alpha$ for the Satake topology such that $W \subset U$. 

Let $(V_i)_{-1 \le i \le m}$ be the flag corresponding to $P$, and let $V^{(i)}$ be as before.
Let $\Gamma_1= \Gamma_K \cap P_u(F)$, and let $\Gamma_0$ be the subgroup of $\Gamma_K$ consisting of the elements that preserve $V^{(i)}$ and act on $V^{(i)}$ as a scalar for all $i$. Then $\Gamma_1$ is a normal subgroup of $\Gamma_{K,(P)}$ and $\Gamma_1\Gamma_0$  is a subgroup of $\Gamma_{K, (P)}$ of finite index.

Let 
$$
	Y_1 = \left\{(a_v)_{v\in S}\in \R^S_{>0} \mid \prod_{v\in S} a_v=1\right\}^m,
$$ 
and set $s = \sharp S$.  We have a surjective continuous map
\begin{eqnarray*}
	\R_{\geq 0}^m \times Y_1 \twoheadrightarrow Y_0, \qquad (t,t') \mapsto (t_i^{1/s}t'_{v,i})_{v,i}.
\end{eqnarray*}
The composition $\R_{\geq 0}^m \times Y_1 \to Y_0 \to \R_{\geq 0}^m$, where the second arrow is $(t_{v,i})_{v,i} \mapsto (\prod_{v\in S} t_{v,i})_i$, coincides with projection onto the first coordinate.

Let 
\begin{eqnarray*}
	\Phi =  P_u(\mb{A}_{F,S}) \times Y_1
	&\text{and}&
	\Psi= \frak{Z}_{F,S}(P) \times \R^m_{\geq 0}.
\end{eqnarray*}
Consider the composite map
$$f \colon \Phi \times \Psi \to P_u(\mb{A}_{F,S}) \times \frak{Z}_{F,S}(P) \times Y_0 \xrightarrow{(\xi_v^{\star})_{v \in S}} \bar X_{F,S}(P).$$
The map $f$ is $\Gamma_1\Gamma_0$-equivariant for the trivial action on $\Psi$ and the following action on $\Phi$: for $(g, t)\in \Phi$, $\gamma_1\in \Gamma_1$ and $\gamma_0\in \Gamma_0$, we have
$$
	\gamma_1\gamma_0 \cdot (g,t) = (\gamma_1\gamma_0 g \gamma_0^{-1}, \gamma_0t),
$$ 
where $\gamma_0$ acts on $Y_1$ via the embedding $\Gamma_K \to P(\mb{A}_{F,S})$ and
the actions of the $P(F_v)$ described in \ref{togr2}. 
The composition 
$$
	\Phi\times \Psi \xrightarrow{f} \bar X_{F,S}(P) \xrightarrow{\psi_{P,S}} \Psi
$$ 
coincides with the canonical projection. 

There exists a compact subset $C$ of $\Phi$ such that $\Phi= \Gamma_1\Gamma_0C$ for the above action of $\Gamma_1\Gamma_0$ on $\Phi$. 
Let $\beta = (\mu, {0})\in \Psi$ be the image of $\alpha$ under $\psi_{P,S}$.
For $x\in \Phi$, we have $f(x, \beta)=\alpha$.  Hence, there 
 is an open neighborhood $U'(x)$ of $x$ in $\Phi$ and an open neighborhood $U''(x)$ of $\beta$ in $\Psi$ 
such that $U'(x) \times U''(x)\subset f^{-1}(U)$. Since $C$ is compact, 
 there is a finite subset $R$ of $C$ such that $C \subset \bigcup_{x\in R} U'(x)$. 
 Let $U''$ be the open subset $\bigcap_{x\in R} U''(x)$ of $\Psi$, which contains $\beta$. The $P$-open set $W = \psi_{P,S}^{-1}(U'') \subset 
 \bar X_{F,S}(P)$ is by definition an open neighborhood of $\alpha$ in the Satake topology on $\bar X_{F,S}$. 
 We show that $W \subset f^{-1}(U)$. Since the map $\Phi \times \Psi\to \bar X_{F,S}(P)$ is surjective, 
 it is sufficient to prove that the inverse image $\Phi \times U''$ of $W$ in $\Phi \times \Psi$ 
 is contained $f^{-1}(U)$.  For this, we note that
 $$
 	\Phi \times U'' = \Gamma_1\Gamma_0 C \times U'' = \Gamma_1 \Gamma_0\left( \bigcup_{x\in R} U'(x) \times U''\right)\subset \Gamma_1\Gamma_0 f^{-1}(U) = f^{-1}(U),
$$ 
the last equality by the stability of $U$ under the action of $\Gamma_{K,(P)}\supset \Gamma_1\Gamma_0$ and
the $\Gamma_1\Gamma_0$-equivariance of $f$. 
\end{proof}

\end{sbpara}

\begin{sbpara}\label{d=2} In the case $d=2$, the canonical surjection $\bar X_{F,S}\to \bar X_{F,S}^{\flat}$ is bijective. It is a homeomorphism for the Borel-Serre topology. If $S$ contains all archimedean places of $F$, it is a homeomorphism for the Satake topology by \ref{top1}. 

\end{sbpara}

\subsection{Comparison of the topologies} \label{comptop}

When considering $\bar X_{F,v}^{\flat}$, we assume that all places of $F$ other than $v$ are non-archimedean. 

\begin{sbpara} For $\bar X_{F,v}$ (resp., $\bar X_{F,v}^{\flat}$), we have introduced several topologies:
the Borel-Serre topology, the Satake topology, and the subspace topology from $\bar X_v$ (resp., $\bar X_v^{\flat}$), which we call the weak topology.  We compare these topologies below; note that we clearly have
\begin{eqnarray*} 
	\text{Borel-Serre topology} \geq \text{Satake topology} &\text{and}& 
	\text{Borel-Serre topology}\geq \text{weak topology}.
\end{eqnarray*}
\end{sbpara}

\begin{sbpara}\label{compa} For both $\bar X_{F,v}$ and $\bar X_{F,v}^{\flat}$, the following hold:
\begin{enumerate}
\item[(1)] Borel-Serre topology $>$ Satake topology if $d \ge 2$,
\item[(2)] Satake topology $>$ weak topology if $d=2$,
\item[(3)] Satake topology $\not\geq$ weak topology if $d>2$. 
\end{enumerate}

We do not give full proofs of these statements.  Instead, we 
 describe some special cases that give clear pictures of the differences between these topologies. The general cases can be proven in a similar manner to these special cases.

Recall from \ref{d=2} that 
in the case $d=2$, the sets $\bar X_{F,v}$ and $\bar X_{F,v}^{\flat}$ are equal, their Borel-Serre topologies coincide, and their Satake topologies coincide. 

\end{sbpara}

\begin{sbpara}\label{BSSa} We describe the case $d=2$ of  \ref{compa}(1).

Take a basis $(e_i)_{i=1,2}$ of $V$. Consider the point $\alpha=(B, \mu)$ of $\bar X_{F,v}$, where $B$ is the Borel subgroup of upper triangular matrices with respect to $(e_i)_i$, and $\mu$ is the unique element of $\frak{Z}_{F,v}(B) = X_{F_ve_1} \times X_{V_v/F_ve_1}$.

Let $\bar{\pi}_v$ be the 
surjection of \ref{lst0}(2), and identify $B_u(F_v)$ with $F_v$ in the canonical manner.  
The images of the sets
$$
	\{(x, t) \in F_v\times \R_{\geq 0}\mid t \le c\}\subset B_u(F_v) \times \R_{\geq 0}
$$ 
in $\bar X_{F,v}(B)$ for $c \in \R_{>0}$ form a base of neighborhoods of $\alpha$ for the Satake topology. 
Thus, while the image of the set 
$$
	\{(x, t) \in F_v\times \R_{\geq 0} \mid t < |x|^{-1}\}
$$ 
is a neighborhood of $\alpha$ for  Borel-Serre topology, it is not a neighborhood of $\alpha$ for the Satake topology.  
\end{sbpara}

\begin{sbpara}\label{BSSb} We prove \ref{compa}(2) in the case that $v$ is non-archimedean. The proof in the archimedean case is similar.  
Since all boundary points of $\bar X_{F,v} = \bar X_{F,v}^{\flat}$ are $\PGL_V(F)$-conjugate, to show \ref{compa}(2), it is sufficient to consider any one boundary point.  We consider $\alpha$ of \ref{BSSa} for a fixed basis  $(e_i)_{i=1,2}$ of $V$.

For $x\in F_v$ and $y\in \R_{>0}$, let $\mu_{y,x}$ be the norm on $V_v$ defined by 
$$
	\mu_{y,x}(ae_1+be_2)= \max(|a-xb|, y|b|).
$$ 
The class of $\mu_{y,x}$ is the image of $(x,y^{-1}) \in B_u(F_v) \times \R_{> 0}$. Any element of $X_v$ is the class of the norm $\mu_{y,x}$ for some $x,y$. 
If we vary $x\in F_{\infty}$ and $y\in \R_{>0}$, the classes of $\mu_{y,x}$ in $\bar X_{F,v}$  converge 
under the Satake topology to the point $\alpha$ if and only if $y$ approaches $\infty$.
 In $\bar X_v$, the point
$\alpha$ is the class of the semi-norm $\nu$ on $V_v^*$ defined by $\nu(ae_1^*+be_2^*)= |a|$. 
By \ref{actdual}, 
 $$\mu_{y,x}^*= \left(\mu_{y,0}\circ \begin{pmatrix} 1 &-x \\ 0 & 1\end{pmatrix}\right)^*
 = \mu_{y,0}^* \circ \begin{pmatrix} 1 & 0\\ x & 1\end{pmatrix}, $$
from which we see that
$$
	\mu_{y, x}^*(ae_1^*+be_2^*)= \max(|a|, y^{-1}|xa+b|).
$$ 
Then $\mu_{y,x}^*$ is equivalent to the norm $\nu_{y,x}$
 on $V_v^*$ defined by 
 $$
 	\nu_{y,x}(ae_1^*+be_2^*) = \min(1, y|x|^{-1})
	\max(|a|, y^{-1}|xa+b|),
$$ 
and the classes of the $\nu_{y,x}$ converge in $\bar X_v$ to the class of the semi-norm $\nu$ as $y \to \infty$.  Therefore, the Satake topology is finer than the weak topology.

Now, the norm $\mu_{1,x}^*$ is equivalent to the norm $\nu_{1,x}$ on $V_v^*$ defined above, which for sufficiently
large $x$ satisfies
 $$
 	\nu_{1,x}(ae_1+be_2) = \max(|a/x|, |a+(b/x)|).
$$ 
Thus, as $|x|\to \infty$, the sequence $\mu_{1,x}$ converges in $\bar X_v= \bar X_v^{\flat}$ to the class of the semi-norm $\nu$.  However, the sequence of classes of the norms $\mu_{1,x}$ does not converge to $\alpha$ in $\bar X_{F,v}=\bar X_{F,v}^{\flat}$ for the Satake topology, so the Satake topology is strictly finer than the weak topology.  
\end{sbpara}

\begin{sbpara} We explain the case $d=3$ of \ref{compa}(3) in the non-archimedean case.

Take a basis $(e_i)_{1\leq i\leq 3}$ of $V$. For $y\in \R_{>0}$, let $\mu_y$ be the norm on $V_v$ defined by 
$$
	\mu_y(ae_1+be_2+ce_3)= \max(|a|, y|b|, y^2|c|).
$$ 
For $x\in F_v$,  consider the norm 
$\mu_y \circ g_x$, where 
$$g_x= \begin{pmatrix}  1 &0&0\\ 0&1&x\\ 0&0&1\end{pmatrix}.$$
If we vary $x\in F_{\infty}$ and let $y\in \R_{>0}$ approach $\infty$, then the class of $\mu_y \circ g_x$ in $X_v$ converges under the Satake topology to the class $\alpha \in \bar{X}_{F,v}$ of the pair that is the Borel subgroup of upper triangular matrices and  the unique element of $\prod_{i=0}^2 X_{(V_i/V_{i-1})_v}$, where $(V_i)_{-1 \le i \le 2}$
is the corresponding flag.
The quotient topology on $\bar X_{F,v}^{\flat}$ of the Satake topology on $\bar X_{F,v}$ is finer than the Satake topology on $\bar X_{F,v}^{\flat}$ by \ref{top1} and \ref{top2}.
Thus, if the Satake topology is finer than the weak topology on $\bar X_{F,v}$ or $\bar{X}_{F,v}^{\flat}$, then the composite $\mu_y \circ g_x$ should converge in $\bar X_v^{\flat}$ to the class of the semi-norm $\nu$ on $V_v^*$ that satisfies $\nu(ae_1^*+be_2^*+ce_3^*) = |a|$.  
However, if $y\to \infty$ and $y^{-2}|x|\to \infty$, then the class of $\mu_y \circ g_x$ in $X_v$ converges in $\bar X_v^{\flat}$ to the class of the semi-norm $ae_1^*+be_2^*+ce_3^*\mapsto |b|$. In fact, by \ref{actdual} we have
\begin{align*} (\mu_y \circ g_x)^*(ae_1^*+be_2^*+ce_3^*) &= \mu_y^*\circ (g_x^*)^{-1}(ae_1^*+be_2^*+ce_3^*) \\
&= \max(|a|, \; y^{-1}|b|, \; y^{-2}|{-bx+c}|)= y^{-2}|x|\nu_{y,x}
\end{align*}
where $\nu_{y,x}$ is the norm 
$$ae_1^*+be_2^*+ce_3^*\mapsto \max(y^2|x|^{-1}|a|, \; y|x|^{-1}|b|, \;  |{-b + x^{-1}c}|)$$
 on $V^*_v$. The norms $\nu_{y,x}$ converge to the semi-norm $ae_1^*+be_2^*+ce_3^*\mapsto |b|$.

\end{sbpara}

\begin{sbpara}\label{toW2}  Let $W$ be a nonzero subspace of $V$.
We demonstrate that the map
$ \phi_{W,v}^{\flat} \colon \bar X_{F,v}^{\flat}(W) \to X_{W_v}$ 
of \ref{toW} given by restriction to $W_v$ need not be continuous for the weak topology, even though by \ref{togrW}, it is continuous for the Borel-Serre topology and (if all places other than $v$ are non-archimedean) for the Satake topology. 

For example, suppose that $v$ is non-archimedean and $d = 3$.  Fix a basis $(e_i)_{1\leq i\leq 3}$ of $V$, and let $W=Fe_1+Fe_2$.  Let $\mu$ be the class of the norm 
$$
	ae_1+be_2\mapsto \max(|a|, |b|)
$$ 
on $W_v$, and consider the element $(W,\mu)\in \bar X_{F,v}^{\flat}$.  For $x\in F_v$ and $\epsilon\in \R_{>0}$, let $\mu_{x,\epsilon}\in X_v$ be the class of the norm 
$$
	ae_1+be_2+ce_3\mapsto \max(|a|, |b|,\epsilon^{-1}|c+bx|)
$$ 
on $V_v$. Then $\mu^*_{x,\epsilon}$ is the class of the norm $$
	ae_1^*+be_2^*+ce_3^*\mapsto \max(|a|, |b-xc|, \epsilon|c|)
$$ 
on $V_v^*$. 
When $x\to 0$ and $\epsilon\to 0$, the last norm converges to the semi-norm 
$$
	ae_1^*+be_2^*+ ce_3^*\mapsto \max(|a|, |b|)
$$ 
on $V_v^*$, 
and this implies that $\mu_{x,\epsilon}$ converges to $(W,\mu)$ for the weak topology. However, the restriction of $\mu_{x,\epsilon}$ to $W_v$ is the class of the norm 
$$
	ae_1+be_2\mapsto \max(|a|, |b|, \epsilon^{-1}|x||b|).
$$ 
If $x\to 0$ and $\epsilon= r^{-1}|x|\to 0$ for a fixed $r >1$, then the latter norms converge to the norm $ae_1+be_2\mapsto \max(|a|, r|b|)$, the class of which does not coincide with $\mu$.

\end{sbpara}

\begin{sbpara} \label{phiweak} Let $P$ be a parabolic subgroup of $\PGL_V(F)$.
We demonstrate that the map $\phi_{P,v} \colon \bar X_{F,v}(P)\to \frak{Z}_{F,v}(P)$ of \ref{togr} is not necessarily continuous for the weak topology, though by \ref{keycor}, it is continuous for the Borel-Serre topology and for the Satake topology. 

Let $d = 3$ and $W$ be as in \ref{toW2}, and let $P$ be the parabolic subgroup of $\PGL_V$ corresponding to the flag
$$
	0=V_{-1}\subset V_0=W \subset V_1=V.
$$ 
In this case, the canonical map $\bar X_{F,v}(P)\to \bar X_{F,v}^{\flat}(W)$ is a homeomorphism for the weak topology on both spaces.  It is also a homeomorphism for the Borel-Serre topology, and for the Satake topology if all places other than $v$ are non-archimedean.  Since $\frak{Z}_{F,v}(P) = X_{(V_0)_v} \times X_{(V/V_0)_v} \cong X_{W_v}$, the argument of \ref{toW2} shows that $\phi_{P,v}$ is not continuous for the weak topology.

\end{sbpara}

\begin{sbpara} For $d\geq 3$, the Satake topology on $\bar X_{F,v}^{\flat}$ does not coincide with the quotient topology for the Satake topology on $\bar X_{F,v}$, which is strictly finer.  This is explained in
\ref{flat8}. 
\end{sbpara}

\subsection{Relations with Borel-Serre spaces and reductive Borel-Serre spaces} \label{relations}

\begin{sbpara}\label{BS6}
In this subsection, we describe the relationship between our work and the theory of Borel-Serre and reductive Borel-Serre spaces (see Proposition \ref{Qcase}). 
We also show that $\bar X_{F,v}^{\sharp}$ is not Hausdorff if $v$ is a non-archimedean place. 
\end{sbpara}
 \begin{sbpara}\label{BS7} 
 
 Let $G$ be a semisimple algebraic group over $\Q$. We recall the definitions of the Borel-Serre and reductive Borel-Serre spaces associated to $G$ from \cite{BS} and \cite[p.~190]{Z}, respectively.

Let $\mathcal{Y}$ be the space of all maximal compact subgroups of $G(\R)$.
Recall from \cite[Proposition 1.6]{BS} that for $K\in \mathcal{Y}$, the Cartan involution $\theta_K$ 
of $G_\R:=\R\otimes_{\Q} G$ corresponding to $K$ is the unique homomorphism $G_\R\to G_\R$ such that 
$$
	K=\{g\in G(\R) \mid \theta_K(g)=g\}.
$$ 

Let $P$ be a parabolic subgroup of  $G$, let $S_P$ be the largest $\Q$-split torus in the center of $P/P_u$, and let $A_P$ be the connected component of the topological group $S_P(\R)$ containing the origin.   We have
$$
 	A_P\cong \R_{>0}^r \subset S_P(\R) \cong (\R^\times)^r
$$ 
for some integer $r$. We define an action of $A_P$ on $\mathcal{Y}$ as follows (see \cite[Section 3]{BS}).
For $K\in \mathcal{Y}$, 
we have a unique subtorus $S_{P,K}$ of $P_\R=\R\otimes_\Q P$ over $\R$ 
such that the projection $P\to P/P_u$ induces an isomorphism 
$$
	S_{P,K} \xrightarrow{\sim} (S_P)_\R:=\R \otimes_\Q S_P
$$
and such that the Cartan involution $\theta_K \colon G_\R \to G_\R$ of $K$ 
satisfies $\theta_K(t) = t^{-1}$ for all $t \in S_{P,K}(\R)$.  
For $t\in A_P$, let $t_K\in S_{P,K}(\R)$ be the inverse image of $t$.  
Then $A_P$ acts on $\mathcal{Y}$ by
$$
 	A_P\times \mathcal{Y} \to \mathcal{Y}, \qquad (t,K) \mapsto t_KKt_K^{-1}.
$$

The Borel-Serre space is the set of pairs $(P, Z)$ 
such that $P$ is a parabolic subgroup of $G$ and $Z$ is an $A_P$-orbit in $\mathcal{Y}$.  
The reductive Borel-Serre space is the quotient of the Borel-Serre space by the equivalence relation under which two elements $(P, Z)$ and $(P', Z')$ are equivalent if $(P', Z')= g(P, Z)$ (that is, $P=P'$ and $Z'=gZ$) for some $g\in P_u(\R)$.
\end{sbpara} 
  
  \begin{sbpara}\label{BS8} Now assume that $F=\Q$ and $G=\PGL_V$. Let $v$ be the archimedean place of $\Q$.
  
  We have a bijection between $X_v$ and the set $\mathcal{Y}$ of all maximal compact subgroups of $G(\R)$, 
whereby an element of $X_v$ corresponds to its isotropy group in $G(\R)$, 
which is a maximal compact subgroup. 

Suppose that $K\in \mathcal{Y}$ corresponds to $\mu\in  X_v$, with $\mu$ the class of a norm that in turn corresponds to a positive definite symmetric bilinear form $(\;\,,\;)$ on $V_v$.  The Cartan involution  $\theta_K : G_\R\to G_\R$ is induced by the unique homomorphism $\theta_K \colon \GL_{V_v}\to \GL_{V_v}$ satisfying
$$
	(gx, \theta_K(g)y)=(x,y) \quad \text{for all } g\in \GL_V(\R) \text{ and } x, y\in V_v. 
$$

For a parabolic subgroup $P$ of  $G$
corresponding to a flag $(V_i)_{-1 \le i \le m}$,  we have 
$$S_P= \Bigg(\prod_{i=0}^m \mb{G}_{\mathrm{m},\Q}\Bigg)/\mb{G}_{\mathrm{m},\Q},$$ where the $i$th term in the product is the group of scalars in $\GL_{V_i/V_{i-1}}$, and where the last $\mb{G}_{\mathrm{m},\Q}$ is embedded diagonally in the product.  The above description of $\theta_K$ shows that $S_{P,K}$ is the lifting of $(S_P)_{\R}$ to $P_\R$ obtained through the orthogonal direct sum decomposition
$$V_v \cong \bigoplus_{i=0}^m \;  (V_i/V_{i-1})_v$$ with respect to $(\;\,,\;)$.

\end{sbpara}

\begin{sbprop} \label{Qcase}
	If $v$ is the archimedean place of $\Q$, then $\bar X_{\Q,v}^{\sharp}$ (resp., $\bar X_{\Q,v}$) is the 
	Borel-Serre space (resp., reductive Borel-Serre space) associated to $\PGL_V$. 
\end{sbprop}

\begin{proof}
	Denote the Borel-Serre space by $(\bar X_{\Q,v}^{\sharp})'$ in this proof.  We define a canonical map
	$$
		\bar X_{\Q,v}^{\sharp}\to (\bar X_{\Q,v}^{\sharp})', \qquad (P, \mu, s) \mapsto (P,Z),
	$$
        where $Z$ is the subset of $\mathcal{Y}$ corresponding to the following subset $Z'$ of $X_v$. Let $(V_i)_{-1 \le i \le m}$ 
        be the flag corresponding to $P$.  Recall that $s$ is an isomorphism
        $$
         	s \colon \bigoplus_{i=0}^m \;  (V_i/V_{i-1})_v \xrightarrow{\sim} V_v.
        $$ 
        Then $Z'$ is the subset of $X_v$ consisting of classes of the norms
        $$
       		\tilde \mu^{(s)} \colon x \mapsto \left(\sum_{i=0}^m {\tilde \mu_i}(s^{-1}(x)_i)^2\right)^{1/2}
        $$ 
        on $V_v$, where $s^{-1}(x)_i\in (V_i/V_{i-1})_v$ denotes the $i$th component of $s^{-1}(x)$ for $x \in V_v$,
        and $\tilde \mu =(\tilde \mu_i)_{0\leq i\leq m}$ ranges over all families of norms $\tilde \mu_i$ on 
        $(V_i/V_{i-1})_v$ with class equal to $\mu_i$. It follows from the description of $S_{P,K}$ in \ref{BS8} that 
        $Z$ is an $A_P$-orbit. 
             
    For a parabolic subgroup $P$ of $G$, let 
    $$
    	(\bar X_{\Q,v}^{\sharp})'(P) = \{ (Q,Z) \in (\bar X_{\Q,v}^{\sharp})' \mid Q\supset P \}.
    $$ 
       By \cite[7.1]{BS}, the subset $(\bar X_{\Q,v}^{\sharp})'(P)$ is open in $(X_{\Q,v}^{\sharp})'$.

    Take a basis of $V$, and let $B$ denote the Borel subgroup of $\PGL_V$ of upper-triangular 
    matrices for this basis. 
    By \ref{lst0}(3), we have a homeomorphism
    $$B_u(\R) \times
    \R_{\geq 0}^{d-1} \xrightarrow{\sim} \bar X_{\Q,v}^{\sharp}(B).$$
     It follows from \cite[5.4]{BS} that the composition 
        $$B_u(\R) \times \R_{\geq 0}^{d-1}\to  (\bar X_{\Q,v}^{\sharp})'(B)$$
    induced by the above map  
    $$
    	\bar X_{\Q,v}^{\sharp}(B) \to (\bar X_{\Q,v}^{\sharp})'(B), \qquad (P, \mu, s)\mapsto (P,Z)
    $$ 
    is also a homeomorphism. 
        This proves that the map $\bar X_{\Q,v}^{\sharp}\to( \bar X_{\Q,v}^{\sharp})'$ restricts to a homeomorphism 
    $\bar X_{\Q,v}^{\sharp}(B) \xrightarrow{\sim} (\bar X_{\Q,v}^{\sharp})'(B)$.  Therefore,
    $\bar X_{\Q,v}^{\sharp}\to  (\bar X_{\Q,v}^{\sharp})'$ is a homeomorphism as well. 
    It then follows directly from the definitions that the reductive Borel-Serre space is identified with $\bar X_{\Q,v}$. 
\end{proof}

\begin{sbpara} \label{BSF}
Suppose that $F$ is a number field, let $S$ be the set of all archimedean places of $F$, and let $G$ be the Weil restriction $\mathrm{Res}_{F/\Q}\PGL_V$, which is a semisimple algebraic group over $\Q$. Then $\mathcal{Y}$ is identified with $X_{F,S}$, and $\bar X_{F,S}$ is related to the reductive Borel-Serre space associated to $G$ but does not always coincide with it.  We explain this below.

Let $(\bar X_{F,S}^{\sharp})'$  and $\bar X_{F,S}'$ be the Borel-Serre space and the reductive Borel-Serre space associated to $G$, respectively. Let $\bar X_{F,S}^{\sharp}$ be the subspace of $\prod_{v\in S} \bar X_{F,v}^{\sharp}$ consisting of all elements $(x_v)_{v\in S}$ such that the parabolic subgroup of $G$ associated to $x_v$ is independent of $v$. Then by similar arguments to the case $F=\Q$, we see that $\cal Y$ is canonically homeomorphic to $X_{F,S}$ and this homeomorphism extends uniquely to surjective continuous maps 
$$(\bar X_{F,S}^{\sharp})' \to \bar X_{F,S}^{\sharp}, \qquad \bar X_{F,S}' \to \bar X_{F,S}.$$
However, these maps are not bijective unless $F$ is $\Q$ or imaginary quadratic.  We illustrate the differences between the spaces in the case that $F$ is a real quadratic field and $d=2$. 

Fix a basis $(e_i)_{i=1,2}$ of $V$. Let $\tilde P$ be the Borel subgroup of upper triangular matrices in $\PGL_V$ for this basis, and let $P$ be the Borel subgroup 
$\text{Res}_{F/\Q}\tilde P$ of $G$. Then $P/P_u \cong \text{Res}_{F/\Q}\mb{G}_{m,F}$ 
and $S_P= \mb{G}_{m,\Q} \subset P/P_u$. We have the natural identifications 
$\mathcal{Y}=X_{F,S}= \frak H \times \frak H$.  For $a\in \R_{>0}$, the set 
$$
	Z_a:=\{(yi, ayi)\in \frak H \times \frak H \mid y\in \R_{>0}\}
$$
is an $A_P$-orbit. If $a\neq b$, the images of $(P, Z_a)$ and $(P, Z_b)$ in $(\bar X_{F,S})'$ 
do not coincide. On the other hand, both the images of $(P, Z_a)$ and $(P, Z_b)$ in $\bar X_{F,S}^{\sharp}$ 
coincide with $(x_v)_{v\in S}$, where $x_v=(P, \mu_v, s_v)$ with $\mu_v$ 
the unique element of $X_{F_v e_1}\times X_{V_v/F_ve_1}$ and $s_v$ the splitting given by $e_2$. 
\end{sbpara}

%\begin{sbpara}
%We have a canonical map $$\bar X_{F,S}^{\sharp}\to \prod_{v\in S} \bar X_{F,v}^{\sharp}, \qquad (P, \mu, s)\mapsto
%(P, \mu_v, s_v)_{v\in S}.$$ This map is not necessarily injective. 
%For example, take $F$ to be a real quadratic field and $V=F^2$. Then $X_v\cong \frak H$ for $v\in S$. For $a\in \R_{>0}$, any sequence of $(yi, ayi)\in \frak H\times \frak H$ 
%converges in $\bar X_{F, S}^{\sharp}$ as $y\to \infty$, the
%limit depending on $a$.  On the other hand, the images of 
%these limit points  in $\prod_{v\in S} X_{F,v}^{\sharp}$ are all equal to $(\infty , \infty)$ 
%so do not depend on $a$. 
%\end{sbpara}

 In the case that $v$ is non-archimedean, the space $\bar X_{F,v}^{\sharp}$ is not good in the following sense.

\begin{sbprop}\label{nonH}
If $v$ is non-archimedean, then $\bar X_{F,v}^{\sharp}$ is not Hausdorff.
\end{sbprop}

\begin{pf}
Fix $a, b\in B_u(F_v)$ with $a\neq b$, for a Borel subgroup $B$ of $\PGL_V$.  When $t\in \R^{d-1}_{>0}$ is sufficiently near to ${0} = (0,\ldots,0)$, the images of $(a,t)$ and $(b,t)$ in $X_v$ coincide by \ref{lst0}(4) applied to $B_u(F_v) \times \R_{>0}^{d-1}\to X_v$. We denote this element of $X_v$ by $c(t)$. The images $f(a)$ of $(a,{0})$ and $f(b)$ of $(b,{0})$ in $\bar X_{F,v}^{\sharp}$ are different.  However, $c(t)$ converges to both $f(a)$ and $f(b)$ as $t$ tends to
${0}$. Thus, $\bar X_{F,v}^{\sharp}$ is not Hausdorff. 
\end{pf}

\begin{sbpara}
Let $F$ be a number field, $S$ its set of archimedean places, and $G = \mathrm{Res}_{F/\Q}\PGL_V$, as in \ref{BSF}.  Then $\bar{X}_{F,S}$ may be identified with the maximal Satake space for $G$ of \cite{Sa2}.  Its Satake topology was considered by Satake (see also \cite[III.3]{BJ}), and its Borel-Serre topology was considered by Zucker \cite{Z2} (see also \cite[2.5]{JMSS}).  The space $\bar{X}_{F,S}^{\flat}$ is also a
Satake space corresponding to the standard projective representation of $G$ on $V$ viewed as a $\Q$-vector space.
\end{sbpara}

\section{Quotients by $S$-arithmetic groups} \label{quotS}
As in \S\ref{globspace}, fix a global field $F$ and a finite-dimensional vector space $V$ over $F$.

\subsection{Results on $S$-arithmetic quotients} \label{mainres}

\begin{sbpara}\label{mains0}

Fix a nonempty finite set $S_1$ of places of $F$ which contains all archimedean places of $F$, fix a finite set $S_2$ of places of $F$ which is disjoint from $S_1$, and let $S=S_1\cup S_2$. 

\end{sbpara}

\begin{sbpara}\label{mains1} 

In the following, we take $\bar X$ to be one of the following two spaces:
\begin{enumerate}
\item[(i)] $\bar X:=\bar X_{F,S_1}$,
\item[(ii)] $\bar X:=\bar X_{F,S_1}^{\flat}$.  
\end{enumerate}
We endow $\bar X$ with either the Borel-Serre or the Satake topology. 

\end{sbpara}

\begin{sbpara}\label{mains3} 
Let $G=\PGL_V$, and let $K$ be a compact open subgroup of $G(\mb{A}_F^S)$, with $\mb{A}_F^S$ as in \ref{top8}.

We consider the two situations in which $(\frak X, \bar{\frak X})$ is taken to be one of the following pairs of spaces 
(for either choice of $\bar{X}$):
\begin{enumerate}
\item[(I)]  $\displaystyle \frak X:= 
 X_S \times G(\mb{A}_F^S)/K \;\subset\;  \bar {\frak X} := \bar X \times X_{S_2} \times G(\mb{A}_F^S)/K,$
\item[(II)]   $\displaystyle \frak X:= X_S \;\subset\; \bar{\frak X}:= \bar X \times X_{S_2}.$
\end{enumerate}
\end{sbpara}

We now come to the main result of this paper.

\begin{sbthm}\label{main} Let the situations and notation be as in \ref{mains0}--\ref{mains3}. 
\begin{enumerate}
\item[(1)] Assume we are in situation (I). 
 Let $\Gamma$ be a subgroup of $G(F)$. Then the quotient space $\Gamma \bs \bar {\frak X}$ is Hausdorff. It is compact if $\Gamma=G(F)$.
\item[(2)] Assume we are in situation (II). Let 
$\Gamma_K\subset G(F)$ be the inverse image of $K$ under the canonical map $G(F)\to G(\mb{A}^S_F)$, and let $\Gamma$ be a subgroup of $\Gamma_K$.  Then the quotient space $ \Gamma\bs \bar {\frak X}$ is Hausdorff. It is compact if $\Gamma$ is of finite index in $\Gamma_K$. 
\end{enumerate}
\end{sbthm}

\begin{sbpara} The case $\Gamma=\{1\}$ of Theorem \ref{main} shows that $\bar X_{F,S}$ and $\bar X_{F,S}^{\flat}$ are Hausdorff. 

\end{sbpara}

\begin{sbpara}
Let $O_S$ be the subring of $F$ consisting of all elements which are integral outside $S$. Take an $O_S$-lattice $L$ in $V$. Then $\PGL_L(O_S)$ coincides with $\Gamma_K$ for the compact open subgroup $K=\prod_{v\notin S} \PGL_L(O_v)$ of $G(\mb{A}_F^S)$. Hence Theorem \ref{thm1} of the introduction follows from Theorem \ref{main}. 
\end{sbpara}

\begin{sbpara}
In the case that $F$ is a number field and $S$ (resp., $S_1$) is the set of all archimedean places of $F$, Theorem \ref{main} in situation (II) is a special case of results of Satake \cite{Sa2} (resp., of Ji, Murty, Saper, and Scherk \cite[Proposition 4.2]{JMSS}).
\end{sbpara}

\begin{sbpara}\label{same} If in Theorem \ref{main} we take $\Gamma =G(F)$ in part (1), or $\Gamma$ of finite index in $\Gamma_K$ in part (2), then the Borel-Serre and Satake topologies on $\bar X$ induce the same topology on the quotient space $\Gamma \bs \bar {\frak X}$. This can be proved directly, but it also follows from the compact Hausdorff property. 
\end{sbpara}

 \begin{sbpara}\label{nonH2} We show that some modifications  of Theorem \ref{main} are not good. 
  
Consider the case  $F=\Q$, $S=\{p, \infty\}$ for a prime number $p$,  and $V=\Q^2$, and consider the $S$-arithmetic group $\PGL_2(\Z[\frac{1}{p}])$. 
Note that 
 $\PGL_2(\Z[\frac{1}{p}])\bs (\bar X_{\Q,\infty}\times X_p)$ is compact Hausdorff, as is well known (and follows from Theorem \ref{main}). We show that some similar spaces are not Hausdorff.  That is, we prove the following statements:
\begin{enumerate}
	\item[(1)] $\PGL_2(\Z[\frac{1}{p}]) \bs (\bar X_{\Q,p}\times X_{\infty})$ is not Hausdorff. 
	\item[(2)] $\PGL_2(\Z[\frac{1}{p}])\bs (\bar X_{\Q,\infty}\times \PGL_2(\Q_p))$ is not Hausdorff.
	\item[(3)] $\PGL_2(\Q)\bs (\bar X_{\Q,\infty}\times \PGL_2(\mb{A}_\Q^{\infty}))$ is not Hausdorff. 
 \end{enumerate}
 
Statement (1) shows that it is important to assume in \ref{main} that $S_1$, not only $S$, contains all archimedean places. 
Statement (3) shows that it is important to take the quotient $G(\mb{A}^S_F)/K$ in situation (I) of \ref{mains3}. 
 
Our proofs of these statements rely on the facts that the quotient spaces $\Z[\frac{1}{p}]\bs \R$, $\Z[\frac{1}{p}]\bs \Q_p$,  and $\Q\bs \mb{A}^{\infty}_\Q$ are not Hausdorff. 

\begin{proof}[Proof of statements (1)--(3)]
 For an element $x$ of a ring $R$, let
 $$g_x=\begin{pmatrix}  1&x\\0&1\end{pmatrix} \in \PGL_2(R).$$ In (1), for $b\in \R$, let $h_b$ be the point $i+b$ of the upper half plane $\frak H= X_{\infty}$. In (2), for $b\in \Q_p$, let $h_b=g_b\in \PGL_2(\Q_p)$. In (3), for $b\in \mb{A}^{\infty}_\Q$, let $h_b=g_b \in \PGL_2(\mb{A}^{\infty}_\Q)$. 
 In (1)  (resp., (2) and (3)), let $\infty\in \bar X_{\Q,p}$ (resp., $\bar X_{\Q,\infty}$) be the boundary point corresponding to the the Borel subgroup of upper triangular matrices. 
 
 In (1) (resp., (2), resp., (3)), take an element $b$ of $\R$ (resp., $\Q_p$, resp., $\mb{A}^{\infty}_{\Q}$) which does not belong to $\Z[\frac{1}{p}]$ (resp., $\Z[\frac{1}{p}]$, resp., $\Q$). Then the images of $(\infty, h_0)$ and $(\infty, h_b)$ in the quotient space are different, but they are not separated. 
Indeed, in (1) and (2) (resp., (3)), some sequence of elements $x$ of $\Z[\frac{1}{p}]$ (resp., $\Q$) will converge to $b$, 
in which case $g_x(\infty, h_0)$ converges to $(\infty, h_b)$ since $g_x\infty=\infty$.  
\end{proof}
 
 \end{sbpara}

\subsection{Review of reduction theory} \label{redrev}

We review important results in the reduction theory of algebraic groups: \ref{red1}, \ref{redc}, and a variant \ref{red11} of \ref{red1}. More details may be found in the work of Borel \cite{Bo} and Godement \cite{Go} in the number field case and 
Harder \cite{H1, H2} in the function field case. 

Fix a basis $(e_i)_{1\leq i\leq d}$ of $V$. Let $B$ be the Borel subgroup of $G=\PGL_V$ consisting of all upper triangular matrices for this basis.  
Let $S$ be a nonempty finite set of places of $F$ containing all archimedean places.

\begin{sbpara}\label{red0} 
For $b = (b_v) \in \mb{A}_F^\times$, set $|b| = \prod_v |b_v|$.  Let $A_v$ be as in \ref{decomp}.
We let $a \in \prod_v A_v$ denote the image of a diagonal matrix $\diag(a_1,\ldots,a_d)$ in $\GL_d(\mb{A}_F)$.  The ratios
$a_ia_{i+1}^{-1}$ are independent of the choice.
For $c \in \R_{>0}$, we let  $B(c) = B_u(\mb{A}_F) A(c)$, where
$$
	A(c) =\left\{ a \in \prod_v A_v \cap \PGL_d(\mb{A}_F) \mid |a_i a_{i+1}^{-1}|\geq c \text{ for all } 1 \le i \le d-1 \right\}.
$$

Let $K^0= \prod_v K_v^0 < G(\mb{A}_F)$, where $K_v^0$ is identified via $(e_i)_i$ with the standard maximal compact 
subgroup of $\PGL_d(F_v)$ of \ref{decomp}.  Note that $B_u(F_v) A_v K_v^0 = B(F_v)K_v^0 = G(F_v)$ for all $v$.

\end{sbpara}

We recall the following known result in reduction theory: see \cite[Theorem 7]{Go} and \cite[Satz 2.1.1]{H1}.

\begin{sblem} \label{red1} 
	For sufficiently small $c\in \R_{>0}$, one has $G(\mb{A}_F)=G(F)B(c)K^0$. 
\end{sblem}

\begin{sbpara}\label{redb}

Let the notation be as in \ref{red0}. For a subset $I$ of $\{1,\dots, d-1\}$, let $P_I$ be the parabolic subgroup of $G$ corresponding to the flag consisting of  $0$, the
 $F$-subspaces $\sum_{1\leq j\leq i} F e_j$ for $i\in I$, and $V$. Hence $P_I \supset B$ for all $I$, with $P_{\varnothing} = G$ and $P_{\{1, \dots, d-1\}}=B$.  

For $c' \in \R_{>0}$, let $B_I(c,c') = B_u(\mb{A}_F) A_I(c,c')$, where
$$
	A_I(c,c') = \{ a \in A(c) \mid  |a_ia_{i+1}^{-1}|\geq c'  \text{ for all }i\in I\}.  
$$	
Note that $B_I(c,c') = B(c)$ if $c \ge c'$.

\end{sbpara}

The following is also known \cite[Satz 2.1.2]{H1} (see also \cite[Lemma 3]{Go}):

\begin{sblem}\label{redc} 

Fix $c \in \R_{>0}$ and a subset $I$ of $\{1, \ldots, d-1\}$. Then 
there exists $c' \in \R_{>0}$ such that 
$$
	\{ \gamma \in G(F) \mid B_I(c,c')K^0 \cap \gamma^{-1}B(c)K^0 \neq \varnothing \} \subset P_I(F).
$$

\end{sblem}

\begin{sbpara} \label{redv} We will use the following variant of \ref{redb}. 

Let $A_S = \prod_{v \in S} A_v$.
For $c \in \R_{>0}$, let 
\begin{eqnarray*}
	A(c)_S = A_S \cap A(c) &\text{and}& B(c)_S = B_u(\mb{A}_{F,S}) A(c)_S.
\end{eqnarray*}
For $c_1, c_2 \in \R_{>0}$, set
\begin{multline*}
	A(c_1,c_2)_S =  \{ a  \in A_S \mid \text{for all } v \in S \text{ and } 1 \le i \le d-1, \\
	|a_{v,i}a_{v,i+1}^{-1}| \geq c_1 \text{ and }
	|a_{v,i}a_{v,i+1}^{-1}| \geq c_2|a_{w,i}a_{w,i+1}^{-1}| \text{ for all } w \in S \}.
\end{multline*}
Note that $A(c_1,c_2)_S$ is empty if $c_2 > 1$.
For a compact subset $C$ of $B_u(\mb{A}_{F,S})$, we then set 
$$
	B(C;c_1,c_2)_S = C \cdot A(c_1,c_2)_S.
$$

Let $D_S =\prod_{v\in S} D_v$, where $D_v = K_v^0 < G(F_v)$ if $v$ is archimedean, and 
$D_v < G(F_v)$ is identified with $S_d\Iw(O_v) < \PGL_d(F_v)$ using the basis $(e_i)_i$ otherwise.
Here, $S_d$ is the symmetric group of degree $d$ and $\Iw(O_v)$ is the Iwahori subgroup of $\PGL_d(O_v)$,
as in \ref{Iwnon}.
\end{sbpara}

\begin{sblem} \label{red11}
Let $K$ be a compact open subgroup of $G(\mb{A}_F^S)$, let $\Gamma_K$ be the inverse image of $K$ under $G(F) \to G(\mb{A}^S_F)$,  and let $\Gamma$ be a subgroup of $\Gamma_K$ of finite index. Then there exist $c_1 ,c_2, C$ as above and a finite subset $R$ of $G(F)$ such that
$$
	G(\mb{A}_{F,S}) = \Gamma R \cdot B(C;c_1,c_2)_S D_S.
$$   
\end{sblem}

\begin{proof}
This can be deduced from \ref{red1} by standard arguments in the following manner.
By the Iwasawa decomposition \ref{decomp}, we have $G(\mb{A}_F^S)=B(\mb{A}_F^S)K^{0,S}$ where $K^{0,S}$ is the non-$S$-component of $K^0$. 
Choose a set $E$ of representatives in $B(\mb{A}^S_F)$ of the finite set
$$
	B(F)\bs B(\mb{A}^S_F)/(B(\mb{A}_F^S)\cap K^{0,S}).
$$
Let $D^0 = D_S \times K^{0,S}$, 
and note that since $A_S \cap D_S = 1$, we can (by the Bruhat decomposition \ref{Iwnon}) replace $K^0$ by
$D^0$ in Lemma \ref{red1}.
Using the facts that $E$ is finite, $|a| = 1$ for all $a \in F^{\times}$, 
and $D^0$ is compact, we then have that there exists $c \in \R_{>0}$ such that
$$
	G(\mb{A}_F) = G(F)(B(c)_S \times E) D^0.
$$
For any finite subset $R$ of $G(F)$ consisting of one element from each of those sets
$G(F) \cap K^{0,S}e^{-1}$ with $e \in E$ that are nonempty, we obtain from this that
$$
	G(\mb{A}_{F,S}) = \Gamma_K R \cdot B(c)_SD_S.
$$ 
As $\Gamma_K$ is a finite union of right $\Gamma$-cosets,
we may enlarge $R$ and replace $\Gamma_K$ by $\Gamma$.
Finally, we can replace $B(c)_S$ by 
$C \cdot A(c)_S$ for some $C$ by
the compactness of the image of 
$$
	B_u(\mb{A}_{F,S}) \to \Gamma \bs G(\mb{A}_{F,S}) / D_S
$$
and then by
$B(C;c_1,c_2)_S$ for some $c_1, c_2 \in \R_{>0}$ by the compactness of the cokernel of 
$$
	\Gamma\cap B(\mb{A}_{F,S}) \to (B/B_u)(\mb{A}_{F,S})_1,
$$ 
where $(B/B_u)(\mb{A}_{F,S})_1$ denotes the kernel of the homomorphism
\begin{eqnarray*}
	(B/B_u)(\mb{A}_{F,S})\to \R^{d-1}_{>0}, \qquad a B_u(\mb{A}_{F,S})
	\mapsto \left(\prod_{v\in S} \left| \frac{a_{v,i}}{a_{v,i+1}}\right|\right)_{1\leq i\leq d-1}.
\end{eqnarray*}	
\end{proof}

\subsection{$\bar X_{F,S}$ and reduction theory}

\begin{sbpara}\label{red20}  Let $S$ be a nonempty finite set of places of $F$ containing all archimedean places.  We consider $\bar X_{F,S}$. From the results \ref{red11} and \ref{redc} of reduction theory, we will deduce results \ref{red21} and \ref{red2cc} on $\bar X_{F,S}$, respectively. We will also discuss other properties of $\bar X_{F,S}$ related to reduction theory. 
Let $G$, $(e_i)_i$, and $B$ be as in \S\ref{redrev}.

For $c_1, c_2 \in \R_{>0}$ with $c_2 \ge 1$, we define a subset $\frak T(c_1,c_2)$ of $(\R_{\ge 0}^S)^{d-1}$ by 
$$
	\frak T(c_1,c_2)= \{t\in (\R_{\ge 0}^S)^{d-1}  \mid  t_{v,i} \leq c_1,\; 
	t_{v,i} \leq c_2t_{w,i} \text{ for all } v,w\in S \text{ and } 1\leq i\leq d-1 \}.
$$
Let $Y_0 = (\R_{>0}^S \cup \{(0)_{v \in S}\})^{d-1}$ as in \ref{togr9} (for the parabolic $B$), and note that $\frak T(c_1,c_2) \subset Y_0$.  Define the subset  $\frak S(c_1,c_2)$ of $\bar X_{F,S}(B)$ as the image of $B_u(\mb{A}_{F,S}) \times \frak T(c_1,c_2)$ 
under the map 
$$ \pi_S = (\pi_v)_{v \in S} \colon B_u(\mb{A}_{F,S}) \times Y_0 \to  \bar X_{F,S}(B),$$
with $\pi_v$ as in \ref{cover2}.
For a compact subset $C$ of $B_u(\mb{A}_{F,S})$, we let ${\frak S}(C;c_1,c_2)\subset \frak S(c_1,c_2)$ denote the image of 
$C \times \frak T(c_1,c_2)$ under $\pi_S$.

\end{sbpara}

\begin{sbpara} We give an example of the sets of \ref{red20}.
\begin{example}
Consider the case that $F=\Q$, the set $S$ contains only the real place, and $d = 2$, as in \S\ref{uphp}. 
Fix a basis $(e_i)_{1\leq i\leq 2}$ of $V$. Identify  
$B_u(\R)$ with $\R$ in the natural manner. We have
$$\frak S(C;c_1,c_2)= \{x+yi\in \frak H\mid x\in C, y\geq c_1^{-1}\}\cup \{\infty\},$$ 
which is contained in
$$
\frak S(c_1,c_2) = \{x+yi\in \frak H\mid x\in \R, y\geq c_1^{-1}\}\cup \{\infty\}.$$
\end{example}
\end{sbpara}

\begin{sbpara}
Fix a compact open subgroup $K$ of $G(\mb{A}^S_F)$, and let $\Gamma_K\subset G(F)$ be the inverse image of $K$ under $G(F) \to G(\mb{A}^S_F)$.
\end{sbpara}

\begin{sbprop}\label{red21} 

 Let $\Gamma$ be a subgroup of $\Gamma_K$ of finite index. Then 
there exist $c_1, c_2, C$ as in \ref{red20} 
  and a finite subset $R$ of $G(F)$ such that 
$$\bar X_{F,S}= \Gamma R \cdot \frak S(C;c_1,c_2) .$$

\end{sbprop}

\begin{pf} It suffices to prove the weaker statement that there are $c_1,c_2,C$ and $R$ such that
$$ X_S = \Gamma R \cdot (X_S \cap \frak S(C;c_1,c_2)).$$
Indeed, we claim that the proposition follows from this weaker statement for the spaces in the product 
$\prod_{v \in S} X_{(V_i/V_{i-1})_v}$, where
$P_I$ is as in \ref{redb} for a subset $I$ of $\{1,\dots,d-1\}$
and $(V_i)_{-1 \le i \le m}$ is the corresponding flag.  To see this, first note that there is a finite subset $R'$ of $G(F)$ such that every parabolic subgroup of $G$ has the form $\gamma P_I\gamma^{-1}$ for some $I$ and $\gamma \in \Gamma R'$.
It then suffices to consider $a = (P,\mu) \in \bar{X}_{F,S}$, where $P = P_I$ for some $I$, and $\mu \in \frak{Z}_{F,S}(P)$.  We use the notation of \ref{xidef} and \ref{togr9}.  By Proposition \ref{keyp}, the set $\bar{X}_{F,S}(P) \cap \frak{S}(C;c_1,c_2)$
is the image under $\xi$ of the image of $C \times \frak{T}(c_1,c_2)$ in $P_u(\mb{A}_{F,S}) \times \frak{Z}_{F,S}(P)
\times Y_0$.  Note that $a$ has image $(1,\mu,0)$ in the latter set (for $1$ the identity matrix of $P_u(\mb{A}_{F,S})$), 
and $\xi(1,\mu,0) = a$.  Since the projection of $\frak{T}(c_1,c_2)$ (resp., $C$) to $(\R_{>0}^S)^{\Delta_i}$ (resp., $B_{i,u}(\mb{A}_{F,S})$) is the analogous set for $c_1$ and $c_2$ (resp., a compact subset), the claim follows.

For $v\in S$, we define subsets $Q_v$ and $Q'_v$ of $X_v$ as follows. If $v$ is archimedean, let $Q_v=Q'_v$ be the one point set consisting of the element of $X_v$ given by the basis $(e_i)_i$ and $(r_i)_i$ with $r_i=1$ for all $i$. If $v$ is non-archimedean, let $Q_v$ (resp., $Q'_v$) be the subset of $X_v$ consisting of elements given by $(e_i)_i$ and $(r_i)_i$ such that $1=r_1\leq \dots \leq r_d\leq q_v$ (resp., $r_1=1$ and $1\leq r_i\leq q_v$ for $1\leq i\leq d$). Then $X_v= G(F_v) Q_v$ for each $v\in S$.  
Hence by \ref{red11}, there exist $c'_1,c'_2, C$ as in \ref{red20} and a finite subset $R$ of $G(F)$ such that 
$$
	X_S = \Gamma R  \cdot B(C;c'_1,c'_2)_S \cdot D_SQ_S,
$$
where $Q_S = \prod_{v \in S} Q_v$.

We have $D_SQ_S = Q'_S$ for $Q'_S = \prod_{v \in S} Q'_v$, noting for archimedean (resp., non-archimedean) $v$ that $K_v^0$
(resp., $\Iw(O_v)$) stabilizes all elements of $Q_v$. We have $B(C;c'_1,c'_2)_SQ'_S \subset \frak S(C;c_1,c_2)$, where 
\begin{eqnarray*}
	c_1= \max\{q_v  \mid v\in S_f\}(c'_1)^{-1} &\text{and}& c_2= \max\{q_v^2\mid v\in S_f\}(c'_2)^{-1},
\end{eqnarray*} 
with $S_f$  the set of all non-archimedean places in $S$ (and taking the maxima to be $1$ if $S_f = \varnothing$).
\end{pf}

\begin{sbpara} \label{tvi}
For $v\in S$ and $1\leq i\leq d-1$, let $t_{v,i } \colon \frak S(c_1,c_2) \to \R_{\geq 0}$ be the map induced by $\phi'_{B,v} \colon \bar X_{F,v}(B)\to \R_{\geq 0}^{d-1}$ (see \ref{togr2}) and the $i$th projection $\R_{\geq 0}^{d-1}\to \R_{\geq 0}$. 
Note that $t_{v,i}$ is continuous. 
\end{sbpara}

\begin{sbpara}\label{red2c}

 Fix a subset $I$ of $\{1,\dots, d-1\}$, and let $P_I$ be the parabolic subgroup of $G$ defined  in \ref{redb}. 
 For $c_1,c_2,c_3\in \R_{>0}$, let 
 $$
 	\frak S_I(c_1,c_2,c_3) = \{ x \in \frak S(c_1,c_2) \mid \min\{ t_{v,i}(x) \mid v \in S\} \leq c_3 \text{ for each } i \in I \}.
$$
 
 \end{sbpara}

\begin{sbpara}\label{parat} For an element $a\in \bar X_{F,S}$, we define the parabolic type of $a$ to be the subset 
$$
	\{\dim(V_i)\mid 0\leq i\leq m-1\}
$$
of $\{1,\dots,d-1\}$,
where $(V_i)_{-1 \le i \le m}$ is the flag corresponding to the parabolic subgroup of $G$ associated to $a$.

\end{sbpara}

\begin{sblem}\label{kep20}

 Let $a \in \bar X_{F,S}(B)$, and let $J$ be the parabolic type of $a$.  Then
the parabolic subgroup of $G$ associated to $a$ is $P_J$. 
\end{sblem}

This is easily proved. 

\begin{sbpara}\label{RGR}  In the following, we will often consider subsets of $G(F)$ of the form $R_1\Gamma_K R_2$, $\Gamma_K R$, or $R\Gamma_K$, 
 where $R_1, R_2, R$ are finite subsets of $G(F)$. These three types of cosets are essentially the same thing when we vary $K$. For finite subsets $R_1, R_2$ of $G(F)$, we have $R_1\Gamma_K R_2= R'\Gamma_{K'}= \Gamma_{K''} R''$ for some compact open subgroups $K'$ and $K''$  of $G(\mb{A}_F^S)$ contained in $K$ and finite subsets 
$R'$ and $R''$ of $G(F)$. 
\end{sbpara}

 \begin{sbprop}\label{red2cc}
Given $c_1 \in \R_{>0}$ and finite subsets $R_1, R_2$ of $G(F)$, 
there exists $c_3\in \R_{>0}$ such that 
for all $c_2 \in \R_{>0}$ we have
$$
	\{ \gamma \in R_1\Gamma_KR_2 \mid \gamma \frak S_I(c_1,c_2, c_3) \cap \frak S(c_1,c_2) \neq \varnothing \} \subset P_I(F).
$$
\end{sbprop}

\begin{pf} First we prove the weaker version that $c_3$ exists if the condition on $\gamma \in R_1\Gamma_KR_2$ is replaced
by $\gamma\frak S_I(c_1,c_2,c_3) \cap \frak S(c_1,c_2) \cap X_S \neq \varnothing$.

Let $Q'_v$ for $v \in S$ and $Q'_S$ be as in the proof of \ref{red21}. 

\medskip

 {\bf Claim 1.}  If $c_1'\in \R_{>0}$ is sufficiently small (independent of $c_2$), then we have 
 $$X_S\cap \frak S(c_1,c_2)\subset  B(c'_1)_SQ'_S. $$

\begin{proof}[Proof of Claim 1] \renewcommand{\qedsymbol}{} 
	Any $x \in X_S \cap \frak S(c_1,c_2)$ satisfies $t_{v,i}(x) \le c_1$ for $1 \le i \le d-1$.
	Moreover, if $\prod_{v \in S} t_{v,i}(x)$
	is sufficiently small relative to $(c_1')^{-1}$
	for all such
	$i$, then $x \in B(c_1')_S Q'_S$.  The claim follows.
\end{proof}

Let $C_v$ denote the compact set
$$
	C_v =\{g\in G(F_v) \mid gQ'_v\cap Q'_v\neq \varnothing\}.
$$ 
If $v$ is archimedean, then $C_v$ is the maximal compact open subgroup $K_v^0$ of \ref{red0}. 
Set $C_S = \prod_{v \in S} C_v$.   We use the decomposition $G(\mb{A}_F)
= G(\mb{A}_{F,S}) \times G(\mb{A}_F^S)$ to write elements of $G(\mb{A}_F)$ as pairs.

\medskip

{\bf Claim 2.} Fix $c_1'\in \R_{>0}$. 
The subset $B(c_1')_SC_S \times R_1KR_2$ of $G(\mb{A}_F)$ is contained in $B(c)K^0$ for sufficiently small $c \in \R_{>0}$.

\begin{proof}[Proof of Claim 2] \renewcommand{\qedsymbol}{} 
This follows from the compactness of the $C_v$ for $v\in S$ and the Iwasawa decomposition $G(\mb{A}_F)=B(\mb{A}_F)K^0$.  
\end{proof}

{\bf Claim 3.} Let $c'_1$ be as in Claim 1, and let $c \le c'_1$. Let $c' \in \R_{>0}$. If $c_3\in \R_{>0}$ is sufficiently small (independent of $c_2$), we have 
$$X_S\cap \frak S_I(c_1,c_2,c_3)\subset B_I(c,c')_S Q'_S,$$
where $B_I(c,c')_S = B(\mb{A}_{F,S}) \cap B_I(c,c')$. 

\begin{proof}[Proof of Claim 3] \renewcommand{\qedsymbol}{} 
	An element $x \in B(c)_SQ'_S$ lies in $B_I(c,c')_SQ'_S$ if
	$\prod_{v \in S} t_{v,i}(x) \le (c')^{-1}$ for all $i \in I$.   An element $x \in X_S \cap \frak S(c_1,c_2)$
	lies in $X_S\cap \frak S_I(c_1,c_2,c_3)$ if 
	$\min\{ t_{v,i}(x) \mid v \in S \} \le c_3$ 
	for all $i \in I$.  In this case, $x$ will lie in $B_I(c,c')_SQ'_S$ if $c_3 \le (c')^{-1}c_1^{1-s}$, with $s = \sharp S$. 
\end{proof}

Let $c'_1$ be as in Claim 1, take $c$ of Claim 2 for this $c_1'$ such that $c\leq c'_1$, and let $c'\in \R_{>0}$. Take $c_3$ satisfying the condition of Claim 3 for these $c'_1$, $c$, and $c'$. 
\medskip

{\bf Claim 4.} 
If $X_S\cap \frak S_I(c_1,c_2, c_3) \cap \gamma^{-1}\frak S(c_1, c_2)$ is nonempty for some 
$\gamma \in R_1\Gamma R_2 \subset G(F)$, then  
$B_I(c, c') \cap \gamma^{-1}B(c)K^0$ contains an element of $G(\mb{A}_{F,S}) \times \{1\}$.

\begin{proof}[Proof of Claim 4] \renewcommand{\qedsymbol}{} 
By Claim 3, any $x \in X_S\cap \frak S_I(c_1,c_2, c_3) \cap \gamma^{-1}\frak S(c_1, c_2)$ lies in 
$g Q'_S$ for some $g \in B_I(c, c')_S$. 
By Claim 1, we have
$\gamma x \in g' Q'_S$ for some $g'\in B(c_1')_S$. Since $\gamma x \in \gamma g Q'_S\cap g' Q'_S$, we have $(g')^{-1}\gamma g \in C_S$.
Hence $\gamma g \in  B(c_1')_S C_S$, and therefore $\gamma(g,1) = (\gamma g, \gamma) \in B(c)K^0$ by Claim 2. 
\end{proof}

We prove the weaker version of \ref{red2cc}:  let $x \in X_S\cap \frak S_I(c_1,c_2, c_3) \cap \gamma^{-1}\frak S(c_1, c_2)$ for some $\gamma \in R_1\Gamma R_2$.  Then by Claim 4 and Lemma \ref{redc},
with $c'$ satisfying the condition of \ref{redc} for the given $c$, we have $\gamma \in P_I(F)$.  

We next reduce the proposition to the weaker version, beginning with the following.
\medskip

{\bf Claim 5.} Let $c_1, c_2\in \R_{>0}$.  If $\gamma \in G(F)$ and 
$x\in \frak S(c_1,c_2) \cap \gamma^{-1}\frak S(c_1,c_2)$, then $\gamma \in P_J(F)$, where $J$ is the parabolic type of $x$.

\begin{proof}[Proof of Claim 5] \renewcommand{\qedsymbol}{}
By Lemma \ref{kep20}, the parabolic subgroup associated to $x$ is $P_J$ and that associated to $\gamma x$ is $P_J$. Hence $\gamma P_J \gamma^{-1}= P_J$. Since a parabolic subgroup coincides with its normalizer, we have $\gamma \in P_J(F)$.
\end{proof}

Fix  $J \subset \{1,\dots, d-1\}$, $\xi\in R_1$, and $\eta \in R_2$.

\medskip

{\bf Claim 6.} 
There exists $c_3\in \R_{>0}$ such that 
if $\gamma \in \Gamma_K$ and $x \in  {\frak S}_I(c_1,c_2,c_3) \cap (\xi\gamma\eta)^{-1} {\frak S}(c_1,c_2)$ is of parabolic type $J$, then $\xi \gamma \eta\in P_I(F)$.

\begin{proof}[Proof of Claim 6] \renewcommand{\qedsymbol}{}
Let $(V_i)_{-1 \le i \le m}$ be the flag corresponding to $P_J$. 
Suppose that we have $\gamma_0\in \Gamma_K$ and $x_0\in  {\frak S}(c_1,c_2) \cap (\xi \gamma_0 \eta)^{-1}{\frak S}(c_1,c_2)$ of parabolic type $J$. 
By Claim 5, we have $\xi \gamma_0\eta, \xi \gamma \eta \in P_J(F)$. Hence 
$$\xi \gamma \eta= \xi (\gamma \gamma_0^{-1})\xi^{-1}  \xi \gamma_0 \eta\in \Gamma_{K'} \eta',$$ 
where $K'$ is the compact open subgroup $\xi K \xi^{-1}\cap P_J(\mb{A}_F^S)$ of $P_J(\mb{A}_F^S)$, and $\eta' = \xi \gamma_0 \eta\in P_J(F)$. 
The claim follows from the weaker version of the proposition in which $V$ is replaced by $V_i/V_{i-1}$ (for $0\leq i\leq m$), the group $G$ is replaced by $\PGL_{V_i/V_{i-1}}$, the compact open subgroup $K$ is replaced by the image of $\xi K \xi^{-1}\cap P_J(\mb{A}_F^S)$ in $\PGL_{V_i/V_{i-1}}(\mb{A}_F^S)$, the set $R_1$ is replaced by $\{1\}$, the set $R_2$ is replaced by the image of $\{\eta'\}$ in $\PGL_{V_i/V_{i-1}}(F)$, and $P_I(F)$ is replaced by the image of $P_I(F) \cap P_J(F)$ in $\PGL_{V_i/V_{i-1}}(F)$.
\end{proof}

By Claim 6 for all $J$, $\xi$, and $\eta$, the result is proven.
\end{pf}

\begin{sblem}\label{lemdet} Let $1\leq i\leq d-1$, and let $V'= \sum_{j=1}^i Fe_j$. Let $x\in X_v$ for some $v \in S$, and let $g\in \GL_V(F_v)$
be such that $gV'=V'$.  For $1 \le i \le d-1$, we have
$$\prod_{j=1}^{d-1} \left(\frac{t_{v,j}(gx)}{t_{v,j}(x)}\right)^{e(i,j)} = \frac{|\rdet(g_v \colon V_v'\to V_v')|^{(d-i)/i}}{|\rdet(g_v \colon 
V_v/V_v'\to V_v/V_v')|},$$
where
$$
	e(i, j) = \begin{cases} \frac{j(d-i)}{i} & \text{if } j\leq i,\\ d-j & \text{if } j\geq i. \end{cases}
$$

\end{sblem}

\begin{proof}
	By the Iwasawa decomposition \ref{decomp} and \ref{ndef}, 
	it suffices to check this in the case that $g$ is represented by a diagonal matrix
	$\diag(a_1,\ldots,a_d)$.  
	It follows from the definitions that $t_{v,j}(gx) t_{v,j}(x)^{-1} = |a_j a_{j+1}^{-1}|$, and the rest of the
	verification is a simple computation.
\end{proof}

\begin{sblem}\label{lemdet2} Let $i$ and $V'$ be as in \ref{lemdet}. Let $R_1$ and $R_2$ be finite subsets of $G(F)$. Then there exist $A, B\in \R_{>0}$ such that for all $\gamma \in \GL_V(F)$ with image in $R_1\Gamma_K R_2 \subset G(F)$ and for which $\gamma V'=V'$, we have 
$$A\leq \prod_{v\in S} \frac{ |\rdet(\gamma \colon V_v'\to V_v')|^{(d-i)/i}}{|\rdet(\gamma\colon V_v/V_v'\to V_v/V_v')|} \leq B.$$

\end{sblem}

\begin{pf}
We may assume that $R_1$ and $R_2$ are one point sets $\{\xi\}$ and $\{\eta\}$, respectively.  Suppose that an element $\gamma_0$ with the stated properties of $\gamma$ exists.  Then for any such $\gamma$,
the image of $\gamma \gamma_0^{-1}$ in $G(F)$ belongs to $\xi \Gamma_K \xi^{-1}\cap P_{\{i\}}(F)$, and hence the image of $\gamma \gamma_0^{-1}$ in $G(\mb{A}_F^S)$ belongs to the compact subgroup $\xi K \xi^{-1} \cap P_{\{i\}}(\mb{A}_F^S)$ of $P_{\{i\}}(\mb{A}_F^S)$. Hence 
$$|\rdet(\gamma \gamma^{-1}_0: V'_v \to V_v')|=|\rdet(\gamma \gamma_0^{-1} : V_v/V'_v\to V_v/V'_v)|=1$$
for every place $v$ of $F$ which does not belong to $S$. By Lemma \ref{lemdet} and the product formula, we have 
$$\prod_{v\in S} (|\rdet(\gamma \gamma^{-1}_0\colon V'_v \to V_v')|^{(d-i)/i}\cdot |\rdet(\gamma \gamma_0^{-1} \colon V_v/V'_v\to V_v/V'_v)|^{-1})=1,$$
so the value of the product in the statement is constant under our assumptions, proving the result.
\end{pf}

\begin{sbprop}\label{kep1}  
Fix $c_1, c_2\in \R_{>0}$ and finite subsets
$R_1, R_2$ of $G(F)$. 
Then there exists $A>1$ such that if $x \in \frak S(c_1,c_2) \cap \gamma^{-1} \frak S(c_1,c_2)$ for some $\gamma\in 
 R_1\Gamma_K R_2$, then 
 $$
 	A^{-1}t_{v,i}(x)\leq t_{v,i}(\gamma x) \leq At_{v,i}(x)
$$ for all $v\in S$ and $1\leq i\leq d-1$. 
\end{sbprop}

\begin{pf} By a limit argument, it is enough to consider $x \in X_S\cap \frak S(c_1,c_2)$. Fix $v\in S$. For $x,x' \in X_S\cap \frak S(c_1,c_2)$ and $1\leq i\leq d-1$,
 let $s_i(x,x')= t_{v,i} (x')t_{v,i}(x)^{-1}$. 
 
 For each $1\leq i\leq d-1$, take $c_3(i)\in \R_{>0}$ satisfying the condition in \ref{red2cc} for the set $I=\{i\}$ and both pairs
 of finite subsets $R_1,R_2$ and $R_2^{-1}, R_1^{-1}$ of $G(F)$. 
 Let 
 $$
 	c_3 = \min\{c_3(i) \mid 1\leq i\leq d-1\}.
$$
   For a subset $I$ of $\{1,\dots, d-1\}$, let $Y(I)$ be the subset of $(X_S\cap \frak S(c_1, c_2))^2$ consisting of all pairs $(x,x')$ such that $x'=\gamma x$ for some $\gamma\in R_1\Gamma_K R_2$ and such that
$$
	I= \{1\leq i\leq d-1 \mid \min(t_{v,i}(x) , t_{v,i}(x')) \leq c_3 \}.
$$ 

For the proof of \ref{kep1}, it is sufficient to prove the following statement $(S_d)$, fixing $I$.

\begin{enumerate}
\item[$(S_d)$] There exists $A>1$ such that 
$A^{-1}\leq s_i(x,x') \leq A$
for all $(x,x')\in Y(I)$ and $1\leq i\leq d-1$. 
\end{enumerate}

By Proposition \ref{red2cc}, if $\gamma \in R_1\Gamma_K R_2$ is such that there exists $x \in X_S$ with
$(x,\gamma x)\in Y(I)$, then $\gamma \in P_{\{i\}}(F)$ for all $i \in I$.
Lemmas \ref{lemdet} and \ref{lemdet2} then imply the following for all $i \in I$, noting that $c_2^{-1}t_{w,i}(y) \le 
t_{v,i}(y) \le c_2t_{w,i}(y)$ for all $w \in S$ and $y \in X_S \cap \frak S(c_1,c_2)$.
\begin{enumerate}
\item[($T_i$)]  There exists $B_i > 1$ such that for all $(x,x')\in Y(I)$,  we have  
$$B_i^{-1} \leq \prod_{j=1}^{d-1} s_j(x,x')^{e(i,j)} \leq B_i,$$
where $e(i,j)$ is as in \ref{lemdet}. 
\end{enumerate}

We prove the following statement ($S_i$) for $0\leq i\leq d-1$ by induction on $i$. 
\begin{enumerate}
\item[$(S_i)$] There exists $A_i>1$ such that
$A_i^{-1}\leq s_j(x,x') \leq A_i$
for all $(x,x')\in Y(I)$ and all $j$ such that $1\leq j\leq i$ and $j$ is not the largest element of $I\cap \{1,\dots,i\}$ (if it is nonempty). 
\end{enumerate}

That ($S_0$) holds is clear. Assume that ($S_{i-1}$) holds for some $i \ge 1$.
If $i\notin I$, then since $c_3 \leq t_{v,i}(x) \leq c_1$ and $c_3 \leq t_{v,i}(x') \leq c_1$, we have
$$\frac{c_3}{c_1} \leq s_i(x,x') \leq \frac{c_1}{c_3},$$
and hence ($S_i$) holds with $A_i:=\max(A_{i-1}, c_1c_3^{-1})$. 

Assume that $i\in I$. If $I\cap \{1,\dots, i-1\} = \varnothing$, then ($S_i$) is evidently true with $A_i:=A_{i-1}$. 
If $I\cap \{1,\dots, i-1\} \neq \varnothing$, then let $i'$ be the largest element of this intersection. We compare ($T_i$) and ($T_{i'}$). We have $e(i,j)=e(i',j)$ if $j\geq i$ and $e(i,j) < e(i',j)$ if $j < i$, 
so taking the quotient of the equations in $(T_{i'})$ and $(T_i)$, we have 
$$
	(B_iB_{i'})^{-1} \le \prod_{j=1}^{i-1} s_j(x,x')^{e(i',j)-e(i,j)} \le B_iB_{i'}.
$$
Since ($S_{i-1}$) is assumed to hold, there then exists $a \in \R_{>0}$ such that
$$
	(B_iB_{i'})^{-1}A_{i-1}^{-a} \le s_{i'}(x,x')^{e(i',i')-e(i,i')} \le B_iB_{i'}A_{i-1}^a
$$
As the exponent $e(i',i')-e(i,i')$ is nonzero, this implies that ($S_i$) holds.

By induction, we have  ($S_{d-1}$). To deduce ($S_d$) from it, we may assume that $I$ is nonempty, and 
let $i$ be the largest element of $I$. Then  ($S_{d-1}$) and ($T_i$) imply ($S_d$). 
\end{pf}

\begin{sbprop}\label{kep2}

  Let $c_1, c_2\in \R_{>0}$ and $a \in \bar X_{F,S}$. 
Let $I$ be the parabolic type (\ref{parat}) of $a$. 
Fix a finite subset $R$ of $G(F)$ and $1 \le i \le d-1$.
\begin{enumerate}
    \item[(1)] If $i\in I$, then for any $\epsilon > 0$, there exists a neighborhood $U$ of $a$ in $\bar X_{F,S}$ for the Satake topology  	such that $\max\{ t_{v,i}(x) \mid v \in S\} <\epsilon$ for all $x \in (\Gamma_KR)^{-1}U \cap \frak{S}(c_1,c_2)$.
    \item[(2)] If $i\notin I$, then there exist a neighborhood $U$ of $a$ in $\bar X_{F,S}$ for the Satake topology and 
    $c\in \R_{>0}$ such that $\min\{ t_{v,i}(x) \mid v \in S \} \geq  c$ for all $x \in (\Gamma_KR)^{-1}U \cap \frak{S}(c_1,c_2)$.
\end{enumerate}
\end{sbprop}

\begin{pf} The first statement is clear by continuity of $t_{v,i}$ and the fact that $t_{v,i}(\gamma^{-1}a) = 0$ for all
$\gamma \in G(F)$, and the second follows from \ref{kep1}, noting \ref{red21}. 
\end{pf}

\begin{sbprop}\label{lemcc} Let $a\in \bar X_{F,S}$, and let $P$ be the parabolic subgroup of $\PGL_V$ associated to $a$. 
Let $\Gamma_{K,(P)}\subset \Gamma_K$ be as in \ref{top7}. Then  there are
 $c_1,c_2\in \R_{>0}$ and $\varphi\in G(F)$ such that   $\Gamma_{K,(P)} \varphi \frak S(c_1, c_2)$ is a neighborhood of $a$ in $\bar X_{F,S}$ for the Satake topology.  
 
\end{sbprop}

\begin{pf} This holds by definition of the Satake topology with $\varphi=1$ if $a\in \bar X_{F,S}(B)$. In general, let $I$ be the parabolic type of $a$. Then the parabolic subgroup associated to $a$ has the form $\varphi P_I\varphi^{-1}$ for some $\varphi\in G(F)$. We have $\varphi^{-1}a \in \bar X_{F,S}(P_I)\subset \bar X_{F,S}(B)$.  By that already
proven case, there exists $\gamma \in \Gamma_{K,(P_I)}$ such that $\Gamma_{K,(P)}\varphi\gamma\frak S(c_1,c_2)$ is a neighborhood
of $a$ for the Satake topoology.
\end{pf}

The following result can be proved in the manner of \ref{lem9} for $\bar X_{F,S}^{\flat}$ below, replacing $R$ by $\{\varphi\}$, and $\Gamma_{K,(W)}$ by $\Gamma_{K,(P)}$.

\begin{sblem}\label{lem8} Let the notation be as in \ref{lemcc}. Let $U'$ be a neighborhood of $\varphi^{-1}a$ in $\bar X_{F,S}$ for the Satake topology. Then there is a neighborhood $U$ of $a$ in $\bar X_{F,S}$ for the Satake topology such that
$$U\subset \Gamma_{K,(P)} \varphi (\frak S(c_1,c_2) \cap U').$$
\end{sblem}

\subsection{$\bar X_{F,S}^{\flat}$ and reduction theory}\label{pflat}

\begin{sbpara} Let $S$ be a finite set of places of $F$ containing the archimedean places.  In this subsection, we consider $\bar X_{F,S}^{\flat}$. Fix a basis $(e_i)_{1\leq i\leq d}$ of $V$. Let $B\subset G=\PGL_V$ be the Borel subgroup of upper triangular matrices for $(e_i)_i$. Let $K$ be a compact open subgroup of 
$G(\mb{A}_F^S)$.

\end{sbpara}

\begin{sbpara}\label{kep9}  

Let $c_1,c_2\in \R_{>0}$.  We let $\frak S^{\flat}(c_1,c_2)$ denote the image of $\frak S(c_1,c_2)$ under $\bar X_{F,S}\to \bar X_{F,S}^{\flat}$. 
For $r\in \{1,\dots,d-1\}$, we then define 
$$
	\frak S^{\flat}_r(c_1,c_2) = \{ (W, \mu) \in \frak S^{\flat}(c_1,c_2) \mid \dim(W) \geq r \}.
$$ 
Then the maps $t_{v,i}$ of \ref{tvi} for $v\in S$ and $1\leq i\leq r$ induce maps
\begin{eqnarray*}
	t_{v,i} \colon \frak S^{\flat}_r(c_1,c_2) \to \R_{>0} \ \ (1\leq i\leq r-1) &\text{and}&
	t_{v,r} \colon \frak S^{\flat}_r(c_1,c_2) \to \R_{\geq 0}.
\end{eqnarray*}
For $c_3\in \R_{>0}$, we also set 
$$
 	\frak S_r^{\flat}(c_1,c_2,c_3) = \{ x \in \frak S^{\flat}_r(c_1,c_2) \mid \min\{ t_{v,r}(x) \mid v \in S\} \leq c_3 \}.
$$

\end{sbpara}

\begin{sbprop}\label{flat5}

 Fix $c_1 \in \R_{>0}$ and finite subsets $R_1, R_2$ of $G(F)$. Then there exists $c_3\in \R_{>0}$ such that
 for all $c_2 \in \R_{>0}$, we have
 $$
 	\{\gamma \in R_1\Gamma_KR_2 \mid \gamma \frak S^{\flat}_r(c_1,c_2,c_3) \cap \frak S^{\flat}_r(c_1,c_2)
	\neq \varnothing\} \subset P_{\{r\}}.
 $$
 \end{sbprop}

\begin{pf} 
Take $(W,\mu) \in  \gamma \frak S^{\flat}_r(c_1,c_2,c_3) \cap \frak S^{\flat}_r(c_1,c_2)$, and let $r' = \dim W$.  
Let $P$ be the parabolic subgroup of $V$ corresponding to the
flag $(V_i)_{-1\le i \le d-r'}$ with $V_i = W + \sum_{j=r'+1}^{r'+i} F e_i$ for $0 \le i \le d-r'$.  
Let $\mu' \in \frak{Z}_{F,S}(P)$ be the unique element such that $a = (P,\mu') \in \bar X_{F,S}$ maps to $(W,\mu)$. Then $a \in  \gamma \frak S_{\{r\}}(c_1,c_2,c_3) \cap \frak S(c_1,c_2)$, so we can apply \ref{red2cc}.
\end{pf}

\begin{sbprop}\label{kep3}

 Fix $c_1, c_2\in \R_{>0}$ and finite subsets $R_1, R_2$ of $G(F)$.
Then  
 there exists $A>1$ such that if $x \in \frak S^{\flat}_r(c_1, c_2) \cap \gamma^{-1}\frak S^{\flat}_r(c_1,c_2)$ for some $\gamma\in R_1\Gamma_K R_2$, then
 $$
 	A^{-1}t_{v,i}(x)\leq t_{v,i}(\gamma x) \leq At_{v,i}(x)
$$ 
for all $v\in S$ and $1\leq i\leq r$.

\end{sbprop}

\begin{pf}
This follows from \ref{kep1}. 
\end{pf}

We also have the following easy consequence of Lemma \ref{kep20}.

\begin{sblem}\label{kep80}  
Let $a$ be in the image of $\bar X_{F,S}(B) \to \bar X_{F,S}^{\flat}$, and let $r$ be the dimension of the $F$-subspace of $V$ associated to $a$.  Then the $F$-subspace of $V$ associated to $a$ is $\sum_{i=1}^r Fe_i$. 

\end{sblem}

\begin{sbprop}\label{kep8} Let $a\in \bar X_{F,S}^{\flat}$ and let $r$ be the dimension of the $F$-subspace of $V$ associated to $a$. Let $c_1,c_2\in \R_{>0}$.
Fix a finite subset $R$ of $G(F)$.
\begin{enumerate}
	\item[(1)]  For any $\epsilon >0$, there exists a neighborhood $U$ of $a$ in $\bar X_{F,S}^{\flat}$ for the Satake topology 	such that $\max\{t_{v,r}(x) \mid v \in S\} <\epsilon$ for all $x\in (\Gamma_K R)^{-1}U \cap \frak S^{\flat}_r(c_1,c_2)$.
	\item[(2)] If $1\leq i <r$, then there exist a neighborhood $U$ of $a$ in $\bar X_{F,S}^{\flat}$ for the Satake topology and $c\in \R_{>0}$ such that $\min\{t_{v,i}(x) \mid v \in S\} \ge c$ for all $x\in (\Gamma_K R)^{-1}U  \cap \frak S^{\flat}_r(c_1,c_2)$.
\end{enumerate}
\end{sbprop}

\begin{pf}
	This follows from \ref{kep3}, as in the proof of \ref{kep2}.
\end{pf}
 
\begin{sbprop}\label{c12R} 
Let $W$ be an $F$-subspace of $V$ of dimension $r \ge 1$.  Let $\Phi$ be set of $\varphi \in G(F)$
such that $\varphi(\sum_{i=1}^r Fe_i) = W$.
\begin{enumerate}
	\item[(1)] There exists a finite subset $R$ of $\Phi$ such that for any $a \in \bar X_{F,S}^{\flat}(W)$,
	there exist $c_1,c_2\in \R_{>0}$ for which the set $\Gamma_{K,(W)} R\frak S^{\flat}_r(c_1,c_2)$ is a neighborhood 
	of $a$ in the Satake topology.  
 	\item[(2)]  For any $\varphi \in \Phi$ and $a \in \bar X_{F,S}^{\flat}$ with associated subspace
	$W$, there exist $c_1,c_2\in \R_{>0}$ such that $a \in \varphi \frak S^{\flat}_r(c_1,c_2)$
	and  $\Gamma_{K,(W)}\varphi\frak S^{\flat}_r(c_1,c_2)$ is a neighborhood 
	of $a$ in the Satake topology.  
\end{enumerate}
\end{sbprop}

\begin{pf} 
We may suppose without loss of generality that $W=\sum_{j=1}^r Fe_j$, in which case $\Phi = G(F)_{(W)}$
(see \ref{top7}).
Consider the set $\cal Q$ of all parabolic subgroups $Q$ of $G$ such that $W$ is contained in the smallest
nonzero subspace of $V$ preserved by $Q$.  Any $Q \in \cal Q$ has the form $Q = \varphi P_I \varphi^{-1}$ for some
$\varphi \in G(F)_{(W)}$ and subset $I$ of $J := \{i\in \Z\mid r\leq i\leq d-1\}$.
There exists a finite subset $R$ of $G(F)_{(W)}$ 
such that we may always choose $\varphi \in \Gamma_{K,(W)}R$.  

By \ref{kep80}, an element of $\bar X_{F,S}(B)$ has image in $\bar X_{F,S}^{\flat}(W)$ if and only if 
the parabolic subgroup associated to it has the form $P_I$ for some $I \subset J$.
The intersection of the image of $\bar X_{F,S}(B) \to \bar X_{F,S}^{\flat}$ with 
$\bar X_{F,S}^{\flat}(W)$ is the union of the $\frak S^{\flat}_r(c_1,c_2)$ with $c_1,c_2\in \R_{>0}$.  
By the above, for any $a \in \bar X_{F,S}^{\flat}(W)$, we may choose $\xi \in \Gamma_{K,(W)}R$ such that
$\xi^{-1}a$ is in this intersection, and part (1) follows.
Moreover, if $W$ is the subspace associated
 to $a$, then $\varphi^{-1}a \in \bar{X}_{F,S}^{\flat}(W)$ is in the image of $\bar X_{F,S}(B)$ 
 for all $\varphi \in G(F)_{(W)}$, from which (2) follows.
\end{pf}

\begin{sblem}\label{lem9} Let $W$, $\Phi$, $R$ be as in \ref{c12R}, fix $a \in \bar{X}_{F,S}^{\flat}(W)$, and let
$c_1, c_2 \in \R_{>0}$ be as in \ref{c12R}(1) for this $a$.
For each $\varphi\in R$, let $U_{\varphi}$ be a neighborhood of $\varphi^{-1}a$ in $\bar X_{F,S}^{\flat}$ for the Satake topology. Then there is a neighborhood $U$ of $a$ in $\bar X_{F,S}^{\flat}$ for the Satake topology such that
$$U\subset \bigcup_{\varphi \in R} \Gamma_{K,(W)} \varphi(\frak S^{\flat}_r(c_1,c_2) \cap U_{\varphi}).$$
\end{sblem}

\begin{pf} We may assume that each $\varphi(U_{\varphi})$ is stable under the action of $\Gamma_{K,(W)}$. Let 
$$
U= \Gamma_{K,(W)}R\frak S_r^{\flat}(c_1,c_2) \cap \bigcap_{\varphi\in R} \varphi(U_{\varphi}).
$$ Then $U$ is a neighborhood of $a$ by  \ref{c12R}(1). Let $x\in U$. Take $\gamma \in \Gamma_{K,(W)}$ and $\varphi\in R$ such that  $x\in \gamma \varphi\frak 
S_r^{\flat}(c_1,c_2)$. Since $\varphi(U_{\varphi})$ is $\Gamma_{K,(W)}$-stable, $\gamma^{-1}x \in \varphi(U_{\varphi})$ and hence $\varphi^{-1}\gamma^{-1}x \in \frak S_r^{\flat}(c_1, c_2)\cap U_{\varphi}$.
\end{pf}

\begin{sbprop}\label{flat4} Let $a=(W,\mu)\in \bar X_{F,S}^{\flat}$, and let $r=\dim(W)$. Take $\varphi \in G(F)$ and $c_1,c_2
\in \R_{>0}$ as in \ref{c12R}(2) such that $\Gamma_{K,(W)}\varphi\frak{S}_r^{\flat}(c_1,c_2)$ is a neighborhood of $a$.
Let $\phi_{W,S}^{\flat} \colon \bar X_{F,S}^{\flat}(W) \to \frak{Z}_{F,S}^{\flat}(W)$ be as in \ref{toW}.
For any neighborhood $U$ of $\mu = \phi_{W,S}^{\flat}(a)$ in $\frak{Z}_{F,S}^{\flat}(W)$ and any $\epsilon\in \R_{>0}$, set 
$$
	\Phi(U, \epsilon) =  (\phi_{W,S}^{\flat})^{-1}(U) \cap \Gamma_{K,(W)}\varphi 
	\{ x \in \frak S^{\flat}_r(c_1,c_2) \mid t_{v,r}(x) <\epsilon \text{ for all } v \in S \}.
$$
Then the set of all $\Phi(U,\epsilon)$ forms a base of neighborhoods of $a$ in $\bar X_{F,S}^{\flat}$ under the
Satake topology.
\end{sbprop}

\begin{proof}
    We may suppose that $W = \sum_{i=1}^r Fe_i$ without loss of generality, in which case $\varphi \in G(F)_{(W)}$.
    Let $P$ be the smallest parabolic subgroup containing $B$ with flag $(V_i)_{-1 \le i \le m}$ such that $V_0 = W$
    and $m = d-r$.  Let $Q$ be the parabolic of
    all elements that preserve $W$.  We then have $G\supset Q \supset P\supset B$.
    Let $B'$ be the Borel subgroup of $\PGL_{V/W}$ that is the image of $P$ and which we regard as a subgroup of $G$ using $(e_{r+i})_{1 \le i \le m}$ to split $V \to V/W$.
    
    Let 
    $$
    	f_v \colon Q_u(F_v) \times \bar X_{V/W,F,v}(B') \times X_{W_v} \times \R_{\geq 0} \to \bar X_{F,v}(P)
    $$
    be the unique surjective continuous map such that $\xi = f_v \circ h$, where
    $\xi$ is as in \ref{xidef} and $h$ is defined as the composition
    $$
    	P_u(F_v) \times X_{W_v} \times \R_{\geq 0}^m
	\xrightarrow{\sim} Q_u(F_v) \times B'_u(F_v)\times X_{W_v} \times \R_{\geq 0}\times \R_{\geq 0}^{m-1}
	\to  Q_u(F_v)\times   
    	\bar X_{V/W,F,v}(B')\times X_{W_v} \times \R_{\geq 0}
    $$
    of the map induced by the isomorphism $P_u(F_v) \xrightarrow{\sim} Q_u(F_v) \times B'_u(F_v)$
     and the map induced by the surjection 
    $\bar{\pi}_{B',v} \colon B'_u(F_v)\times \R_{\geq 0}^{m-1}\to \bar X_{V/W,F,v}(B')$
    of \ref{lst0}(2). 
    The existence of $f_v$ follows from \ref{lst0}(4).  
    
     Set $Y_0 = \R_{>0}^S \cup \{(0)_{v \in S}\}$, and let
     $$
     	f_S \colon Q_u(\mathbb{A}_{F,S}) \times \bar X_{V/W,F,S}(B') \times 
    	\frak Z_{F,S}^{\flat}(W) \times Y_0 \to \bar X_{F,S}(P)
     $$
     be the product of the maps $f_v$.
    Let $t_{v,r} \colon \bar X_{F,v}(P) \to \R_{\geq 0}$ denote the composition 
    $$
    	\bar X_{F, v}(P) \to \bar X_{F,v}(B) \xrightarrow{\phi'_{B,v}} \R_{\geq 0}^{d-1} \to \R_{\geq 0},
    $$ 
    where the last arrow is the $r$th projection.  
    The composition of $f_S$ with $(t_{v,r})_{v \in S}$ is projection onto $Y_0$
     by \ref{keyp} and \ref{togr2}.     
    
    Let $\bar X_{F,S}(W)$ denote the inverse image 
    of $\bar X^{\flat}_{F,S}(W)$ under the canonical surjection $\Pi_S \colon \bar X_{F,S} \to \bar X_{F,S}^{\flat}$. 
    Combining $f_S$ with the action of $G(F)_{(W)}$, we obtain a surjective map
    $$
    	f'_S \colon G(F)_{(W)} \times (Q_u(\mathbb{A}_{F,S}) \times \bar X_{V/W,F,S}(B') \times 
    	\frak Z_{F,S}^{\flat}(W) \times Y_0) \to \bar X_{F,S}(W), \qquad f'_S(g,z) = gf_S(z).
    $$
    The composition of $f'_S$ with $\phi^{\flat}_{W,S} \circ \Pi_S$ is projection onto  $\frak{Z}_{F,S}^{\flat}(W)$ by 
    \ref{xi2} and \ref{xixi}.

    Applying \ref{red21} with $V/W$ in place of $V$, there exists a compact subset $C$ of $Q_u(\mb{A}_{F,S}) \times \bar X_{V/W,F,S}(B')$ and a finite subset $R$ of $G(F)_{(W)}$ 
    such that $f'_S(\Gamma_{K,(W)}R \times C \times \frak Z_{F,S}^{\flat}(W) \times Y_0) = \bar{X}_{F,S}(W)$. 
    Consider the restriction of $\Pi_S \circ f'_S$ to a surjective map 
    $$
    	\lambda_S \colon \Gamma_{K,(W)}R \times C \times \frak Z_{F,S}^{\flat}(W) \times Y_0 \to \bar{X}_{F,S}^{\flat}(W).
    $$
    We may suppose that $R$ contains $\varphi$, since it lies in $G(F)_{(W)}$.

    Now, let $U'$ be a neighborhood of $a$ in $\bar X^{\flat}_{F,S}(W)$ for the Satake topology.
    It is sufficient to prove that 
    there exist an open neighborhood $U$ of $\mu$ in $\frak Z_{F,S}^{\flat}(W)$ and $\epsilon\in \R_{>0}$ 
    such that $\Phi(U, \epsilon) \subset U'$.  For $\epsilon \in \R_{>0}$, set 
    $Y_{\epsilon} = \{ (t_v)_{v \in S} \in Y_0 \mid t_v < \epsilon \text{ for all } v \in S \}$.
  
    For any $x \in C$, we have $\lambda_S(\alpha,x,\mu,0) = (W,\mu) \in U'$ for all $\alpha \in R$.
    By the continuity of $\lambda_S$, there exist 
    a neighborhood $D(x) \subset Q_u(\mathbb{A}_{F,S}) \times \bar X_{V/W,F,S}(B')$ of $x$, a neighborhood 
    $U(x) \subset \frak{Z}_{F,v}^{\flat}(W)$ of $\mu$, and $\epsilon(x)\in \R_{>0}$ such that  
    $$
    	\lambda_S(R \times D(x) \times U(x) \times Y_{\epsilon(x)}) \subset U'.
    $$ 
    Since $C$ is compact, some finite collection of
    the sets $D(x)$ cover $C$.  Thus, 
    there exist a neighborhood $U$ of $\mu$ in $\frak{Z}_{F,v}^{\flat}(W)$ and $\epsilon \in \R_{>0}$ such that 
    $\lambda_S(R \times C\times U \times Y_{\epsilon}) \subset U'$.
    Since $U'$ is $\Gamma_{K,(W)}$-stable by \ref{top2}, we have 
    $\lambda_S(\Gamma_{K,(W)}R \times C\times U \times Y_{\epsilon}) \subset U'$.
     
    Let $y \in \Phi(U,\epsilon)$, and write $y= g x$ with $g \in \Gamma_{K,(W)}\varphi$
    and $x \in \frak S^{\flat}_r(c_1,c_2)$ such that $t_{v,r}(x) <\epsilon$ for all $v \in S$.  
    Since $\Phi(U,\epsilon) \subset \bar{X}_{F,S}^{\flat}(W)$,
    we may by our above remarks write $y = \lambda_S(g,c,\nu,t) = g\Pi_S(f_S(c,\nu,t))$, where
    $c \in C$, $\nu = \phi^{\flat}_{W,S}(y)$, and $t = (t_{v,r}(x))_{v \in S}$.  Since $\nu \in U$ and $t \in Y_{\epsilon}$
    by definition, $y$ is contained in $U'$.  Therefore, we have $\Phi(U,\epsilon) \subset U'$.
\end{proof}

\begin{sbeg} \label{flatex}  Consider the case $F=\Q$, $S=\{v\}$ with $v$ the archimedean place, and $d=3$. 
We construct a base of neighborhoods of a point in $\bar X_{\Q,v}^{\flat}$ for the Satake topology.

Fix a basis $(e_i)_{1\leq i\leq 3}$ of $V$. Let $a=(W, \mu)\in \bar X_{\Q,v}^{\flat}$, where $W=\Q e_1$, and 
$\mu$ is the unique element of $X_{W_v}$. 

For $c\in \R_{>0}$, let $U_c$ be the subset of $X_v=\PGL_3(\R)/\PO_3(\R)$ consisting of the elements
$$\begin{pmatrix} 1 &0\\ 0 & \gamma \end{pmatrix} \begin{pmatrix} 1&x_{12}&x_{13}\\
0&1&x_{23}\\
0&0&1 \end{pmatrix}\begin{pmatrix} y_1y_2& 0& 0\\0 & y_2& 0\\ 0&0&1\end{pmatrix}$$
 such that 
$\gamma \in \PGL_2(\Z)$, $x_{ij}\in \R$, $y_1\geq c$, and $y_2\geq \frac{\sqrt{3}}{2}$.
When $\gamma$, $x_{ij}$ and $y_2$ are fixed and $y_1\to \infty$, these elements converge to $a$ in $\bar X_{\Q,v}^{\flat}$ under the Satake topology. When $\gamma$, $x_{ij}$, and $y_1$ are fixed and $y_2\to \infty$, they converge in
the Satake topology to 
$$\mu(\gamma, x_{12}, y_1):= \begin{pmatrix} 1 &0\\ 0 & \gamma \end{pmatrix} \begin{pmatrix} 1&x_{12}&0\\
0&1&0\\
0&0&1 \end{pmatrix}\mu(y_1),$$
where $\mu(y_1)$ is the class in $\bar X_{\Q,v}^{\flat}$ of 
the semi-norm $a_1e_1^*+a_2e_2^*+a_3e_3^*\mapsto (a_1^2y_1^2+a_2^2)^{1/2}$ on $V_v^*$.

The set of  
$$\bar U_c := \{a\}\cup \{\mu(\gamma, x, y)\mid \gamma \in \PGL_2(\Z), x \in \R, y\geq c\}\cup U_c.$$ 
is a base of neighborhoods for $a$ in $\bar X_{\Q,v}^{\flat}$ under the Satake topology.
Note that $\frak H=\SL_2(\Z)\{z\in \frak H\mid\text{Im}(z)\geq \frac{\sqrt{3}}{2}\}$, which is the reason
for the appearance of $\frac{\sqrt{3}}{2}$. 
It can of course be replaced by any $b \in \R_{>0}$ such that $b \leq \frac{\sqrt{3}}{2}$. 

\end{sbeg}

\begin{sbpara}\label{flatex2} We continue with Example \ref{flatex}.  Under the canonical surjection $\bar X_{\Q,v}\to \bar X_{\Q,v}^{\flat}$, the inverse image of $a=(W, \mu)$ in $\bar X_{\Q,v}$ is canonically homeomorphic to $\bar X_{(V/W)_v}=\frak H\cup \mb{P}^1(\Q)$ under the Satake topology on both spaces. This homeomorphism sends 
$x+y_2i\in \frak H$ ($x\in \R, y_2\in \R_{>0}$) to the limit for the Satake topology of 
$$\begin{pmatrix} 1 & 0&0\\0&1&x\\
0&0&1\end{pmatrix}\begin{pmatrix} y_1y_2 &0&0\\
0&y_2&0\\
0&0&1\end{pmatrix} \in \PGL_3(\R)/\PO_3(\R)$$ as $y_1\to \infty$.  (This limit
 in $\bar X_{\Q,v}$ depends on $x$ and $y_2$, but the limit in $\bar X_{\Q,v}^{\flat}$ is $a$.) 
 \end{sbpara}

\begin{sbpara}\label{flat8} In the example of \ref{flatex}, 
we explain that the quotient topology on $\bar X_{\Q,v}^{\flat}$ of the Satake topology on $\bar X_{\Q,v}$ is different from the Satake topology on $\bar X_{\Q,v}^{\flat}$.

For a map 
$$f \colon \PGL_2(\Z)/\begin{pmatrix} 1& \Z\\ 0&1 \end{pmatrix} \to \R_{>0},$$
define a subset $U_f$ of $X_v$ as in the definition of $U_c$ but replacing the condition on $\gamma, x_{ij}, y_i$ by
$\gamma \in \PGL_2(\Z)$, $x_{ij}\in \R$, $y_1\geq f(\gamma)$, and $y_2\geq \frac{\sqrt{3}}{2}$.
Let $$\bar U_f= \{a\}\cup \{\mu(\gamma, x, y) \mid \gamma \in \PGL_2(\Z), x \in \R, y\geq f(\gamma)\}\cup U_f.$$ When $f$ varies, the $\bar U_f$ form a base of neighborhoods of $a$ in $\bar X_{\Q,v}^{\flat}$ for the quotient topology of the Satake topology on $\bar X_{\Q,v}$. 
On the other hand, if $\inf\{f(\gamma)\mid \gamma\in \PGL_2(\Z)\}=0$, then $\bar U_f$ is not a neighborhood of $a$ for the Satake topology on $\bar X_{\Q,v}^{\flat}$. 

\end{sbpara}

\subsection{Proof of the main theorem} 

In this subsection, we prove Theorem \ref{main}.  We begin with the quasi-compactness asserted therein.
Throughout this subsection, we set 
$Z= X_{S_2}\times G(\mb{A}_F^S)/K$ in situation (I) and $ Z =X_{S_2}$ in situation (II), so $\bar{\frak X} = \bar{X} \times Z$.

\begin{sbprop}\label{prop6} In situation (I) of \ref{mains3}, the quotient $G(F) \bs \bar {\frak X}$ is quasi-compact. In situation (II), the quotient $\Gamma \bs \bar {\frak X}$ is quasi-compact for any subgroup $\Gamma$ of $\Gamma_K$ of finite index. 

\end{sbprop}

\begin{pf} We may restrict to case (i) of \ref{mains1} that $\bar{X} = \bar{X}_{F,S_1}$, as $\bar{X}_{F,S_1}^{\flat}$ of
case (ii) is a quotient of $\bar{X}_{F,S_1}$ (under the Borel-Serre topology). 
In situation (I), we claim that there exist $c_1,c_2\in \R_{>0}$, a compact subset $C$ of $B_u(\mb{A}_{F,S})$, and a compact subset $C'$ of $Z$ such that $\bar {\frak X}= G(F)(\frak S(C;c_1,c_2) \times C')$. In situation (II), we claim that there exist $c_1,c_2, C, C'$ as above and a finite subset $R$ of $G(F)$ such that $\bar {\frak X}= \Gamma R (\frak S(C;c_1,c_2) \times C')$.  It follows that in situation (I) (resp., (II)), there is a surjective continuous map from 
 the compact space $C\times \frak T(c_1, c_2)\times C'$ (resp., $R\times C\times \frak T(c_1, c_2)\times C'$) onto the quotient space under consideration, which yields the proposition.

For any compact open subgroup $K'$ of $G(\mb{A}_F^{S_1})$, the set $G(F)\bs G(\mb{A}_F^{S_1})/K'$ is finite. 
Each $X_v$ for $v \in S_2$ may be identified with the geometric realization of the Bruhat-Tits building for $\PGL_{V_v}$,
the set of $i$-simplices of which for a fixed $i$ can be identified with $G(F_v)/K'_v$ for some $K'$.
So, we see that in situation (I) (resp., (II)), there is a compact subset $D$ of $Z$ such that $Z= G(F) D$ (resp., $Z= \Gamma D$).  

Now fix such a compact open subgroup $K'$ of $G(\mb{A}_F^{S_1})$.
By \ref{red21}, there are $c_1,c_2\in \R_{>0}$, a compact subset $C$ of $P_u(\mb{A}_{F,S_1})$, and a finite subset $R'$ of $G(F)$ such that $\bar X_{F,S} = \Gamma_{K'} R' \frak S(C;c_1,c_2)$. 
We consider the compact subset $C':= (R')^{-1}K' D$ of $Z$.

Let $(x,y) \in \bar{\frak X}$, where $x\in \bar X_{F,S}$ and $y\in Z$.
Write $y=\gamma z$ for some  $z\in D$ and $\gamma \in G(F)$ (resp., $\gamma \in \Gamma$) in situation (I) (resp., (II)). 
In situation (II), we write $\Gamma \Gamma_{K'} R'= \Gamma R$ for some finite subset $R$ of $G(F)$.   Write $\gamma^{-1}x = \gamma' \varphi s$ where $\gamma'\in \Gamma_{K'}$, $\varphi \in R'$, $s\in \frak S(C;c_1,c_2)$. We have 
  $$
  (x,y)= \gamma (\gamma^{-1}x, z)= 
\gamma (\gamma' \varphi s, z)= (\gamma \gamma' \varphi)(s, \varphi^{-1}(\gamma')^{-1} z).
$$
As $\gamma \gamma' \varphi$ lies in $G(F)$ in situation (I) and in $\Gamma R$ in situation (II), we have the claim.
  \end{pf}

\begin{sbpara}
To prove Theorem \ref{main}, it remains only to verify the Hausdorff property.  For this, 
it is sufficient to prove the following.
\end{sbpara}

\begin{sbprop}\label{ga=b3} Let $\Gamma = G(F)$ in situation (I) of \ref{mains3}, 
and let $\Gamma = \Gamma_K$ in situation (II). 
For every $a, a'\in \bar {\frak X}$, there exist neighborhoods $U$ of $a$ and $U'$ of $a'$ such that
if $\gamma\in \Gamma$ and $\gamma U \cap U'\neq \varnothing$, then $\gamma a=a'$. 

\end{sbprop}

In the rest of this subsection, let the notation be as in \ref{ga=b3}. It is sufficient to prove \ref{ga=b3} for the Satake topology on $\bar{\frak X}$.  In \ref{r=s1}--\ref{pfmain},  we prove  \ref{ga=b3} in situation (II) for $S=S_1$.  That is, we suppose that $\bar{\frak X} = \bar X$.
 In \ref{ga=b2} and \ref{ga=b4}, we deduce \ref{ga=b3} in general from this case. 

\begin{sblem}\label{r=s1} Assume that $\bar{\frak X}=\bar X_{F,S}$. Suppose that $a,a'\in \bar {\frak X}$ have distinct parabolic types (\ref{parat}).  Then there exist neighborhoods $U$ of $a$ and $U'$ of $a'$ such that $\gamma U\cap U'=\varnothing$ for all $\gamma \in \Gamma$. 

\end{sblem}

\begin{pf} 
Let $I$ (resp., $I'$) be the parabolic type of $a$ (resp., $a'$).  We may assume that 
there exists an $i\in I$ with $i\notin I'$. 

By \ref{lemcc}, there exist 
$\varphi, \psi\in G(F)$ and $c_1,c_2\in \R_{>0}$ such that 
$\Gamma_K\varphi \frak S(c_1,c_2)$ is a neighborhood of $a$ and $\Gamma_K \psi \frak S(c_1,c_2)$ is a neighborhood of $a'$. By \ref{kep2}(2), there exist a neighborhood 
$U'\subset \Gamma_K \psi \frak S(c_1,c_2)$ of $a'$ 
and $c\in \R_{>0}$ with the property that $\min\{t_{v,i}(x) \mid v \in S\} \geq c$ 
for all $x\in (\Gamma_K \psi)^{-1}U' \cap \frak S(c_1,c_2)$.  Let $A\in \R_{>1}$ be as in \ref{kep1}
for these $c_1,c_2$ for $R_1=\{\varphi^{-1}\}$ and $R_2= \{\psi\}$. Take $\epsilon \in \R_{>0}$ such that $A\epsilon \leq c$. By \ref{kep2}(1), there exists a neighborhood $U\subset \Gamma_K \varphi \frak S(c_1,c_2)$ of $a$ 
such that $\max\{ t_{v,i}(x) \mid v \in S\} <\epsilon$ for all $x\in (\Gamma_K \varphi)^{-1}U \cap \frak S(c_1,c_2)$.  

We prove that $\gamma U \cap U'=\varnothing$ for all $\gamma \in \Gamma_K$.  If $x\in \gamma U\cap U'$, then we
may take $\delta, \delta'\in \Gamma_K$ such that $(\delta \varphi)^{-1}\gamma^{-1}x\in \frak S(c_1,c_2)$ and $(\delta' \psi)^{-1} x\in \frak S(c_1,c_2)$.  Since 
$$
	(\delta \varphi)^{-1}\gamma^{-1} x= \varphi^{-1} (\delta^{-1} \gamma^{-1}\delta')\psi (\delta'\psi)^{-1}x
	\in \varphi^{-1}\Gamma_K\psi \cdot (\delta'\psi)^{-1}x,
$$ 
we have by \ref{kep1} that
$$c\leq t_{v,i}((\delta'\psi)^{-1}x)\leq A t_{v,i}((\delta\varphi)^{-1}\gamma^{-1}x) < A \epsilon,$$
for all $v\in S$  and hence $c < A\epsilon$, a contradiction. 
\end{pf}

\begin{sblem}\label{r=s2} Assume that $\bar{\frak X}=\bar X^{\flat}_{F,S}$. Let $a,a'\in \bar {\frak X}$ and assume that the dimension of the $F$-subspace associated to $a$ is different from that of $a'$. 
Then there exist neighborhoods $U$ of $a$ and $U'$ of $a'$ such that $\gamma U\cap U'=\varnothing$ for all $\gamma \in \Gamma$. 
\end{sblem}

\begin{pf} The proof is similar to that of \ref{r=s1}. In place of  \ref{kep1}, \ref{kep2}, and \ref{lemcc}, we use
 \ref{kep3},  \ref{kep8}, and \ref{c12R}, respectively. 
\end{pf}

\begin{sblem}\label{ga=b1} 
Let $P$ be a parabolic subgroup of $G$. Let $a,a'\in \frak{Z}_{F,S}(P)$ (see \ref{togr}), and let $R_1$ and $R_2$ be finite subsets of $G(F)$. 
Then there exist neighborhoods $U$ of $a$ and  $U'$ of $a'$ in $\frak{Z}_{F,S}(P)$ such that $\gamma a = a'$
for every $\gamma\in R_1\Gamma_K R_2\cap P(F)$ for which $\gamma U\cap U'\neq \varnothing$.

\end{sblem}

\begin{pf} For each $\xi\in R_1$ and $\eta\in R_2$, the set $\xi \Gamma_K \eta\cap P(F)$ is a $\xi\Gamma_K\xi^{-1} \cap P(F)$-orbit for the left action of $\xi \Gamma_K \xi^{-1}$. 
Hence its image in $\prod_{i=0}^m \PGL_{V_i/V_{i-1}}(\mb{A}_{F,S})$ is discrete, for $(V_i)_{-1\le i \le m}$ the flag corresponding to $P$, and thus the image of $R_1\Gamma_K R_2\cap P(F)$ in $ \prod_{i=0}^m \PGL_{V_i/V_{i-1}}(\mb{A}_{F,S})$ is discrete as well. 
On the other hand, for any compact neighborhoods $U$ of $a$ and $U'$ of $a'$, the set
$$
	\left\{g\in \prod_{i=0}^m \PGL_{V_i/V_{i-1}}(\mb{A}_{F,S})\mid gU\cap U'\neq \varnothing\right\}
$$ 
is compact. Hence the intersection $M :=\{\gamma \in R_1\Gamma_K R_2\cap P(F) \mid \gamma U\cap U'\neq \varnothing\}$ is finite. If $\gamma \in M$ and $\gamma a\neq a'$, then replacing $U$ and $U'$ by smaller neighborhoods of $a$ and $a'$, respectively, we have $\gamma U \cap U'=\varnothing$. Hence for sufficiently small neighborhoods $U$ and $U'$ of $a$ and $a'$, respectively, we have that if $\gamma \in M$, then $\gamma a=a'$.
\end{pf}

\begin{sblem}\label{ga=b12} 
Let $W$ be an $F$-subspace of $V$.  Let $a,a'\in \frak{Z}_{F,S}^{\flat}(W)$ (see \ref{toW}), and let $R_1$ and $R_2$ be finite subsets of $G(F)$. Let $P$ be the parabolic subgroup of $G$ consisting of all elements which preserve $W$. 
Then there exist neighborhoods $U$ of $a$ and $U'$ of $a'$ in $\frak{Z}_{F,S}^{\flat}(W)$ such that $\gamma a =a'$
for every $\gamma\in R_1\Gamma_K R_2\cap P(F)$ for which $\gamma U\cap U'\neq \varnothing$.

\end{sblem}

\begin{pf} This is proven in the same way as \ref{ga=b1}.
\end{pf}

\begin{sbpara}\label{pfmain} We prove \ref{ga=b3} in situation (II), supposing that $S=S_1$.

In case (i) (that is, $\bar X=\bar {\frak X}= \bar X_{F,S}$), we may assume by
 \ref{r=s1} that $a$ and $a'$ have the same parabolic type $I$. 
 In case (ii) (that is, $\bar X=\bar {\frak X}= \bar X_{F,S}^{\flat}$), 
we may assume by \ref{r=s2} that the dimension $r$ of the $F$-subspace of $V$ associated to $a$ coincides with that of $a'$. In case (i) (resp., (ii)),
  take $c_1,c_2\in \R_{>0}$ and elements $\varphi$ and $\psi$ (resp., finite subsets $R$ and $R'$) of $G(F)$ 
such that $c_1,c_2,\varphi$ (resp., $c_1,c_2, R$) satisfy the condition in \ref{lemcc} (resp., \ref{c12R}) for $a$ and $c_1,c_2,\psi$ (resp., $c_1,c_2,R'$) satisfy the condition in \ref{lemcc} (resp.,  \ref{c12R}) for $a'$.
 In case (i), we set $R=\{\varphi\}$ and $R'=\{\psi\}$.

Fix a basis $(e_i)_{1\leq i\leq d}$ of $V$.  In case (i) (resp., (ii)), denote $\frak S(c_1,c_2)$ (resp., $\frak S^{\flat}_r(c_1,c_2)$) by $\frak S$. In case (i), let $P=P_I$, and let $(V_i)_{-1\le i \le m}$ be the associated flag. In case (ii), let $W=\sum_{i=1}^r Fe_i$, and let $P$ be the parabolic subgroup of $G$ consisting of all elements which preserve $W$. 

Note that in case (i) (resp., (ii)), for all $\varphi\in R$ and $\psi\in R'$, the parabolic subgroup $P$ is  associated to $\varphi^{-1}a$ and to $\psi^{-1}a'$ (resp., $W$ is associated to $\varphi^{-1}a$ and to $\psi^{-1}a'$) and hence these elements are determined by their images in $\frak{Z}_{F,S}(P)$ (resp. $\frak{Z}_{F,S}^{\flat}(W)$).  

In case (i) (resp., case (ii)), apply  \ref{ga=b1} (resp.,  \ref{ga=b12}) to the images of $\varphi^{-1}a$ and $\psi^{-1}a'$ for
$\varphi \in R$, $\psi \in R'$
in $\frak{Z}_{F,S}(P)$ (resp., $\frak{Z}_{F,S}^{\flat}(W)$). By this, and 
 by \ref{red2cc} for case (i) and \ref{flat5} for case (ii), we see that
 there exist neighborhoods $U_{\varphi}$ of $\varphi^{-1}a$ for each $\varphi\in R$ and $U'_{\psi}$ of $\psi^{-1}a'$ for each $\psi\in R'$ for the Satake topology with the following two properties:
\begin{enumerate}
\item[(A)] $\{ \gamma \in (R')^{-1}\Gamma_K R \mid \gamma(\frak S\cap U_{\varphi}) \cap (\frak S\cap U'_{\psi})\neq \varnothing \text{  for some }\varphi\in R, \psi\in R' \} \subset P(F)$,
\item[(B)] 
	if $\gamma \in (R')^{-1}\Gamma_K R \cap P(F)$ and $\gamma U_{\varphi}\cap U'_{\psi} \neq \varnothing$ for
	$\varphi\in R$ and $\psi\in R'$, then $\gamma \varphi^{-1}a =\psi^{-1}a'$.
\end{enumerate}

In case (i) (resp., (ii)), 
take a neighborhood $U$ of $a$ satisfying the condition in \ref{lem8} (resp., \ref{lem9}) for $(U_{\varphi})_{\varphi\in R}$, and take a neighborhood $U'$ of $a'$ satisfying the condition in \ref{lem8} (resp., \ref{lem9}) for $(U'_{\psi})_{\psi\in R'}$. 
Let $\gamma \in \Gamma_K$ and assume $\gamma U \cap U'\neq \varnothing$. We prove $\gamma a=a'$. Take $x\in U$ and $x' \in U'$ such that $\gamma x=x'$. By \ref{lem8} (resp., \ref{lem9}), there are $\varphi \in R$,  $\psi\in R'$, and $\epsilon \in\Gamma_{K,(\varphi P \varphi^{-1})}$ and $\delta \in \Gamma_{K,(\psi P\psi^{-1})}$ 
in case (i) (resp., $\epsilon \in \Gamma_{K,(\varphi W)}$ and $\delta \in \Gamma_{K,(\psi W)}$ in case (ii)) such that 
$\varphi^{-1}\epsilon^{-1}x \in \frak S\cap U_{\varphi}$ and $\psi^{-1}\delta^{-1} x' \in \frak S\cap U'_{\psi}$. Since 
$$
	(\psi^{-1}\delta^{-1}\gamma \epsilon\varphi)\varphi^{-1}\epsilon^{-1}x= \psi^{-1}\delta^{-1}x',
$$ 
we have $\psi^{-1}\delta^{-1}\gamma \epsilon\varphi \in P(F)$ by property (A). By property (B), we have
$$
	(\psi^{-1} \delta^{-1} \gamma \epsilon \varphi)\varphi^{-1} a= \psi^{-1}a'.
$$ 
Since $\epsilon a= a$ and $\delta a'=a'$, this proves $\gamma a=a'$. 

\end{sbpara}

We have proved \ref{ga=b3} in situation (II) under the assumption $S=S_1$. In the following \ref{ga=b2} and \ref{ga=b4}, we reduce the general case to that case.

\begin{sblem}\label{ga=b2}  

Let $a, a'\in Z$. In situation (I) (resp., (II)), let $H=G(\mb{A}_F^{S_1})$ (resp., $H = G(\mb{A}_{F,S_2})$). Then there exist  neighborhoods $U$ of $a$  and $U'$ of $a'$ in $Z$ such that $ga=a'$ for all $g\in H$ for which $gU\cap U'\neq \varnothing$. 

\end{sblem}

\begin{pf} 
For any compact neighborhoods $U$ of $a$ and $U'$ of $a'$, 
the set $M :=\{g\in H \mid gU\cap U'\neq \varnothing\}$ is compact. 
By definition of $Z$, there exist a compact open subgroup $N$ of $H$ and a compact neighborhood $U$ of $a$ such that $gx=x$ for all $g\in N$ and $x\in U$.  For such a choice of $U$, the set $M$ is stable under the right translation by $N$, and $M/N$ is finite because $M$ is compact and $N$ is an open subgroup of $H$. 
If $g\in M$ and if $ga\neq a'$, then by shrinking the neighborhoods $U$ and $U'$,
we have that $gU \cap U' = \varnothing$.
As $M/N$ is finite, we have sufficiently small neighborhoods $U$ and $U'$ such that 
if $g\in M$ and $gU\cap U'\neq \varnothing$, then $ga=a'$. 
\end{pf}

\begin{sbpara}\label{ga=b4}  We prove Proposition \ref{ga=b3}. 

Let $H$ be as in Lemma \ref{ga=b2}. 
 Write $a=(a_{S_1}, a_Z)$ and $a'=(a'_{S_1}, a'_Z)$ as elements of $\bar{X} \times Z$.
By \ref{ga=b2}, there exist neighborhoods $U_Z$ of $a_Z$ and $U'_Z$ of $a'_Z$ in $Z$ such that if $g\in H$  and $gU_Z\cap U'_Z \neq \varnothing$, then $g a=a'$.  The set $K' := \{g\in H \mid ga_Z=a_Z\}$ is a compact open subgroup of $H$. Let $\Gamma'$ be the inverse image of $K'$ under $\Gamma \to H$, where $\Gamma = G(F)$ in situation (I).  In situation (II), the group $\Gamma'$ is of finite index in the inverse image of the compact open subgroup $K'\times K$ under $G(F) \to G(\mb{A}_F^{S_1})$. In both situations, the set $M :=\{\gamma \in \Gamma\mid \gamma a_Z=a'_Z\}$ is either empty or a $\Gamma'$-torsor for the right action of $\Gamma'$.  

Assume first that $M \neq \varnothing$, in which case we may choose $\theta \in \Gamma$ such that $M=\theta \Gamma'$. 
Since we have proven \ref{ga=b3}  in situation (II) for $S_1=S$, there exist neighborhoods $U_{S_1}$ of $a_{S_1}$ and $U'_{S_1}$ of $\theta^{-1}a'_{S_1}$ such that if $\gamma \in \Gamma'$ 
satisfies $\gamma U_{S_1}\cap U'_{S_1}\neq \varnothing$, then $\gamma a_{S_1}=\theta^{-1}a'_{S_1}$. 
Let $U= U_{S_1} \times U_Z$ and
$U'= \theta U'_{S_1} \times U'_Z$, which are neighborhoods of $a$ and $a'$ in $\bar {\frak X}$, respectively.
Suppose that $\gamma \in \Gamma$ satisfies $\gamma U\cap U'\neq\varnothing$. Then, since $\gamma U_Z\cap U'_Z\neq \varnothing$, we have $\gamma a_Z=a'_Z$ and 
hence $\gamma = \theta \gamma'$ for some $\gamma'\in \Gamma'$. Since $\theta \gamma' U_{S_1} \cap \theta U'_{S_1}\neq \varnothing$, we have $\gamma' U_{S_1}\cap U'_{S_1}\neq \varnothing$, and hence $\gamma' a_{S_1}= \theta^{-1}a'_{S_1}$. That is, we have $\gamma a_{S_1}=a'_{S_1}$, so $\gamma a=a'$. 

In the case that $M = \varnothing$, take any neighborhoods $U_{S_1}$ of $a_{S_1}$ and $U'_{S_1}$ of $a'_{S_1}$, and set $U=U_{S_1}\times U_Z$ and $U'=U'_{S_1}\times U'_Z$. Any $\gamma \in \Gamma$ such that $\gamma U\cap U'\neq \varnothing$ is contained in $M$, so no such $\gamma$ exists.
\end{sbpara}

\subsection{Supplements to the main theorem}

We use the notation of \S\ref{mainres} throughout this subsection.
We suppose that $\Gamma=G(F)$ in situation (I), and we let $\Gamma$ be a subgroup of $\Gamma_K$ of finite index in situation (II).  For $a \in \bar{\frak X}$, let $\Gamma_a < \Gamma$ denote the stabilizer of $a$.

 \begin{sbthm}\label{thm3}  
For $a\in \bar {\frak X}$ (with either the Borel-Serre or the Satake topology), 
there is an open neighborhood $U$ of the image of $a$ in $\Gamma_a\bs \bar {\frak X}$  such that the image $U'$ of $U$ under the quotient map $\Gamma_a \bs \bar {\frak X} \to \Gamma \bs \bar {\frak X}$ is open and the map $U\to U'$ is a homeomorphism.

 \end{sbthm}
 
\begin{pf}  By the case  $a=a'$ of Proposition \ref{ga=b3}, there is an open neighborhood $U'' \subset \bar{\frak X}$ of $a$ 
such that if $\gamma\in \Gamma_K$ and $\gamma U''\cap U''\neq \varnothing$, then $\gamma a =a$. Then the subset $U:=\Gamma_a \bs \Gamma_a U''$ of $\Gamma_a \bs \bar {\frak X}$ is open and has the desired property.
\end{pf}

\begin{sbprop}\label{prop3} Suppose that $S=S_1$, and let $a \in \bar{\frak X}$.
\begin{enumerate}
\item[(1)] Take $\bar X = \bar{X}_{F,S}$, and let $P$ be the parabolic subgroup associated to $a$. Then 
$\Gamma_{(P)}$ (as in \ref{top7}) is a normal subgroup of $\Gamma_a$ of finite index.
\item[(2)] Take $\bar X = \bar{X}_{F,S}^{\flat}$, and let $W$ be the $F$-subspace of $V$ associated to $a$. Then $\Gamma_{(W)}$ (as in \ref{top7}) is a normal subgroup of $\Gamma_a$ of finite index. 
\end{enumerate}
\end{sbprop}

\begin{pf} We prove (1), the proof of (2) being similar.
Let $(V_i)_{-1\le i \le m}$ 
be the flag corresponding to $P$. The image of $\Gamma \cap P(F)$ in 
$\prod_{i=0}^m  \PGL_{V_i/V_{i-1}}(\mb{A}_{F,S})$ is discrete. 
On the other hand, the stabilizer in $\prod_{i=0}^m  \PGL_{V_i/V_{i-1}}(\mb{A}_{F,S})$ of the image of $a$ in 
$\frak{Z}_{F,S}(P)$ is compact. 
Hence the image of $\Gamma_a$ in $\prod_{i=0}^m \PGL_{V_i/V_{i-1}}(F)$, which is isomorphic to $\Gamma_a/\Gamma_{(P)}$, is finite.
\end{pf}

\begin{sbthm}\label{thm5} Assume that $F$ is a function field and $\bar X= \bar X_{F,S_1}$, where $S_1$ consists of a single place $v$. 
Then the inclusion map $\Gamma \bs \frak X \hookrightarrow \Gamma \bs \bar{\frak X}$
is a homotopy equivalence. 
\end{sbthm}

\begin{pf} 

Let $a\in \bar {\frak X}$.  In situation (I) (resp., (II)), write $a=(a_v, a^v)$ with $a_v\in \bar X_{F,v}$ 
and $a^v \in X_{S_2} \times G(\mb{A}_F^S)/K$  (resp., $X_{S_2}$).
Let $K'$ be the isotropy subgroup of $a^v$ in $G(\mb{A}^v_F)$ (resp., $\prod_{w\in S_2} G(F_w)$), and let $\Gamma' < \Gamma$ be the inverse image of $K'$ under the map $\Gamma \to G(\mb{A}^v_F)$ (resp., $\Gamma\to \prod_{w\in S_2} G(F_w)$). 

Let $P$ be the parabolic subgroup associated to $a$.  Let $\Gamma_a$ be the isotropy subgroup of $a$ in $\Gamma$,
which is contained in $P(F)$ and equal to the isotropy subgroup $\Gamma'_{a_v}$ of $a_v$ in $\Gamma'$.
In situation (I) (resp., (II)), take a $\Gamma_a$-stable open neighborhood $D$ of $a^v$ in $X_{S_2} \times G(\mb{A}_F^S)/K$ (resp., $X_{S_2}$) that has compact closure.  

\medskip

{\bf Claim 1.}
The subgroup
$\Gamma_D := \{\gamma\in \Gamma_a \mid\gamma x = x \text{ for all } x \in D\}$ 
of $\Gamma_a$ is normal of finite index.   

\begin{proof}[Proof of Claim 1] \renewcommand{\qedsymbol}{} 
	Normality follows from the $\Gamma_a$-stability of $D$.  For any $x$ in the closure $\bar{D}$ of $D$, there exists an open neighborhood $V_x$ of $x$ and a compact open subgroup $N_x$ of $G(\mb{A}^v_F)$ (resp., $\prod_{w\in S_2} G(F_w)$) in situation (I) (resp., (II)) such that $gy = y$ for all $g \in N_x$ and $y \in V_x$.  For a finite subcover $\{ V_{x_1}, \ldots, V_{x_n} \}$ of $\bar{D}$, the group $\Gamma_D$ is the inverse image in $\Gamma_a$ of $\bigcap_{i = 1}^n N_{x_i}$, so is of finite index.  
\end{proof}

{\bf Claim 2.}
The subgroup $H := \Gamma_D \cap P_u(F)$ of $\Gamma_a$ is normal of finite index.   

\begin{proof}[Proof of Claim 2] \renewcommand{\qedsymbol}{} 
Normality is immediate from Claim 1 as $P_u(F)$ is normal
in $P(F)$.    
Let $H' = \Gamma'_{(P)} \cap P_u(F)$, which has finite index in $\Gamma'_{(P)}$ and equals
$\Gamma' \cap P_u(F)$ by definition of $\Gamma'_{(P)}$.  Since $\Gamma'_{(P)} \subset \Gamma'_{a_v} \subset \Gamma'$ and $\Gamma'_{a_v} = \Gamma_a$, we have $H' = \Gamma_a \cap P_u(F)$ as well.  By Claim 1, we then have that $H'$ contains
$H$ with finite index, so $H$ has finite index in $\Gamma'_{(P)}$.  
Proposition \ref{prop3}(1) tells us that $\Gamma'_{(P)}$ is of finite index in $\Gamma'_{a_v} = \Gamma_a$.  
\end{proof}

Let $(V_i)_{-1\le i \le m}$ be the flag corresponding to $P$.
By Corollary \ref{keycor}, we have a homeomorphism 
$$
	\chi \colon P_u(F_v)\bs \bar X_{F,v}(P) \xrightarrow{\sim} \frak Z_{F,v}(P) \times \R_{\geq 0}^m
$$ 
on quotient spaces arising from the $P(F_v)$-equivariant homeomorphism $\psi_{P,v} = (\phi_{P,v},\phi'_{P,v})$ 
of \ref{togr9} (see \ref{togr} and \ref{togr2}).
\medskip

{\bf Claim 3.}
For a sufficiently small open neighborhood $U$ of $0 = (0,\dots,0)$ in $\R^m_{\geq 0}$, 
the map $\chi$ induces a homeomorphism
$$ \chi_U \colon H \bs \bar X_{F,v}(P)_U \xrightarrow{\sim} \frak Z_{F,v}(P) \times U,$$
where $\bar X_{F,v}(P)_U$ denotes the inverse image of $U$
under $\phi'_{P,v} \colon \bar X_{F,v}(P) \to \R^m_{\geq 0}$.

\begin{proof}[Proof of Claim 3] \renewcommand{\qedsymbol}{} 
By definition, $\chi$ restricts to a homeomorphism 
$$
	P_u(F_v) \bs \bar{X}_{F,v}(P)_U \xrightarrow{\sim} \frak Z_{F,v}(P) \times U
$$ 
for any open neighborhood $U$ of $0$.
For a sufficiently large compact open subset $C$ of $P_u(F_v)$, we have 
$P_u(F_v)= HC$. For $U$ sufficiently small, every $g \in C$ fixes all $x\in \bar X_{F,v}(P)_U$,
which yields the claim.
\end{proof}

{\bf Claim 4.}
The map $\chi_U$ and the identity map on $D$ induce a homeomorphism
$$
	\chi_{U,a} \colon \Gamma_a\bs (\bar X_{F,v}(P)_U\times D) \xrightarrow{\sim} (\Gamma_a \bs \frak (\frak{Z}_{F,v}(P) \times D)) \times U.
$$

\begin{proof}[Proof of Claim 4] \renewcommand{\qedsymbol}{} 
The quotient group $\Gamma_a/H$ is finite by Claim 2.  Since the determinant of an automorphism of $V_i/V_{i-1}$ of finite order has trivial absolute value at $v$, the $\Gamma_a$-action on $\R_{\geq 0}^m$ is trivial. 
Since $H$ acts trivially on $D$, the claim follows from Claim 3.
\end{proof}

Now let $c\in \R_{>0}^m$, and set 
$U= \{t\in \R^m_{\geq 0} \mid t_i<c \text{ for all } 1\leq i\leq m\}$.
Set $(X_v)_U = X_v \cap \bar{X}_{F,v}(P)_U$.
If $c$ is sufficiently small, then
$$
	(\Gamma_a \bs (\frak{Z}_{F,v}(P) \times D)) \times (U\cap \R_{>0}^m) \hookrightarrow (\Gamma_a \bs (\frak{Z}_{F,v}(P) \times D)) 
	\times U
$$ 
is a homotopy equivalence, and we can apply $\chi_{U,a}^{-1}$ to both sides to see that the inclusion map 
$$
	\Gamma_a\bs ( (X_v)_U \times D) \hookrightarrow 
\Gamma_a\bs (\bar X_{F,v}(P)_U\times D)
$$ 
is also a homotopy equivalence.  By Theorem \ref{thm3}, this proves Theorem \ref{thm5}. 
\end{pf}

\begin{sbrem} Theorem \ref{thm5} is well-viewed as a function field analogue of the homotopy equivalence for Borel-Serre spaces of \cite{BS}. 
\end{sbrem}

\begin{sbpara} Theorem \ref{main} remains true if we replace $G=\PGL_V$ by $G=\SL_V$ in \ref{mains3} and \ref{main}.
It also remains true if we replace $G=\PGL_V$ by $G=\GL_V$ and we replace $\bar {\frak X}$ in \ref{main} in situation (I) (resp., (II)) by 
$$
 	\bar X \times X_{S_2} \times (\R_{>0}^S \times G(\mb{A}_F^S)/K)_1 \quad
	\text{(resp., }\bar X \times X_{S_2} \times (\R_{>0}^S)_1),
$$ 
where $(\;\,)_1$ denotes the kernel of 
$$
	((a_v)_{v\in S}, g) \mapsto  |\text{det}(g)|\prod_{v \in S} a_v \quad
	\text{(resp., }(a_v)_{v\in S}\mapsto \prod_{v \in S} a_v),
$$ 
and $\gamma \in \GL_V(F)$ (resp., $\gamma \in \Gamma_K$) acts on this kernel by multiplication by $((|\det(\gamma)|_v)_{v\in S}, \gamma)$ (resp., 
 $(|\text{det}(\gamma)|_v)_{v\in S}$). 

Theorems \ref{thm3} and \ref{thm5} also remain true under these modifications. 
These modified versions of the results are easily reduced to the original case $G=\PGL_V$. 
\end{sbpara}

\subsection{Subjects related to this paper} \label{toro}

\begin{sbpara} In this subsection, 
as possibilities of future applications of this paper, we describe connections with the study of
\begin{itemize}
	\item toroidal compactifications of moduli spaces of Drinfeld modules (\ref{pk}--\ref{cone})
	\item the asymptotic behavior of Hodge structures and $p$-adic Hodge structures associated to a 
	degenerating family of motives over a number field  (\ref{sbj1}, \ref{sbj2}), and
	\item modular symbols over function fields (\ref{sbj3}, \ref{sbj4}).
\end{itemize}
 
\end{sbpara}

\begin{sbpara} \label{pk}
In \cite{P2}, Pink constructed a compactification of the moduli space of Drinfeld modules that is similar to the Satake-Baily-Borel compactification of the moduli space of polarized abalian varieties. 
In a special case, it had been previously constructed by Kapranov \cite{K}.

In \cite{P1},
Pink, sketched a method for constructing a compactification of the moduli space of Drinfeld modules that is 
similar to the toroidal compactification of the moduli space of polarized abelian varieties (in part, using ideas of K.~Fujiwara). However, the details of the construction have not been published.  
Our plan for constructing toroidal compactifications seems to be different from that of \cite{P1}.

\end{sbpara}

\begin{sbpara}\label{descr} 
We give a rough outline of the relationship that we envision between this paper and the analytic theory of toroidal compactifications.   Suppose that $F$ is a function field, and fix a place $v$ of $F$. Let $O$ be the ring of all elements of $F$ which are integral outside $v$.  In \cite{D}, the notion of a Drinfeld $O$-module of rank $d$ is defined,
and the moduli space of such Drinfeld modules is constructed.

Let $\C_v$ be the completion of an algebraic closure of $F_v$ and let $|\;\;| \colon \C_v \to \R_{\geq 0}$ be the absolute value which extends the normalized absolute value of $F_v$. 
Let $\Omega\subset \mb{P}^{d-1}(\C_v)$ be the $(d-1)$-dimensional Drinfeld upper half-space consisting of all points $(z_1:\dots :z_d)\in \mb{P}^{d-1}(\C_v)$ such that $(z_i)_{1\leq i\leq d}$ is linearly independent over $F_v$. 

For a compact open subgroup $K$ of $\GL_d(\mb{A}_F^v)$,  the set of $\C_v$-points of the moduli space  $M_K$ of Drinfeld $O$-modules of rank $d$ with $K$-level structure is expressed as $$M_K(\C_v)=\GL_d(F) \bs (\Omega\times \GL_d(\mb{A}^v_F)/K)$$ (see \cite{D}).

Consider the case $V=F^d$ in \S\ref{globspace} and \S\ref{quotS}. 
We have a map $\Omega \to X_v$ which sends $(z_1:\dots : z_d)\in \Omega$ to the class of the the norm 
$V_v=F^d_v\to \R_{\geq 0}$ given by $(a_1,\dots, a_d) \mapsto |\sum_{i=1}^d a_iz_i|$ for $a_i\in F_v$. 
This map  induces a 
 canonical continuous map 
\begin{enumerate}
\item[(1)]  $M_K(\C_v) = \GL_d(F) \bs (\Omega\times \GL_d(\mb{A}^v_F)/K)\to \GL_d(F)\bs (X_v\times \GL_d(\mb{A}^v_F)/K).$
\end{enumerate}
The map (1) extends to a canonical continuous map
\begin{enumerate}
\item[(2)] $\bar M_K^{\KP}(\C_v) \to \GL_d(F)\bs (\bar X_{F,v}^{\flat}\times \GL_d(\mb{A}^v_F)/K)$,
\end{enumerate}
where $\bar M_K^{\KP}$ denotes the compactification of Kapranov and Pink of $M_K$.  In particular, $\bar M_K^{\KP}$ is related to $\bar X_{F,v}^{\flat}$. On the other hand, the toroidal compactifications of $M_K$ should be related to $\bar X_{F,v}$. If we denote by $\bar M_{K}^{\tor}$ the projective limit of all toroidal compactifications of $M_K$, then the map of (1) should extend to a canonical continuous map
\begin{enumerate}
\item[(3)] $\bar M_{K}^{\tor}(\C_v) \; \to \GL_d(F)\bs (\bar X_{F,v}\times \GL_d(\mb{A}^v_F)/K)$.
\end{enumerate}
that fits in a commutative diagram
%$$ \SelectTips{cm}{} \xymatrix@R=12pt{ \bar M_{K}^{\tor}(\C_v) \ar[r] \ar[d] & \GL_d(F)\bs (\bar X_{F,v}\times \GL(\mb{A}^v_F)/K)
%\ar[d]\\
%\bar M_K^{\KP}(\C_v) \ar[r] & \GL_d(F) \bs (\bar X_{F,v}^{\flat}\times \GL_d(\mb{A}^v_F)/K).}
%$$
$$ \begin{tikzcd}[row sep = small] 
	\bar M_{K}^{\tor}(\C_v) \arrow{r} \arrow{d} & 
	\GL_d(F)\bs (\bar X_{F,v}\times \GL(\mb{A}^v_F)/K) \arrow{d} \\
	M_K^{\KP}(\C_v) \arrow{r} & \GL_d(F) \bs (\bar X_{F,v}^{\flat}\times \GL_d(\mb{A}^v_F)/K).
\end{tikzcd}
$$

\end{sbpara}

\begin{sbpara} The expected map of \ref{descr}(3) is the analogue of the canonical continuous map from the projective limit of all toroidal compactifications of the moduli space of polarized abelian varieties to the reductive Borel-Serre compactification (see \cite{GT, KU}).

From the point of view of our study, the reductive Borel-Serre compactification and $\bar X_{F,v}$ are enlargements of spaces of norms.  A polarized abelian variety $A$ gives a norm on the polarized Hodge structure associated to $A$ (the Hodge metric).  This relationship between a polarized abelian variety and a norm forms the foundation of the relationship between the toroidal compactifications of a moduli space of polarized abelian varieties and the reductive Borel-Serre compactification. This is similar to the relationship between $M_K$ and the space of norms $X_v$ given by the map of \ref{descr}(1), as well as the relationship between  $\bar M_{K}^{\tor}$ and $\bar X_{F,v}$ given by \ref{descr}(3). 

\end{sbpara}

\begin{sbpara} \label{cone} In the usual theory of toroidal compactifications, cone decompositions play an important role. In the toroidal compactifications of \ref{descr}, the simplices of Bruhat-Tits buildings (more precisely, the simplices contained in the fibers of $\bar X_{F,v}\to \bar X_{F,v}^{\flat}$) should play the role of the cones in cone decompositions.

\end{sbpara}

\begin{sbpara}\label{sbj1}
We are preparing a paper in which our space $\bar X_{F,S}$ with $F$ a number field and with $S$
containing a non-archimedean place appears in the following way.

Let $F$ be a number field, and let $Y$ be a polarized projective smooth
variety over $F$. Let $m\geq 0$, and let $V =H^m_{\dR}(Y)$ be the de Rham cohomology. 
For a place $v$ of $F$, let $V_v = F_v \otimes_F V$.

For an archimedean place $v$ of $F$, it is well known that $V_v$ has a Hodge metric.  
For $v$ non-archimedean, we can under certain assumptions define
a Hodge metric on $V_v$  by the method illustrated in the example of \ref{sbj2} below. 
The $[F_v:\Q_v]$-powers of these Hodge metrics for $v\in S$ are norms and therefore
provide an element of $\prod_{v\in S} X_{V_v}$. When $Y$ degenerates,
this element of $\prod_{v\in S} X_{V_v}$ can converge to a boundary
point of $\bar X_{F,S}$.

\end{sbpara}

\begin{sbpara}\label{sbj2}  Let $Y$ be an elliptic
curve over a number field $F$, and take $m=1$.

Let $v$ be a non-archimedean place of $F$,
and assume that $F_v \otimes_F Y$ is a Tate elliptic curve of
$q$-invariant $q_v \in F_v^{\times}$ with $|q_v| < 1$. 
Then the first
log-crystalline cohomology group $D$ of the special
fiber  of this elliptic curve 
is a free module of rank $2$ over the
Witt vectors $W(k_v)$ with
a basis $(e_1, e_2)$ on which the Frobenius $\varphi$ acts as
$\varphi(e_1)=e_1$ and $\varphi(e_2)=pe_2$, where $p = \mathrm{char}\,k_v$. The first $\ell$-adic \'etale cohomology group of this elliptic curve is 
a free module of rank $2$ over $\Z_{\ell}$ with a basis $(e_{1,\ell}, e_{2,\ell})$ such that
the
inertia subgroup of $\Gal(\bar F_v/F_v)$ fixes $e_1$. 
The monodromy
operator $N$ satisfies 
$$
Ne_2=\xi'_ve_1, \quad  Ne_1=0, \quad Ne_{2,\ell}=\xi'_ve_{1,\ell}, \quad Ne_{1,\ell}=0
$$
where
$\xi'_v =\text{ord}_{F_v}(q_v) > 0$.  The standard polarization $\langle\;\,,\;\rangle$ of the elliptic curve 
satisfies $\langle e_1, e_2\rangle=1$ and hence 
$$
\langle Ne_2, e_2\rangle =\xi'_v,\quad \langle e_1, N^{-1}e_1\rangle = 1/\xi'_v\quad 
\langle Ne_{2,\ell}, e_{2,\ell}\rangle =\xi'_v,\quad \langle e_{1,\ell}, N^{-1}e_{1,\ell}\rangle = 1/\xi'_v.
$$
For $V = H_{\dR}^1(Y)$, 
we have an isomorphism $V_v \cong F_v
\otimes_{W(k_v)} D$.
The Hodge metric $|\;\; |_v$ on $V_v$ is  defined by
$$
	|a_1e_1+a_2e_2|_v= \max(\xi_v^{-1/2}|a_1|_p, \xi_v^{1/2}|a_2|_p)
$$
for $a_1,a_2\in F_v$, where $|\;\;|_p$ denotes the absolute value on $F_v$ satisfying $|p|_p=p^{-1}$
and 
$$
\xi_v: =- \xi_v' \log(|\varpi_v|_p)= -\log(|q_v|_p)>0,
$$
where $\varpi_v$ is a prime element of $F_v$. 
That is, to define the Hodge metric on $V_v$, we  modify the naive
metric (coming from the integral structure of the log-crystalline
cohomology) by using $\xi_v$ which is determined by the polarization $\langle\;\,,\;\rangle$, the monodromy operator $N$, and the integral structures of the log-crystalline and $\ell$-adic cohomology groups (for $\ell\neq p$). 

This is similar to what happens at an archimedean place $v$.  We have
$Y(\C) \cong \C^\times/q_v^{\Z}$ with $q_v \in  F_v^{\times}$.
Assume for simplicity that we can take $|q_v|<e^{-2\pi}$ where
$|\;\;|$ denotes the usual absolute value. Then $q_v$ is determined by
$F_v \otimes_F Y$ uniquely. Let $\xi:=-\log(|q_v|)>2\pi$.
If $v$ is real, we further assume that $q_v>0$ and that we have an
isomorphism $Y(F_v)\cong F_v^\times/q_v^{\Z}$ which is compatible with
$Y(\C)\cong \C^\times/q_v^{\Z}$.
Then in the case $v$ is real (resp., complex), there is a basis
$(e_1,e_2)$ of $V_v$ such that $(e_1,(2\pi i)^{-1}e_2)$ is a
$\Z$-basis of $H^1(Y(\C), \Z)$ and such that
the Hodge metric $|\;\;|_v$ on $V_v$ satisfies  $|e_1|_v= \xi_v^{-1/2}$
and $|e_2|_v=\xi_v^{1/2}$ (resp., $||e_2|_v-\xi_v^{1/2}|\leq \pi
\xi_v^{-1/2}$).

Consider for example the elliptic curves $y^2=x(x-1)(x-t)$ with $t\in
F=\Q, t\neq 0,1$.  As $t$
approaches $1 \in \Q_v$ for all $v\in S$, the elliptic curves $F_v\otimes_F Y$ satisfy
the above assumptions for all $v\in S$, and  each $q_v$ approaches
$0$, so
$\xi_v$ tends to $\infty$. The corresponding elements of $\prod_{v\in
S} X_{V_v}$ defined by the classes
of the $|\;\;|_v$ for $v \in S$ converge to the unique boundary point
of $\bar X_{F,S}$ with associated
parabolic equal to the Borel subgroup of upper triangular matrices in
$\text{PGL}_V$ for the basis
$(e_1, e_2)$.

We hope that this subject about $\bar X_{F,S}$ is an interesting
direction to be studied.  It may be related to  the asymptotic
behaviors of heights of motives in degeneration studied in \cite{Ko}.
\end{sbpara}

\begin{sbpara}\label{sbj3} Suppose that $F$ is a function field and let $v$ be a place of $F$. 
Let $\Gamma$ be as in \ref{thm0}. 

Kondo and Yasuda \cite{KY} proved that the image of $H_{d-1}(\Gamma \bs X_v,\Q) \to H_{d-1}^{\BM}(\Gamma \bs X_v,\Q)$ is generated by modular symbols, where $H_*^{\BM}$ denotes Borel-Moore homology.
Our hope is that the compactification $\Gamma\bs \bar X_{F,v}$ of $\Gamma \bs X_v$ is useful in further studies of modular symbols. 

Let $\partial:= \bar X_{F,v} \setminus X_v$. Then we have an isomorphism $$H_*^{\BM}(\Gamma\bs X_v,\Q)\cong H_*(\Gamma \bs \bar X_{F,v}, \Gamma \bs \partial, \Q)$$ and  an exact sequence 
$$\dots \to H_i(\Gamma\bs \bar X_{F,v}, \Q)\to H_i(\Gamma\bs \bar X_{F,v}, \Gamma\bs \partial, \Q) \to H_{i-1}(\Gamma \bs \partial,\Q) \to H_{i-1}(\Gamma\bs \bar X_{F,v}, \Q)\to \dots.$$

Since $\Gamma \bs X_v \to \Gamma \bs \bar X_{F,v}$ is a homotopy equivalence by Thm. \ref{thm5}, we have
$$H_*(\Gamma \bs X_v, \Q) \xrightarrow{\sim} H_*(\Gamma\bs \bar X_{F,v}, \Q).$$
Hence the result of Kondo and Yasuda shows that the kernel of 
$$H_{d-1}^{\BM}(\Gamma\bs X_v,\Q)\cong H_{d-1}(\Gamma \bs \bar X_{F,v}, \Gamma \bs \partial, \Q)\to H_{d-2}(\Gamma \bs \partial,\Q)$$
is generated by modular symbols. 

If we want to prove that $H_{d-1}^{\BM}(\Gamma\bs X_v,\Q)$ is generated by modular symbols, it is now sufficient to prove that the kernel of $H_{d-2}(\Gamma \bs \partial,\Q) \to H_{d-2}(\Gamma\bs \bar X_{F,v}, \Q)$ is generated by the images (i.e., boundaries) of modular symbols.
\end{sbpara}

\begin{sbpara}\label{sbj4}
In \ref{sbj3}, assume $d=2$.
Then we can prove that $H_1^{\BM}(\Gamma \bs X_v,\Q)$ is generated by modular symbols. 
In this case, the map $H_0(\Gamma \bs \partial, \Q)=\text{Map}(\Gamma \bs \partial, \Q)  \to H_0(\bar X_{F,v}, \Q)=\Q$ is just summation, and it is clear that the kernel of it is generated by the boundaries of modular symbols. 
\end{sbpara}

\end{document}